\documentclass[a4paper,12pt]{article}
\usepackage{amsfonts}
\usepackage{pst-all} 
\usepackage{amssymb}
\usepackage{latexsym}
\usepackage{amsmath}
\usepackage{amscd}
\usepackage{amsthm}
\usepackage{graphicx}
\usepackage{epsfig}
\usepackage{hyperref}
\usepackage{url}
\usepackage{authblk}

\newtheorem{theorem}{Theorem}[section]
\newtheorem{definition}[theorem]{Definition}

\newtheorem{proposition}[theorem]{Proposition}
\newtheorem{lemma}[theorem]{Lemma}
\newtheorem{remark}[theorem]{Remark}
\newtheorem{corollary}[theorem]{Corollary}

\newtheorem{prop}[theorem]{Proposition}
\newtheorem{corol}[theorem]{Corollary}

\newcommand{\NN}{{\mathcal{N}}}

\newcommand{\E}{{\mathbb E}}
\newcommand{\N}{{\mathbb N}}
\newcommand{\Z}{{\mathbb Z}}
\renewcommand{\P}{{\mathbb P}}

\newcommand{\R}{{\mathbb R}}

\newcommand{\x}{{\mathbf x}}
\newcommand{\y}{{\mathbf y}}
\renewcommand{\u}{{\mathbf u}}
\renewcommand{\v}{{\mathbf v}}
\newcommand{\w}{{\mathbf w}}

\newcommand{\WW}{{\mathcal W}}

\usepackage[utf8]{inputenc}
\usepackage[normalem]{ulem}
\usepackage{color}

\definecolor{lightblue}{rgb}{0.9,0.9,1}

\title{Scaling limit of a drainage network model on perturbed lattice }

\author[1]{Rahul Roy}
\author[2]{Kumarjit Saha}
\author[1]{Anish Sarkar}
\affil[1]{Indian Statistical Institute, Delhi}
\affil[2]{Ashoka University, Sonipat}

\begin{document}

\maketitle  

\begin{abstract}
Study of random networks generally requires
the nodes to be independently and uniformly distributed such as a Poisson point process. In this
work, we venture beyond this standard paradigm and investigate a stochastic 
forest obtained from a drainage network model constructed on a randomly perturbed
subset of $\Z^2$, where both horizontal and vertical  
perturbations are given by exponentially decaying unbounded discrete random variables 
and vertical perturbations are allowed in the upward direction only. 
We show that the resultant stochastic
network is a single tree a.s. We further establish that as a collection of paths, 
under diffusive scaling the resultant network converges to the Brownian web.
\end{abstract}

\section{Introduction and main results}

Directed random networks, where edges have a preferred direction of propagation,
 have long been an important class of models for understanding the large
scale behaviour of systems in a wide array of applications. These include, but are not limited
to, transport networks, power grids, various kinds of social networks, different types
of communication networks including wireless sensor networks, multicast communication
networks, peer-to-peer networks and drainage networks. 

It has been empirically observed that river networks satisfy various scaling laws
 and studying drainage networks is a 
statistical approach to understand the unifying behaviour of river basins. 
Leopold and Langbein \cite{LL62} first carried out studies of drainage basins by simulating
drainage networks through random walks in a rectangular region. 
Scheidegger \cite{S67} was the first to introduce a directed network by imposing a preferential
flow condition where each source emptying to one of it’s two neighbours in a
preferred direction. Howard \cite{H71} removed the restriction of drainage to 
neighbouring sites only. Iturbe and Rinaldo \cite{RR97} presented an excellent survey
of development in this field.

Informally drainage network models can be described in the following way.
A random subset of $\Z^d$ or $\R^d$ is selected as a collection of ‘source’ vertices.
From each source $\x = (\x(1), \cdots , \x(d)) \in \R^d$ exactly one directed edge is
drawn to another source vertex $\y$ with $\y(d) > \x(d)$, representing the flow 
from $\x$ to $\y$. We briefly describe the Howard’s model here. 
Each vertex in $\Z^d$ is a source vertex independently with probability
$p \in (0, 1)$ and a source vertex $u\in \Z^d$ connects to the nearest source at the next level
$\{ \w \in \Z^d: \w(d) = u(d) + 1 \}$. If there are more than one closest source, one of them
is chosen uniformly. For this model Gangopadhyay et.al. [GRS04] showed that
the random graph is a (connected) tree a.s. for $d = 2, 3$ and it is a forest with infinitely
many disjoint tree components for $d \geq 4$. Since then various other drainage network 
models with complex dependencies in constructing edges have been proposed
and the tree-forest dichotomy problem depending on dimensions has been studied for these models
(see \cite{FLT04}, [ARS08], [CT13],\cite{RSS16A}).

In the mathematical study of such networks, 
an important modelling hypothesis is the random distribution of their
nodes. Generally speaking, the distribution of the random nodes in stochastic networks is 
often taken to be uniform over space and independently distributed over disjoint regions, 
like, the Poisson point process and its variants. Such distribution of nodes is comparatively 
easily amenable to rigorous mathematical treatment, 
but is often limited in its effectiveness to model the reality, e.g., on a global scale the
homogeneous Poisson process exhibits clusters of points interspersed with vacant spaces,
whereas a more spatially uniform distribution might be a closer representation of the
reality (see, e.g., \cite{GL17}). However, little is understood about stochastic geometry of
networks arising from such strongly correlated point processes, principally because the tools
and techniques for studying the Poisson model heavily rely on its exact spatial independence
property.

This motivates the authors of \cite{GS20} to study a drainage network model 
constructed on a point process obtained from perturbations of lattice points where 
perturbations are restricted to compact domains. The question of tree-forest dichotomy depending 
on dimensions has been explored in \cite{GS20} for the resultant network and 
it's scaling limit has also been studied. This work has initiated the analysis of directed 
networks constructed  on disordered lattice points, though working with 
perturbations restricted to compact domain is certainly a limitation of \cite{GS20}. 
In this paper, we remove this limitation and examine a set of disordered lattice 
points for $d=2$ generated due to `unbounded' i.i.d. perturbations  along $x$-axis and $y$ axis
with exponentially decaying tails. The generated set of perturbed lattice points 
exhibit much greater measure of spatial homogeneity 
compared to the Poisson process. It is important to observe that the generated 
point process is hyperuniform.
Hyperuniformity of point processes have attracted a lot of interest in recent years, especially
in the statistical physics literature (see, e.g., \cite{GL17}, 
\cite{T02}, \cite{TS03} and the references therein). A point
process is said to be hyperuniform if the variance of the number of points in an expanding
domain scales like its surface area (or slower), rather than its volume, which is the case for
Poisson or any other extensive system that exhibits FKG-type properties. In fact, hyperuniformity
is closely related to negative association at the spatial level, which precludes the
application of many arguments that are ordinarily staple in stochastic geometry. 
In the subsequent paragraphs, we lay out the details of the model and give an account of our principal
results.

We are going to define our model now.  
In what follows, for $\x \in \R^2$, the notation $\x(i)$
 denotes the $i$-th coordinate of $\x$ for $1 \leq i \leq 2$. 
 We consider an i.i.d. family of random vectors 
 \begin{align}
\label{eq:PerturbRV}
\{ \Gamma_\u := (B_\u, R_\u, \Lambda_\u) : \u \in \Z^2 \}
\end{align}
such that the following holds:
\begin{itemize}
\item[(i)] $B_\u$ is a Bernoulli r.v. with success probability $p \in (0,1)$ 
which indicates whether a lattice point $\u$ is an open (source) vertex or not;
\item[(ii)] $R_\u$ is a Rademacher r.v. associated to $\u$ which helps to resolve a tie 
(in case there is one) for deciding the Howard step;
\item[(iii)] $\Lambda_\u := (X_\u, Y_\u)$ denotes the perturbation random vector 
associated to $\u$ where $X_\u$ and $Y_\u$ denote the $x$ coordinate perturbation  
and $y$ coordinate perturbation r.v.'s respectively. Set $\theta_x, \theta_y \in (0,1)$
and the respective p.m.f.'s for $X_\u$ and $Y_\u$ are given by: 
\end{itemize}   
\begin{align}
\label{eq:ExpTailPerturbRV}
P(X_\u = j) & = \theta_x  ( (1-\theta_x)/2  )^{|j|} 
\text{ for }j \in \Z  \text{ and }\nonumber\\
P(Y_\u = j) & = \theta_y  (  1-\theta_y  )^{j} 
\text{ for } j \geq 0.
\end{align}
Let $\{\u \in \Z^2 : B_\u = 1\}$ denote the collection of open (source) vertices. Using 
(random) perturbation vector $\Lambda_\w = (X_\w, Y_\w)$ a lattice point $\w \in \Z^2$ 
gets perturbed to a new location $\w + \Lambda_\w$ and this perturbed version
 is denoted by $\tilde{\w}$.
Note that two vertices $\u$ and $\w$ may perturb to the same 
location (i.e., $\tilde{\u } = \tilde{\w} $) and in that case,
 we don't distinguish between them. We consider the set of 
 \textit{perturbed} open vertices denoted by $V$ and defined as  
\begin{align}
\label{def:Vertex_Set_V}
V := \{ \tilde{\u} = \u + \Lambda_\u : \u \in \Z^2, B_\u = 1 \}.
\end{align}
Based on the point process $V$ we construct the Howard's network as follows.
For $\u \in \Z^2$ we define the non-negative integer valued random variable $J(\u)$ as
\begin{align}
\label{def:NearestPerturbedOpen}
J(\u) := \min \{ |\w(1) - \u(1)| : \w \in V, \w(2) = \u(2) + 1\}. 
\end{align} 
In other words, $J(\u)$ denotes the distance of the {\it nearest} 
point in $V$ from $\u$ having $y$ coordinate as $\u(2) + 1$.

Starting from $\u \in \Z^2$, based on the set $V$ we define
the `perturbed Howard' (PH) step $h(\u)=h(\u, V)$ as the almost 
surely (a.s.) unique point in $V$ with $h(\u)(2) = \u(2) + 1$ such that:
\begin{align}
h(\u, V) = h(\u) := 
\begin{cases}
\u + (J(\u), 1) & \text{ if }\u + (J(\u), 1) \in V \text{ and } \u + (-J(\u), 1) \notin V  \\
\u + (- J(\u), 1) & \text{ if }\u + (- J(\u), 1) \in V\text{ and } \u + (J(\u), 1) \notin V  \\
\u + (R_\u J(\u), 1) & \text{ if }\u + (- J(\u), 1), \u +  (J(\u), 1) \in V.
\end{cases}
\end{align}
 We drop the point set $V$ from the notation $h(\u, V)$ when it is clear from the 
context and denote it simply as $h(\u)$.
We consider the random graph $G := (V, E)$ with vertex set $V$ 
and edge set $E := \{\langle \u , h(\u)\rangle : \u \in V\}$ and call
 it the perturbed Howard (PH) model. 
We observe that each $\u \in V$ has exactly one outgoing  edge and therefore, 
the generated random graph does not have any cycle or loop a.s.

In what follows, we assume that the i.i.d. collections $\{B_\u : \u \in \Z^2\}$, 
$\{R_\u : \u \in \Z^2\}$ and $\{\Lambda_\u : \u \in \Z^2\}$ are 
independent of each other. We do not require independence of $x$ coordinate 
and $y$ coordinate perturbation random variables. Rather, we assume that the joint 
distribution of the perturbation random vector $\Lambda_\u = (X_\u, Y_\u)$ 
is such that we have $\P(X_\u = Y_\u = 0 ) > 0$ and $X_\u, Y_\u$
respectively follow marginal distributions as specified in (\ref{eq:ExpTailPerturbRV}).
Based on these assumptions our first result shows that the random graph
$G$ is connected a.s.
\begin{theorem}
\label{thm:PerturbedHoward_Tree}
The perturbed Howard network $G = (V, E)$ is connected and consists of a single tree a.s.
Further, there is no bi-infinite path in $G$ a.s.
\end{theorem}

Our next main result is that the graph $G$ observed as a collection of paths,
converges to the Brownian web under a suitable diffusive scaling. The standard Brownian
web is originated in the work of Arratia \cite{A79} as the scaling limit of 
the voter model on $\Z$. Intuitively,
the Brownian web can be thought of as a collection of one-dimensional coalescing Brownian
motions starting from every point in the space time plane $\R^2$. 
Later \cite{FINR04} provided
a framework in which the Brownian web is realized as a random variable taking values in a
Polish space. In Subsection \ref{subsec:BwebIntro} 
we present the relevant topological details from \cite{FINR04}.

Set $h^0(\u) = \u$ and for $k \geq 1$, let $h^k(\u) = h(h^{k-1}(\u))$ denote 
the $k$-th step starting from $\u$.  
Joining successive steps $\langle h^{k-1}(\u), h^k(\u) \rangle$
for all $k \geq 1$ linearly, we get the PH path 
$\pi^\u$ starting from $\u$ constructed using the point set $V$.
Sometimes we call this path as the process $\{ h^k(\u) : k \geq 0\}$ of successive 
steps also. We consider this two-dimensional PH network as a 
collection of paths and want to study it's scaling limit under diffusive scaling.
We need some notations.

Let ${\cal X} := \{\pi^{\u}:\u \in V\}$ denote the collection of all PH paths.
For given  $\gamma , \sigma > 0$ and for any $n \in \N$, the scaled path
 is given by 
\begin{align}
\label{eq:PathScale}
\pi_n(\gamma,\sigma) (t) := \pi(n \gamma t)/(\sqrt{n} \sigma) \text{ and }
{\cal X}_n(\gamma,\sigma) := \{\pi_n^{\u}(\gamma,\sigma):\u \in V\}
\end{align} 
denotes the collection of the scaled paths.
Let $\bar{{\cal X}}_n(\gamma,\sigma)$ denote the closure of ${\cal X}_n(\gamma,\sigma)$
w.r.t. certain metric which is explained in detail in SubSection \ref{subsec:BwebIntro}. 
Now we are ready to state our second theorem regarding the convergence of the diffusively 
scaled PH network to the Brownian web.
\begin{theorem}
\label{thm:PerturbedHoward_Bweb}
There exist $\sigma  = \sigma(p, \theta_x, \theta_y) > 0$ and $\gamma = \gamma(p, \theta_x, \theta_y) > 0$
such that $\bar{{\cal X}}_n(\gamma,\sigma)$ converges in distribution to the 
Brownian web ${\cal W}$ as $n \to \infty$.
\end{theorem}

We should mention here that though we have assumed particular distributions for $x$ coordinate 
and $y$ coordinate perturbation random variables, our arguments hold in much generality as mentioned in 
the following remark.
\begin{remark}
\label{rem:GenExpTailAssumptions} 
All our arguments hold and we have Theorem \ref{thm:PerturbedHoward_Tree} and Theorem 
\ref{thm:PerturbedHoward_Bweb} as long as the i.i.d. collection of 
perturbation random vectors $\{\Lambda_\u = (X_\u, Y_\u) : \u \in \Z^2\}$ 
satisfies the following assumptions:
\begin{itemize}
\item[(i)] $\P(X_\u = Y_\u = 0 ) > 0$ and  $\P(Y_\u < 0 ) = 0$; 
\item[(ii)] $\P(X_\u = j ) = \P(X_\u = -j )$ for all $j \in \N$;
\item[(iii)] There exist $C_0, C_1 > 0$ such that 
$$
\P(|X_\u| \geq  n)\vee \P(Y_\u \geq  n)\leq C_0 \exp{(-C_1 n)} \text{ for all }n.
$$
\end{itemize}
In particular, if we consider an i.i.d. collection of 
Gaussian random vectors $\{\Psi_\u = (\Psi_\u(1), \Psi_\u(2))   : \u \in \Z^2\}$ and 
take the perturbation random vector $\Lambda_\u = (X_\u, Y_\u)$ defined as  
$$
X_\u := \lfloor \Psi_\u(1) \rfloor \text{ and }Y_\u := \lfloor | \Psi_\u(2) | \rfloor ,
$$
still Theorem \ref{thm:PerturbedHoward_Tree} and Theorem 
\ref{thm:PerturbedHoward_Bweb} hold for the Howard's network 
constructed on the point process of perturbed lattice points.

We should mention here that the point process of perturbed lattice points under
Gaussian perturbations have attracted a lot of interests in recent years. 
In particular Holroyd and Soo \cite{HS13} showed that in two dimensions, 
the resulting Gaussian perturbed lattice point process is `rigid' in the 
sense that for this strongly correlated point process, for any bounded domain ${\cal D}$ 
point process configuration on ${\cal D}^c$ almost surely determines the number of points 
inside ${\cal D}$. In three dimensions,  Peres and Sly \cite{PS14}  examined rigidity properties of this point process in greater detail and establish a phase transition.     
\end{remark}

\subsection{The Brownian web}
\label{subsec:BwebIntro}

 Fontes \textit{et. al.} \cite{FINR04} provided a suitable framework so that the Brownian 
web (BW) can be regarded as a random variable taking values in a Polish space.
In this section, we recall the relevant topological details from \cite{FINR04}.

Let $\R^{2}_c$ denote the completion of the space time plane $\R^2$ with
respect to the metric
\begin{equation*}
\rho((x_1,t_1),(x_2,t_2))=|\tanh(t_1)-\tanh(t_2)|\vee \Bigl| \frac{\tanh(x_1)}{1+|t_1|}
-\frac{\tanh(x_2)}{1+|t_2|} \Bigr|.
\end{equation*}
As a topological space $\R^{2}_c$ can be identified with the
continuous image of $[-\infty,\infty]^2$ under a map that identifies the line
$[-\infty,\infty]\times\{\infty\}$ with the point $(\ast,\infty)$, and the line
$[-\infty,\infty]\times\{-\infty\}$ with the point $(\ast,-\infty)$.
A path $\pi$ in $\R^{2}_c$ with starting time $\sigma_{\pi}\in [-\infty,\infty]$
is a mapping $\pi :[\sigma_{\pi},\infty]\rightarrow [-\infty,\infty]$ such that
$\pi(\infty)= \pi(-\infty)= \ast$ and $t\rightarrow (\pi(t),t)$ is a continuous
map from $[\sigma_{\pi},\infty]$ to $(\R^{2}_c,\rho)$.
We then define $\Pi$ to be the space of all paths in $\R^{2}_c$ with all possible starting times in $[-\infty,\infty]$ equipped with the following metric, 
\begin{equation*}
d_{\Pi} (\pi_1,\pi_2)= |\tanh(\sigma_{\pi_1})-\tanh(\sigma_{\pi_2})|\vee\sup_{t\geq
\sigma_{\pi_1}\wedge
\sigma_{\pi_2}} \Bigl|\frac{\tanh(\pi_1(t\vee\sigma_{\pi_1}))}{1+|t|}-\frac{
\tanh(\pi_2(t\vee\sigma_{\pi_2}))}{1+|t|}\Bigr|,
\end{equation*}
for $\pi_1,\pi_2\in \Pi$. This metric 
makes $\Pi$ a complete, separable metric space. Convergence in this
metric can be described as locally uniform convergence of paths as
well as convergence of starting times.
Let ${\cal H}$ be the space of compact subsets of $(\Pi,d_{\Pi})$ equipped with
the Hausdorff metric $d_{{\cal H}}$ given by,
\begin{equation*}
d_{{\cal H}}(K_1,K_2)= \sup_{\pi_1 \in K_1} \inf_{\pi_2 \in
K_2}d_{ \Pi} (\pi_1,\pi_2)\vee
\sup_{\pi_2 \in K_2} \inf_{\pi_1 \in K_1} d_{\Pi} (\pi_1,\pi_2).
\end{equation*}
The space $({\cal H},d_{{\cal H}})$ is a complete separable metric space. Let
$B_{{\cal H}}$ be the Borel  $\sigma-$algebra on the metric space $({\cal H},d_{{\cal H}})$.
The Brownian web ${\mathcal W}$ is then defined and characterized as an 
$({\cal H}, B_{{\cal H}})$ valued random variable by the following result:

\begin{theorem}[Theorem 2.1 of \cite{FINR04}]
\label{theorem:Bwebcharacterisation}
There exists an $({\mathcal H}, {\mathcal B}_{{\mathcal H}})$-valued random variable
${\mathcal W}$ whose distribution is uniquely determined by
the following properties:
\begin{itemize}
\item[$(a)$] from any deterministic point $\x\in\R^2$, there is  almost surely a unique path $\pi^{\x}\in {\mathcal W}$  starting from $\x$;
\item[$(b)$] for a finite set of deterministic points $\x_1,\dotsc, \x_k \in \R^2$, the collection $(\pi^{\x_1},\dotsc,\pi^{\x_k})$ is distributed as coalescing 
Brownian motions starting from $\x_1,\dotsc,\x_k$;
\item[$(c)$] for any countable deterministic dense set ${\mathcal D}$ of $\R^2$, ${\mathcal W}$ is the closure of $\{\pi^{\x}: \x\in {\mathcal D} \}$ in $(\Pi, d_{\Pi})$  almost surely.
\end{itemize}
\end{theorem}
The above theorem shows that the collection is almost surely determined by countably many coalescing Brownian motions.

We present a short introduction on the dual Brownian web $\widehat{\WW}$. 
As in case of forward paths, one can consider a similar metric space of collection of backward paths 
denoted by $(\widehat{\Pi}, d_{\widehat{\Pi}})$. 
The notation $(\widehat{{\mathcal H}}, d_{\widehat{{\mathcal H}}})$
denotes the corresponding Polish space of compact collections of backward paths with the induced 
Hausdorff metric. The Brownian web and its dual denoted 
by $({\mathcal W},\widehat{{\mathcal W}})$ is a $({\mathcal H}\times \widehat{{\mathcal H}}, {\mathcal B}_{{\mathcal H}}\times {\mathcal B}_{\widehat{{\mathcal H}}})$-valued random variable such that:
\begin{itemize}
\item[$(i)$] $\widehat{{\mathcal W}}$ is distributed as the Brownian web
 rotated $180^0$ about the origin;
\item[$(ii)$] ${\mathcal W}$ and $\widehat{{\mathcal W}}$ uniquely determine each other
in the sense that the paths of ${\cal W}$ a.s. do not cross with (backward) paths in 
$\widehat{{\cal W}}$. The interaction between the paths in ${\mathcal W}$ and $\widehat{\mathcal W}$ is that of Skorohod reflection (see \cite{STW00}).
\end{itemize}

Before concluding this section, we explain the notion of `non-crossing' paths
as this notion will be frequently used in this of the paper. 
 Two paths $\pi_1 , \pi_2 \in \Pi$
are said to be non-crossing 
if there does not exist any $s_1, s_2 \in [\sigma_{\pi_1}\vee \sigma_{\pi_2}, 
\infty)$ such that 
\begin{align}
\label{eq:NonCrossing}
(\pi_1(s_1) - \pi_2(s_1))(\pi_1(s_2) - \pi_2(s_2)) > 0.
\end{align}
It follows that for the PH model, paths are a.s. 
non-crossing. For any $n \geq 1$, clearly 
${\cal X}_n(\gamma, \sigma)$ a.s. forms a family of non-crossing paths and it's 
closure in $\Pi$ denoted by $\overline{{\cal X}}_n(\gamma, \sigma)$ which is 
a $({\cal H}, {\cal B}_{\cal H})$-valued random variable a.s. We will show that as $n\to \infty$, 
the $({\cal H}, {\cal B}_{\cal H})$-valued random variable 
$\overline{{\cal X}}_n(\gamma, \sigma)$ converges 
in distribution to the Brownian web ${\cal W}$ (see Theorem \ref{thm:PerturbedHoward_Bweb}).

\section{The joint exploration process and a sequence of `In' steps}
\label{sec:Expreg}

Fix $k \geq 1$. In this section we start from $k$ many \textit{lattice} 
points $\x_1, \cdots, \x_k \in \Z^2$ such that 
$\x_1(2) = \cdots = \x_k(2) $ and we consider the joint exploration process 
of successive steps $\{h^n(\x_1), \cdots, 
h^n(\x_k) : n \geq 0\}$. Note that the starting points $\x_1, 
\cdots, \x_k$ are not necessarily in $V$ and $h^n(\x_i) = h^n(\x_i, V)$ for all $1 \leq i \leq k$. 
Without loss of generality 
we consider $\x_1(2) = \cdots = \x_k(2) = 0$. 
For $n \in \Z$, let ${\cal F}_n$ denote the $\sigma$ field
\begin{align}
\label{def:F_j_SigmaField}
{\cal F}_n := \sigma \bigl( \Gamma_\w  : \w(2) \leq n \bigl).
\end{align} 
We observe that the joint exploration process $\{h^n(\x_1), \cdots, h^n(\x_k) : n \geq 0\}$
 is measurable w.r.t. the filtration $\{{\cal F}_n : n \geq 0\}$.
We note that the filtration $\{{\cal F}_n : n \geq 0\}$ is {\it not} the
natural or minimal filtration for the joint exploration process.   
Nevertheless, we continue to work with this filtration.     

We introduce some notations. For $l \in \Z$, let 
$$
\mathbb{H}^+(l) := \{\w \in \Z^2 : \w(2) > l\}
\text{ and }\mathbb{H}^-(l) := \{\w \in \Z^2 : \w(2) \leq  l\}
$$ respectively denote the (open) upper and (closed) lower 
half-planes w.r.t. the line $\{y = l\}$. 
For any $n \geq 0$, we observe that the $\sigma$-field  ${\cal F}_n$
has `some' information about certain points in the set $V \cap \mathbb{H}^+(n)$, viz., 
those vertices in $V$ obtained as perturbations of open lattice points in
$\mathbb{H}^-(n)$. The $\sigma$-field has information about all 
$\tilde{\w}  \in V \cap \mathbb{H}^+(n)$ with $\w \in \mathbb{H}^-(n)$ and $B_\w = 1$. 
Such points can affect the distribution of subsequent steps and consequently, 
the joint exploration process $\{h^n(\x_1), \cdots, 
h^n(\x_k) : n \geq 0\}$ is {\it not} Markov. In the next subsection we show that, 
together with this information about ${\cal F}_n$ `explored' points in the upper half-plane 
$\mathbb{H}^+(n)$, the joint exploration process exhibits Markov property.

\subsection{Markov property of the joint exploration process}
\label{subsec:JtExp_Markov}

We recall that we are studying the joint exploration process of $k$ paths starting 
from $\x_1, \cdots  , \x_k$ with $\x_1(2) = \cdots  = \x_k(2) = 0$.  
We need to introduce some more notations. For $l \in \Z$ we partition the set $V$ as 
\begin{align}
\label{def:V_Subsets}
V^+_l & := \{ \tilde{\w} = \w + \Lambda_\w : \w  \in \mathbb{H}^+(l), B_\w = 1\} \text{ and }\nonumber\\ 
V^-_l & := \{ \tilde{\w} = \w + \Lambda_\w : \w  \in \mathbb{H}^-(l), B_\w = 1\}.
\end{align} 
We observe that $V^+_l \subset \mathbb{H}^+(l)$ and 
the $\sigma$-field ${\cal F}_l$ does not have any 
 information about the set $V^+_l$. On the other hand, the set $V^-_l$ has been completely 
explored by ${\cal F}_l$ and the set $V^-_l$ is not necessarily contained in the lower 
half-plane $\mathbb{H}^-(l)$. As discussed earlier, 
the information that the $\sigma$-field ${\cal F}_l$ has 
about the point set $V \cap \mathbb{H}^+(l)$, is given by the set 
\begin{align}
\label{def:InformationSet}
I_l := V^-_l \cap \mathbb{H}^+(l) = \{ \tilde{\w} = \w + \Lambda_\w 
\in \mathbb{H}^+(l) : \w  \in \mathbb{H}^-(l), B_\w = 1\}.
\end{align} 
In other words, $I_l$ represents the information that the $\sigma$-field
${\cal F}_l$ has about the point set $V\cap \mathbb{H}^+_l$. 
The next proposition shows that, together with this information set, 
the joint exploration process is Markov.   
\begin{proposition}
\label{prop:Markov}
The process $\{ (h^{n}(\x_1), \cdots, h^{n}(\x_k), I_n) : n \geq 0\}$
is Markovian.
\end{proposition}
\begin{proof}
We consider independent collection of i.i.d. random vectors 
$$
\bigl \{ \Gamma^{\text{ind}}_\w := \bigl (  B^{\text{ind}}_\w, R^{\text{ind}}_\w, \Lambda^{\text{ind}}_\w
\bigr ) : \w \in \Z^2 \bigr \}
$$
independent of the collection $\{ \Gamma_\w : \w \in \Z^2\}$. 
Fix $n \geq 1$. Conditional on the event $\{(h^n(\x_1), 
\cdots, h^n(\x_k), I_n) = (\w_1, \cdots, \w_k, \Delta)\}$, we consider the point process 
\begin{align*}
\Psi = \Psi(\w_1, \cdots, \w_k, \Delta):= \{ \w + \Lambda^{\text{ind}}_\w : \w \in \mathbb{H}^+(n), B^{\text{ind}}_\w = 1\} \cup \Delta
\end{align*} 
in the upper half-plane $\mathbb{H}^+(n)$. 

We observe that starting from $\w_1, \cdots, \w_k$ together with
the information set $\Delta$, the future evolution of the process has the same
distribution as starting with the point process $\Psi$ on the upper half-plane $\mathbb{H}^+(n)$.  
In other words, we have the following 
\begin{align*}
(h^{n+1}(\x_1), \cdots, h^{n+1}(\x_k), I_{n+1}) \mid & \bigl( 
(h^{n}(\x_1), \cdots, h^{n}(\x_k), I_n) = (\w_1, \cdots, \w_k, \Delta), {\cal F}_n \bigr )\\
& \stackrel{d}{=} f \bigl( (\w_1, \cdots, \w_k, \Delta), 
\{ \Gamma^{\text{ind}}_\w  : \w \in \Z^2\} \bigr ),
\end{align*}
for some measurable function $f$. 
Hence, by the random mapping theorem (see \cite{LPW08}) Proposition \ref{prop:Markov} follows.
\end{proof}
In the next subsection we define a sequence of random steps such that 
starting from these steps, future evolution of each of these $k$ paths stays 
inside a specific region. Later  we will use this sequence to construct
a (random) subsequence which will give us sequence of renewal steps.   

\subsection{Sequence of `In' steps}
\label{subsec:In_steps}
We first define a specific region of our interest in the upper-half plane.
We consider the parabolic curve $\Upsilon := \{( \pm y^2, y) : y \geq 0\}$.
The region `inside' this parabolic curve is denoted as 
\begin{align*}
\nabla = \nabla(\mathbf{0}) := \{( x , y) \in \R^2 : y \geq 0, x \in [- y^2, y^2] \} .
\end{align*}
For $\x \in \R^2$ and for any subset $O \subset \R^2$, the notation $\x \oplus O$ denotes the set 
$\{\x + \y : \y \in O \}$, i.e, the set $O$ translated by $\x$.
For $\u \in \R^2$, let $\nabla(\u) := \u \oplus \nabla$ denote
 the translated version of $\nabla$ translated by $\u$. 
 Below we list an important {\it nesting} property between these parabolic regions.

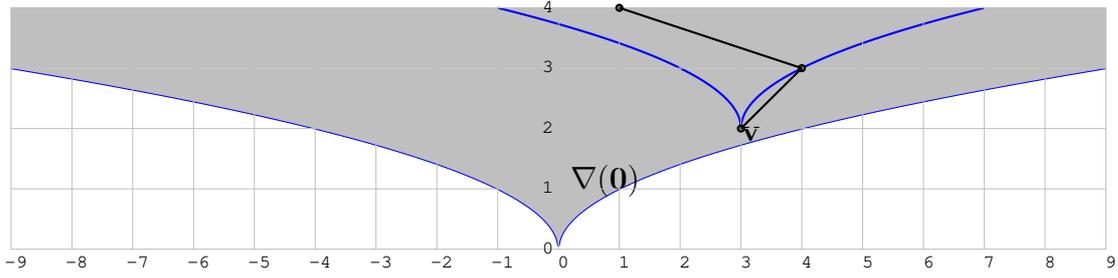
\begin{figure}
\begin{center}
\psset{unit=.8 cm}
\begin{pspicture}(-9,0)(9,4)

\psplot[plotpoints=2000,linecolor=blue]{-9}{9}{ x abs sqrt } 

\pscustom[linestyle=none,fillstyle=solid,fillcolor=lightgray]{%
\psline(-9,3)(-9, 4)
\psline(-9,4)(9,4)
\psline(9, 4)(9,3)
 \psplot[plotpoints=2000,linecolor=blue]{-9}{9}{ x abs sqrt }
  
        }

\psplot[plotpoints=2000, linecolor=blue]{-1}{7}{ x 3 sub abs sqrt 2 add }
\put(3,1.8){$\mathbf{v}$}
\put(0.2,1){$\nabla(\mathbf{0})$}

\psgrid[gridcolor=lightgray, gridwidth=0.25pt, subgriddiv=1, gridlabels=7pt](0,0)(-9,0)(9,4.0)

\pscircle[fillcolor=gray,fillstyle=solid](3,2){.05}
\pscircle[fillcolor=gray,fillstyle=solid](4,3){.05}
\pscircle[fillcolor=gray,fillstyle=solid](1,4){.05}
\psline(3,2)(4,3)
\psline(4,3)(1,4)

\end{pspicture}
\caption{ The parabolic curve $\{( \pm y^2, y) : y \geq 0\}$ is represented in this figure. $\nabla(\mathbf{0})$ is the gray region inside
this curve. Observe that for any $\v \in \nabla(\mathbf{0}) \cap \Z^2$ we have 
$\nabla(\v) \subseteq \nabla(\mathbf{0})$ as mentioned in Remark \ref{rem:Nesting}. 
`In' event occurs at $\v$ as the PH path $\{h^j(\v) : j \geq 1\}$
starting from $\v$ stays inside the region $\nabla(\v)$ throughout.}
\label{fig:InEvent_1}
\end{center}
\end{figure}

\begin{remark}[{\bf Nesting property:}]
\label{rem:Nesting}
For  any $\u, \v \in \R^2$ with $\v \in \nabla(\u) $, 
the regions $\nabla(\u)$ and $\nabla(\v)$ are nested in the sense that $
\nabla(\v) \subseteq \nabla(\u)$ (refer to Figure \ref{fig:InEvent_1}  for an illustration). 
\end{remark}
For $\v \in \Z^2$, we define the event $\text{In}(\v)$ as 
\begin{equation}
\label{def:InEvent}
\text{In}(\v) := \{ h^n(\v) = h^n(\v, V) \in \nabla(\v) \text{ for all }n \geq 1 \}.
\end{equation}
We refer the reader to Figure \ref{fig:InEvent_1} for an example of `In' event.  

For the joint exploration process of $k$ paths starting from $\x_1, \cdots , \x_k$, 
we say that the `In' event occurs
at the $n$-th step if the event $\cap_{i=1}^{k} \text{In}(h^n(\x_i))$ occurs. 
In other words, occurrence of the `In' event at the $n$-th step ensures 
that for each $1 \leq i \leq k$, the PH path
starting from $h^n(\x_i)$ stays inside the region $\nabla(h^n(\x_i))$. 
Now we are ready to define our sequence of `In' steps.

Set $\tau_0 = \tau_0(\x_1, \x_2, \cdots, \x_k) = 0$. For $j \geq 1$ we
define the random step $\tau_j$ as 
\begin{align}
\label{def:Tau_Step}
\tau_j = \tau_j(\x_1, \x_2, \cdots, \x_k) &  
:= \inf\{ n > \tau_{j -1 } : \text{ `In' event occurs}\} \nonumber\\
& = \inf\{ n > \tau_{j -1 } : \text{ event }
\cap_{i=1}^{k} \text{In}(h^n(\x_i)) \text{ occurs}\}.
\end{align}
First we need to show that the r.v. $\tau_j$ is a.s. finite for all $j \geq 1$.
We will do that shortly (in Proposition \ref{prop:TauExpTail}). For the moment, we 
assume that $\tau_j$ is well defined for all $j \geq 1$ and proceed.  
We observe that the r.v. $\tau_j$ is {\it not} a stopping time 
w.r.t. our filtration $\{{\cal F}_n : n \geq 0\}$. We need to extend our filtration  
to make it a stopping time. For $n \geq 1$, we define the $\sigma$-field
\begin{align}
\label{def:F_Tilde_Sigma_Field}
\overline{{\cal F}}_n := \sigma \bigl ( {\cal F}_n, \text{In}(h^m(\x_i)) \text{ for }
0 \leq m \leq n , 1 \leq i \leq k \bigr ).
\end{align}
For each $j \geq 1$, the r.v. $\tau_j$ is  a stopping time 
w.r.t. the extended filtration $\{ \overline{{\cal F}}_n : n \geq 0\}$. This allows us to define 
the filtration 
\begin{align}
\label{def:G_Sigma_Field}
\{ {\cal G}_j := \overline{{\cal F}}_{\tau_j} : j \geq 1 \}.
\end{align}
We observe that for all $j \geq 1$, the r.v. $\tau_j$ is ${\cal G}_j$ measurable.  
The next proposition implies that for all $j \geq 1$, the stopping time $\tau_j$
is a.s. finite. Before we proceed further, it is important 
to mention that several results of this paper involve constants. 
For the sake of clarity, we will use $C_0$ and $C_1$ to denote two positive constants, whose exact values may change from one line to the other. The important thing is that both $C_0$ and $C_1$ are universal constants whose values will depend {\it only} 
on parameters of the process, viz., $p, \theta_x, \theta_y$ and  $k$ (the number of trajectories
considered). We are now ready to state our result which would imply $\tau_j$ 
is a.s. finite for all $j \geq 1$.
\begin{proposition}
\label{prop:TauExpTail}
For any $j \geq 0$ there exist constants $C_0, C_1 > 0$, which do not depend on
 $j$,  such that for all $n \in \N$ we have 
 $$
 \P(\tau_{j+1} - \tau_j > n \mid {\cal G}_j) \leq C_0 \exp{(-C_1 n)}.
 $$
\end{proposition}
We need other results to prove Proposition \ref{prop:TauExpTail}. 
First we need to introduce certain notions of exponential tails for a general family
of random variables. We require these notions to prove Proposition \ref{prop:TauExpTail}
and they will be used repeatedly in later part of this paper . 
\begin{definition}
\label{def:Exp_Tail}
We say that a family of r.v.'s $\{X_i : i \geq 1\}$ has uniform exponential tail 
decay if uniformly for all $i \geq 1$ there exist constants $C_0, C_1 > 0$ such that 
$$
\P(X_i > n) \leq C_0\exp{ (- C_1 n) } \text{ for all }n, i \in \N.
$$

We say that a family of r.v.'s $\{X_i : i \geq 1\}$ has strong uniform exponential tail 
decay if uniformly for all $i \geq 1$ there exist constants $C_0, C_1 > 0$ such that 
\begin{align}
\label{eq:StrongExpDecay_1}
\P \bigl ( X_i > n \mid (X_{i-1},\cdots, X_1) \bigr )
\leq C_0 \exp{(-C_1 n)} \text{ for all }n \in \N.
\end{align}
\end{definition} 
We need Proposition \ref{cor:SubExpDecay_1} regarding tail decay of a random sum of r.v.'s with
exponential tail decay to prove Proposition \ref{prop:TauExpTail}.  
\begin{corollary}  
\label{cor:SubExpDecay_1} 
Consider a family of r.v.'s $\{ X_i : i \geq 1\}$ with strong uniform exponential 
tail decay and a exponentially decaying positive integer valued r.v. $Y$.
Then there exist $C_0, C_1 > 0$ such that for all $ n \in \N$ we have
$$
\P( \sum_{i=1}^Y X_i > n) \leq  C_0 \exp{(-C_1 n)}.
$$
\end{corollary}
\begin{proof}
Firstly, we observe that because of (\ref{eq:StrongExpDecay_1}) in Definition 
\ref{def:Exp_Tail} there exists a
r.v. $W$ such that for all $i \geq 1$ we have
\begin{align}
\label{eq:StrongExpDecay_2}
\P \bigl ( X_i > n \mid (X_{i-1},\cdots, X_1) \bigr )
 \leq  \P(W > n)  \leq C_0 \exp{(-C_1 n)} \text{ for all }n \in \N.
\end{align} 
Next, independent of the family $\{X_i  : i \geq 1\}$, 
we generate i.i.d. copies $W_1, W_2, \dotsc, $ of $W$ where $W$ is as in (\ref{eq:StrongExpDecay_2}). 
Because of Lemma of 2.7 \cite{RSS16A} to prove Corollary \ref{cor:SubExpDecay_1}
it is enough to show that for all $j \geq 1$ we have 
\begin{align}
\label{eq:StrongExpDecay}
\P( \sum_{i=1}^j X_i > n) \leq \P( \sum_{i=1}^j W_i > n).
\end{align}  
We prove (\ref{eq:StrongExpDecay}) using the method of induction.  
Clearly, the choice of $W_1$ in (\ref{eq:StrongExpDecay}) holds for $j=1$.  
Assuming that (\ref{eq:StrongExpDecay}) holds for $j \geq 1$ we obtain
\begin{align*}
& \P (\sum_{ i=1}^{j+1} X_i > n )  \\
 & =  \sum_{ m } P (\sum_{ i=1}^{j} X_i = m ) P ( X_{j+1} > n - m | X_1, X_2, \dotsc, X_{j} )\\
 & \leq  \sum_{ m } P (\sum_{ i=1}^{j} X_i = m ) P (W_{j+1} \geq  n-m) \\
 & \leq  \sum_{ m } P (\sum_{ i=1}^{j} W_i = m ) P (W_{j+1} \geq  n-m) \\
 & =  P (\sum_{ i=1}^{j+1} W_i  \geq n ) . 
\end{align*}
The last inequality follows from the induction hypothesis. 
This completes the proof.  
\end{proof} 
\begin{remark}
\label{rem:UnifSubExpTail}
We note that Proposition \ref{cor:SubExpDecay_1} does not require $\{X_i : i \in \N\}$
to be a family of i.i.d. r.v.'s. Further, the above corollary does not assume independence of 
$Y$ and the family $\{X_ i : i \geq 1\}$. We only require strong uniform exponential tail decay 
for the family $\{X_ i : i \geq 1\}$ as mentioned in Definition (\ref{def:Exp_Tail}).
It should be observed that Proposition \ref{prop:TauExpTail} actually states  
strong uniform exponential tail decay for the family $\{ \tau_j  : j \geq 1\}$
where the decay constants depend only on parameters of the process
and on $k$ (number of paths considered).     
\end{remark}


Proposition \ref{prop:TauExpTail} will be proved through a sequence of lemmas.
In the next lemma, we show that given ${\cal F}_n$, the probability of occurrence of `In' event 
at the $n$-th step $(h^n(\x_1), \cdots , h^n(\x_k))$
has a strictly positive lower bound which does not depend on $n \geq 1$ or on the 
choice of starting points $\x_1, \cdots, \x_k$. 
\begin{lemma}
\label{lem:InEventProb_bd}
There exists  $p_{\text{in}} = p_{\text{in}}(p,\theta_x, \theta_y, k) > 0$, 
which does not depend on $n$ and the starting points $\x_1, \cdots, \x_k$, such that 
for any $n \geq 0$ we have  
\begin{equation*}
\P(\cap_{i=1}^k \text{In}(h^n(\x_i)) \mid {\cal F}_n ) \geq p_{\text{in}}.
\end{equation*}
\end{lemma}
\noindent 
In order to prove the above lemma we need to introduce a `special' 
subset of $V$ denoted by $V_{\text{sp}}$ and defined as 
\begin{align}
\label{def:Sp_vertices}
V_{\text{sp}} := \{ \u  \in \Z^2 : B_\u = 1, X_\u = Y_\u = 0\}.
\end{align}
In other words, $V_{\text{sp}}$ represents the collection of open points in $V$
with \textit{no} perturbations at all. The distribution of the random vector $\Gamma_\u$,
defined as (\ref{eq:PerturbRV}), ensures that the set $V_{\text{sp}}$ must be non-empty a.s.
and for any $\u \in \Z^2$ we have 
\begin{align}
\label{eq:P0}
\P(\u \in V_{\text{sp}}) = \P(B_\u = 1)
\P( X_\u =  Y_\u = 0) = p_0 > 0.
\end{align}
It is not difficult to see that for any $n \in \N$, the set of special points in 
the upper half-plane $\mathbb{H}^+(n)$, given by $V_{\text{sp}} \cap \mathbb{H}^+(n)$, 
is independent of the $\sigma$-field ${\cal F}_n$. This observation enables 
us to obtain exponential decay for the increments of the PH process as follows: 
given ${\cal F}_n$ for any $\w \in \mathbb{H}^+(n)$ we have 
$$
\P((h(\w) - \w)(1) > m \mid {\cal F}_n ) \leq (1-p_0)^{2m + 1} \text{ for all }m \in \N.
$$
This readily gives us the following corollary:  
\begin{corollary}
\label{cor:PH_decay}
Fix any $\alpha, \beta > 0$ with $\alpha > \beta > 0$. Then 
we have 
$$
\P \bigl ( \max\{ |h^j(\mathbf{0})(1)| : 1 \leq j \leq n^\beta \}  > n^\alpha \mid {\cal F}_0 \bigr)
\leq C_0 \exp{(-C_1 n^{\alpha - \beta})}.
$$
\end{corollary}   
Next we prove Lemma \ref{lem:InEventProb_bd} for $k=1$, i.e., for the marginal process 
$\{h^n(\x_1) : n \in \N\}$. Later for general $k\geq 2$, Lemma \ref{lem:InEventProb_bd} 
is proved by invoking the FKG inequality for certain `increasing' events (see Lemma \ref{def:FKG}).

\noindent {\bf Proof of Lemma \ref{lem:InEventProb_bd} for $k = 1$ : } 
We first explain the heuristics. The idea is to create a `shield' of special points 
which would keep the PH path inside the parabolic region. 
This shield should be sufficiently spread out to ensure that the PH path is
enclosed in the parabolic region has a positive probability of occurrence.
This motivates our choice of parabolic regions in the definition of `In' event. 
To make everything rigorous, we need to introduce some notations.
   
For $m \in \N$, let $I^R_m, I^L_m \subset \Z^2$ denote the sets
\begin{align}
\label{def:I_m_set}
I^R_m = I^R_m(\mathbf{0}) &:= \{(u,m) \in \Z^2 : u \in [(m-1)^2 + 1, m^2] \} 
\text{ and }\nonumber\\
I^L_m = I^L_m(\mathbf{0}) &:= \{(u,m) \in \Z^2 : u \in [- m^2, - (m-1)^2 - 1] \}. 
\end{align}
We make certain observations about the above defined sets and we refer the reader to 
Figure \ref{fig:InEvent_2}. 
For any $m \geq 1$, by construction of the sets $I^R_m $ and $I^L_m$ we have:     
\begin{itemize}
\item[(i)]  The right endpoint of the set $I^R_m$ and the 
left endpoint of the set $I^L_m$ lie on the 
parabolic curve $\Upsilon = \{ (\pm y^2, y) : y \geq 0\}$.

\item[(ii)] The $x$ coordinate of the left end point of the set $I^R_{m+1}$ on the line $y = m + 1$ is 
(strictly) larger than the coordinate of the right end point of $I^R_{m}$ on the line $y = m $. 
On the other hand, the coordinate of the right end point of the set $I^L_{m+1}$ on the line $y = m + 1$ 
is (strictly) smaller than the left end point of $I^L_{m}$ on the line $y = m $.
\end{itemize}

%

\begin{figure}
\begin{center}
\psset{unit = .8 cm}
\begin{pspicture}(-9,-0.5)(9,4)

\psplot[plotpoints=2000,linecolor=blue]{-9}{9}{ x abs sqrt }

\psgrid[gridcolor=lightgray, gridwidth=0.25pt, subgriddiv=1, gridlabels=7pt](0,0)(-9,0)(9,3)

\pscircle[fillcolor=black,fillstyle=solid](1,1){.1}
\pscircle[fillcolor=black,fillstyle=solid](-1,1){.1}

\pscircle[fillcolor=gray,fillstyle=solid](4,2){.1}
\pscircle[fillcolor=black,fillstyle=solid](3,2){.1}
\pscircle[fillcolor=black,fillstyle=solid](2,2){.1}
\pscircle[fillcolor=black,fillstyle=none](-2,2){.1}
\pscircle[fillcolor=black,fillstyle=solid](-4,2){.1}
\pscircle[fillcolor=gray,fillstyle=solid](-3,2){.1}

\psline[linecolor=blue](0,0)(1,1)
\psline[linecolor=blue](1,1)(2,2)
\psline[linecolor=blue](2,2)(5,3)

\pscircle[fillcolor=gray,fillstyle=solid](5,3){.1}
\pscircle[fillcolor=black,fillstyle=solid](6,3){.1}
\pscircle[fillcolor=gray,fillstyle=solid](7,3){.1}
\pscircle[fillcolor=gray,fillstyle=none](8,3){.1}
\pscircle[fillcolor=gray,fillstyle=none](9,3){.1}

\pscircle[fillcolor=gray,fillstyle=none](-5,3){.1}
\pscircle[fillcolor=black,fillstyle=none](-6,3){.1}
\pscircle[fillcolor=gray,fillstyle=none](-7,3){.1}
\pscircle[fillcolor=black,fillstyle=solid](-8,3){.1}
\pscircle[fillcolor=black,fillstyle=solid](-9,3){.1}

\end{pspicture}
\caption{The event $A_{\text{sp}}(\mathbf{0})$ creates a shield of `special' points which ensures that the PH path $\{ h^{j}(\mathbf{0}) : j \geq 1\}$ stays inside the region $\nabla(\mathbf{0})$. Black points represent special points and gray points are perturbed versions of open points from some other locations.
There might be some points of $V$ in the intermediate region too. }
\end{center}
\label{fig:InEvent_2}
\end{figure}
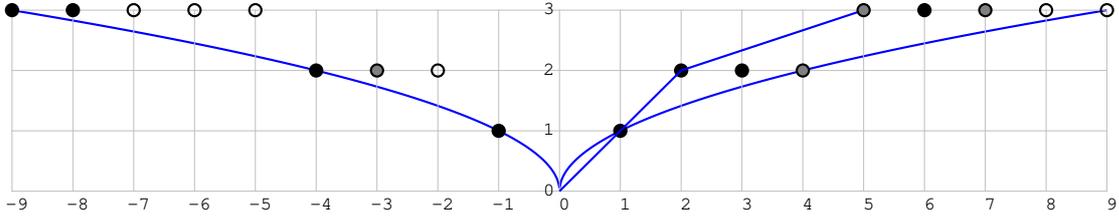


Next, we define the event $A_{\text{sp}}= A^{\text{sp}}(\mathbf{0})$ as
\begin{align}
\label{def:A_sp}
A_{\text{sp}}(\mathbf{0}) := \bigcap_{m=1}^\infty  \{ I^R_m\cap V_{\text{sp}}\neq \emptyset \}\cap \{ I^L_m\cap V_{\text{sp}}\neq \emptyset \}. 
\end{align}
In other words, the event $A_{\text{sp}}(\mathbf{0})$ ensures that for all $m \geq 1$, both the sets $I^R_m$ and $I^L_m$ must have at least one special point each. 
Heuristically, the event $A_{\text{sp}}(\mathbf{0})$ prepares a {\it `shield'} using special points which 
ensures that the path $\{h^j(\mathbf{0}) : j \geq 1\}$
 can not cross the parabolic curve $\Upsilon$
and must stay inside the region $\nabla(\mathbf{0})$ throughout. Figure \ref{fig:InEvent_2}
presents an illustration of the event $A_{\text{sp}}(\mathbf{0})$. Hence, 
we have $A_{\text{sp}}(\mathbf{0}) \subseteq \text{In}(\mathbf{0})$.
For $\u \in \Z^2$, the notation $A_{\text{sp}}(\mathbf{u})$ denotes the 
event $A_{\text{sp}}(\mathbf{0})$ translated to the point $\u$.

Similarly, the occurrence of the event  $A_{\text{sp}}(h^j(\x_1))$ implies the
occurrence of the event $\text{In}(h^j(\x_1))$.
Also $A_{\text{sp}}(h^j(\x_1))$ involves points from the set 
$V^{+}_{j}$ only (because $\x_1(2)=0 $). 
Therefore, it is {\it independent} of the $\sigma$-field  ${\cal F}_j$. This gives us  that
\begin{align*}
\P(\text{In}(h^j(\x_1)) \mid {\cal F}_j) & \geq \P(A_{\text{sp}}(h^j(\x_1)) \mid {\cal F}_j)\\
& = \P(A_{\text{sp}}(h^j(\x_1)) )\\
& = \P(A_{\text{sp}}(\mathbf{0})) \\
& = \prod_{m=1}^{\infty} (1 - (1- p_0)^{2m-1})^2 > 0. 
\end{align*}
where $p_0$ is as in (\ref{eq:P0}). In the penultimate equality we have used the
translation invariance nature of our model. This completes the proof.  
\qed

For proving Lemma \ref{lem:InEventProb_bd} for $k \geq 2$
 we need an FKG property among `increasing' events expressed in terms of special vertices only. 
We consider a natural partial order relation on the space $\{0,1\}^{\Z^2}$
given as :  for $\omega, \omega^\prime \in \{0,1\}^{\Z^2}$
$$
\omega \leq \omega^\prime \text{ if }\omega(\w)\leq \omega^\prime(\w) \text{ for all }
\w \in \Z^2.
$$
\begin{definition} 
\label{def:IncreasingEvent}
An event $A$ measurable with respect to the $\sigma$-field generated by the collection of indicator 
r.v.'s $\bigl \{ \mathbf{1}_{\{\w \in V_{\text{sp}}\}} : \w \in \Z^2 \bigr \}$
is said  to be increasing if 
$$
\omega \in A \Rightarrow \omega^\prime \in A \text{ for all } \omega^\prime \geq \omega. 
$$
\end{definition}  
We observe that the collection of indicator r.v.'s
$\{ \mathbf{1}_{\{\w \in V_{\text{sp}}\}} : \w \in \Z^2 \}$
gives a collection of i.i.d. Bernoulli r.v.'s with success probability $p_0$ where $p_0$
is as in (\ref{eq:P0}). Hence, the following lemma regarding FKG property 
between increasing events (expressible in terms of special vertices only)
follows directly:

\begin{lemma}[FKG property] 
\label{def:FKG}
For any two increasing events $A, B \in \sigma \bigl (\{ \mathbf{1}_{\{\w \in V_{\text{sp}}\}} : \w \in \Z^2 \} \bigr )$
we have 
$$
\P(A \cap B) \geq \P(A) \P(B). 
$$
\end{lemma}
We are now ready to prove Lemma \ref{lem:InEventProb_bd} for general $k \geq 2$.

\noindent{\bf Proof of Lemma \ref{lem:InEventProb_bd} for $k \geq 2$:}
For any $j \geq 1$, we observe the following inclusion of events  
$$
\cap_{i=1}^k A_{\text{sp}}(h^j(\x_i)) \subseteq 
\cap_{i=1}^k \text{In}(h^j(\x_i)).  
$$
Further, given $(h^j(\x_1), \cdots, h^j(\x_k)) = (\w_1, \cdots, \w_k)$,
the events $A_{\text{sp}}(h^j(\x_i)) = A_{\text{sp}}(\w_i)$ for $1 \leq i \leq k$, are
expressed in terms of special points only in the upper half-plane $\mathbb{H}^+(j)$
and all of them are increasing. Further, the $\sigma$-field ${\cal F}_j$ does not 
have any information about the point set $V^+_j$.
Hence, an application of FKG property together with the translation invariance property of our model
give us 
$$
\P \bigl( \cap_{i=1}^k A_{\text{sp}}(h^j(\x_i)) \mid {\cal F}_j\bigr )
\geq \P \bigl( A_{\text{sp}}(\mathbf{0}) \bigr )^k > 0
$$
and complete the proof.
\qed

We are now ready to prove Proposition \ref{prop:TauExpTail}. 
We first explain our heuristics in words. We look for the occurrence of `In' event
at the first step $(h(\x_1), \cdots , h(\x_k))$. Because of Lemma \ref{lem:InEventProb_bd}, 
 the probability of occurrence of the `In' event has a strictly positive lower bound. 
 But, it is important to observe that
on the complementary event, the non-occurrence of `In' event gives
us some information about the point set $V^+_1$. This prevents us from applying 
 Lemma \ref{lem:InEventProb_bd} directly at the next step
  $(h^{2}(\x_1), \cdots, h^{2}(\x_k))$. In fact, given that the `In' event does not 
  occur at $(h(\x_1), \cdots , h(\x_k))$, we can apply the bound obtained in Lemma \ref{lem:InEventProb_bd} only after the effect of this information goes away. 
We will make these heuristics rigorous now and first present a formal proof 
of Proposition \ref{prop:TauExpTail} for $j=0$. In subsection 
\ref{subsubsec:Tau_j} we complete the proof of 
Proposition \ref{prop:TauExpTail} for general $j\geq 1$.

\noindent{\bf Proof of Proposition \ref{prop:TauExpTail} for $j =0$ : }
Given ${\cal F}_1$, at the first step $(h(\x_1), \cdots, h(\x_k))$ we test 
for occurrence of the `In' event. By Lemma \ref{lem:InEventProb_bd} we have that  
$$
\P(\text{`In' event occurs at }(h(\x_1), \cdots, h(\x_k)) \mid {\cal F}_1)\geq p_{\text{in}}.
$$
When `In' event does not occur at the first step $(h(\x_1), \cdots, h(\x_k))$, we wait for 
the joint exploration process to discover this fact, i.e., till the time 
one of the $k$ paths steps outside the respective parabolic region. To  make it rigorous 
for $\u \in \Z^2$ let $\beta(\u)$ denote the r.v. defined as
\begin{align}
\label{def:beta}
\beta(\u) := 
\min\{ n \in \N : h^{n}(\u) \notin \nabla(\u) \}.
\end{align}
Set $\beta_0 = 0$ and $\beta_1 := \min\{\beta(h(\x_i)) : 1 \leq i \leq k\}$.

We mention that the r.v. $\beta_1 $ takes the value $+ \infty$ with positive probability 
and this  implies occurrence of the `In' event. On the other hand, on the event 
$\{\tau_1 > 1\}$ the r.v. $\beta_1$ becomes finite. In fact, the two events $\{\tau_1 > 1\}$ and 
$\{\beta_1 < + \infty \}$ are equal.  Further, we observe that 
$\beta_1$ is a stopping time w.r.t. the filtration $\{{\cal F}_n : n \in \N\}$
and we have 
$$
\{\tau_1 > 1\} = \{\beta_1  < + \infty \} \in {\cal F}_{\beta_1}.
$$ 
On the event $\{\tau_1 > 1\}$, which is same as the event 
$ \{\beta_1  < + \infty \}$, the $\sigma$-field ${\cal F}_{\beta_1}$ 
 does {\it not} have any information about the point set $V^+_{\beta_1}$. 
We can test for occurrence of the `In' event again at  the
$(\beta_1 + 1)$-th step and  Lemma \ref{lem:InEventProb_bd} is applicable here. 
 Thus we have
$$
\mathbf{1}_{\{\beta_1 < \infty\}}\P(\text{In event occurs at the }(\beta_1 + 1)\text{-th step} \mid
{\cal F}_{\beta_1} ) \geq \mathbf{1}_{\{\beta_1 < \infty\}} p_{\text{in}},
$$
by Lemma \ref{lem:InEventProb_bd}. On the event $\{\beta_1 < \infty\}$, 
starting from the $(\beta_1 + 1)$-th step,
we define the r.v. $\beta_2 := \beta_1 + \min\{\beta(h^{\beta_1 + 1}(\x_i)) : 1 \leq i \leq k\}$
which is a stopping time as well. 

The event $\{\beta_2 = + \infty\}$ implies occurrence of the event
$\cap_{i=1}^k \text{In}(h^{\beta_1 + 1}(\x_i))$ and we have $\tau_2 = \beta_1 + 1$. 
Otherwise, i.e., on the event $\{\beta_2 < \infty\}$ the $\sigma$-field ${\cal F}_{\beta_2}$
does not have any information about the point set $V \cap \mathbb{H}^+(\beta_2)$ giving us that 
$$
\mathbf{1}_{\{\beta_2 < \infty\}}\P(\text{In event occurs at the }(\beta_2 + 1)\text{-th step} \mid
{\cal F}_{\beta_2} ) \geq \mathbf{1}_{\{\beta_2 < \infty\}} p_{\text{in}}.
$$
By repeating the same argument recursively we get 
that the number of $\beta_j$'s explored to find the value of $\tau_1$ is dominated by a geometric 
random variable with success probability $p_{\text{in}} > 0$. Unfortunately, 
we do not have independence for $\beta_j$'s. 
Because of Proposition \ref{cor:SubExpDecay_1}, to 
prove Proposition \ref{prop:TauExpTail} it suffices to show that 
for any $j \geq 0$ given $\beta_j < \infty$, the tail probability 
$\P( n < (\beta_{j+1} - \beta_{j})\mathbf{1}_{\{\beta_j < \infty\}} < \infty \mid {\cal F}_{\beta_j} )$
decays exponentially in $n$. 

We show this for $j = 0$. The same argument works for general $j \geq 1$. 
For $\u \in \Z^2$, we define the random variable $J_{\text{sp}}(\u)$ as the minimum 
distance from $\u$ of a vertex from the set $V^{\text{sp}}$ 
at the next level $y = \u(2) + 1$, i.e., 
$$
J_{\text{sp}}(\u) := \min\{ k \geq 0 : \text{ either }\u + (k,1) \text{ or }
\u + (-k,1) \in V_{\text{sp}}\}.
$$ 
We observe that 
\begin{align}
\label{eq:Beta_Estimate}
 \P( n < (\beta_1 - \beta_0 ) \mathbf{1}_{\{\beta_0 < \infty\}} < \infty)  =
  \P( n < \beta_1  < \infty )  =
  \sum_{l =1}^\infty \P(\beta_1 = n + l ). 
\end{align}
In order to estimate the probability $\P(\beta_1 = n + l )$, we observe that for any $1 \leq i \leq k$ 
the PH path starting from $\x_i$ exits the parabolic region $\nabla(\x_i)$
for the first time at the $n+l$-th step only if there is no point of $V$, in particular, 
no special point within $n+l$ distance of the vertex $h^{n+l-1}(\x_i)$ on the line 
$y = n + l$. Hence, (\ref{eq:Beta_Estimate}) becomes       
\begin{align*}
& \sum_{l =1}^\infty \P(\beta_1 = n + l )\\
& \leq   \sum_{l =1}^\infty \P( J_{\text{sp}}(h^{n+l-1}(\x_i)  > 
2(n + l) - 1 \text{ for some }1 \leq i )\\
&  \leq   k\sum_{l =1}^\infty (1 - p_0)^{ 2(n + l) - 1} \\
& \leq  C_0 \exp{(- C_1 n)},
\end{align*}
for some $C_0, C_1 > 0$.
This completes the proof of Proposition \ref{prop:TauExpTail} for $j = 0$.
 \qed

\subsubsection{Proof of Proposition \ref{prop:TauExpTail} for $j \geq 1$}
\label{subsubsec:Tau_j}

We now prove Proposition \ref{prop:TauExpTail} for general $j \geq 1$.
Fix any $j \geq 1$. The problem is that the occurrence of `In' event depends on the
infinite future. Given that the step $\tau_j$ has occurred, 
the $\sigma$-field ${\cal G}_j$ has some information about the random vectors 
$\{\Gamma_\w : \w(2) > \tau_j\}$.  
In order to deal with this, for any $j \geq 1$, we represent the occurrence of the $j$-th $\tau$ step 
as an intersection of  the `In' event at the last step and a ${\cal F}_{\tau_j}$ measurable event.
As an example for the marginal process $\{h^n(\x_1) : n \geq 0\}$, the 
occurrence of the event $\{(h^{\tau_1}(\x_1), \cdots , h^{\tau_j}(\x_1)) = (\v_1, \cdots, \v_j)\}$ 
is represented as an intersection of the event 
 $\text{In}(\v_j)$ and one ${\cal F}_{\v_j(2)}$ measurable event.
 For simplicity of notations we prove this for $k=1$, i.e., for the marginal process only. 
 The same argument holds for general $k \geq 1$.  
 We use this representation and show that given  
 $\{(h^{\tau_1}(\x_1), \cdots , h^{\tau_j}(\x_1)) = (\v_1, \cdots, \v_j)\}$, 
 the probability of the occurrence of the `In' event at the $\tau_j + l$-th step is 
 still bounded below by $p_{\text{in}}$ for any $l \geq 1$.  

Recall that for $\v \in \Z^2$, $\beta(\v)$ is defined as the (extended) integer valued r.v. 
$$
\beta(\v) := \inf\{ l \geq 1: h^{l}(\v) \notin \nabla(\v)\}.
$$  
\begin{lemma}
\label{lem:EventDecomposition_1}
We have the following equality of events:
\begin{align*}
 \{\tau_1 = l, h^l(\x_1) =  \v \} = \Bigl [ \{ h^l(\x_1) = \v\} \cap 
\bigl [ \cap_{j=0}^{l-1}
\bigl ( \{ \beta(h^j(\x_1)) \leq l - j \}  \bigr ) \bigr ] \Bigr ] 
\bigcap (\text{In}(\v)).
\end{align*} 
\end{lemma}
\begin{proof}
By definition of the first $\tau$ step $\tau_1$, we have 
\begin{align*}
\{\tau_1 = l, h^l(\x_1) = \v\} = & \{h^l(\x_1) = \v, \text{In}(\v)\} 
\bigcap \bigl [ \cap_{j=0}^{l-1} ( h^j(\x_1) \text{ is not a }\tau\text{ step}) \bigr]\\ 
= & \{h^l(\x_1) = \v, \text{In}(\v)\} \bigcap \bigl[ \cap_{j=0}^{l-1}  (\beta(h^j(\x_1)) < \infty ) \bigr].
\end{align*}
In order to complete the proof, for any $0 \leq j \leq l-1$ we need to show that
 $$
 \{h^l(\x_1) = \v, \text{In}(\v), \beta(h^j(\x_1)) < \infty \} =
\{h^l(\x_1) = \v, \text{In}(\v),\beta(h^j(\x_1)) \leq l - j \}.  
 $$  
 We prove it for $j =0$ and the argument is exactly the same for general $j \geq 1$.
 On the event $\{h^l(\x_1) = \v, \text{In}(\v), \beta(\x_1) > l  \}$, we must have 
$$
h^{l}(\x_1) = \v \in \nabla(\x_1). 
$$
Otherwise, $\beta(\x_1)$ must be smaller than $l$. 
Further, on the event $\text{In}(\v) \cap \{h^j(\x_1) = \v \in \nabla(\x_1)\}$ 
the nesting property mentioned in Remark \ref{rem:Nesting} implies that
\begin{equation*}
h^{l+m}(\x_1) = h^m(\v) \in \nabla(\v) \subset \nabla(\x_1) \text{ for all }m \geq 1.
\end{equation*} 
Hence, on the event $\{h^l(\x_1) = \v, \text{In}(\v), \beta(\x_1) > l \}$, 
the r.v. $\beta(\x_1)$ must take the value $+\infty$ implying that 
$$
\{ h^l(\x_1) = \v, \text{In}(\v), l < \beta(\x_1) < +\infty \} = \emptyset .
$$
This completes the proof.
\end{proof}
\begin{remark}
\label{rem:lem_EventDecomposition_1}
We observe that in Lemma \ref{lem:EventDecomposition_1}, the event 
$$
\{ h^l(\x_1) = \v \} \cap 
\bigl [ \cap_{j=0}^{l-1}
\bigl ( \{ \beta(h^j(\x_1)) \leq l - j \}  \bigr ) \bigr ]
$$
is ${\cal F}_l$ measurable by construction. 
\end{remark}
We would like to have a similar result for $\tau_j$-th ($ j \geq 2$) step, i.e., 
consider the event
$$
\{ \tau_1 = l_1, h^{l_1}(\x_1) = \v_1, \tau_2 = l_1 + l_2, h^{l_1 + l_2}(\x_1) = \v_2 \},
$$
 the difficulty is that, both the events, $\text{In}(\v_1)$  and $\text{In}(\v_2)$,
 depend on the infinite future. In order to deal with this, for $l \in \N$ and $\v \in \Z^2$, 
we define the event $\text{In}^{(l)}(\v)$ as 
\begin{align}
\label{eq:InEvent_l}
\text{In}^{(l)}(\v) := \{h^m(\v) \in \nabla(\v) \text{ for all }1 \leq m \leq l\}.
\end{align}
Clearly, for any $l \in \N$ we have $\text{In}(\v) \subset \text{In}^{(l)}(\v)$. 
We further observe that the event $\text{In}^{(l)}(\v)$ is ${\cal F}_{\v(2) + l}$ measurable.

The nesting property as mentioned in Remark \ref{rem:Nesting} ensures us that, 
on the event $\text{In}(\v)$ for any $l \geq 1$ we have 
$ \nabla(h^l(\v)) \subset \nabla(\v) $.
This allows us to have the following equality of events
\begin{align*}
 \text{In}(\v) \cap \text{In}(h^j(\v)) 
 & = \text{In}^{(j)}(\v)\cap \text{In}(h^j(\v)) \text{ for any }j \geq 1.
\end{align*}
Using the above reasoning, we write the event 
$$
\{ \tau_1 = m_1, \tau_2 = m_2, h^{m_1}(\x_1) = \v_1, h^{m_2}(\x_1) = \v_2 \}
$$ 
as
\begin{align}
\label{eq:EqualityEvent_3}
& \Bigl [ \{ h^j(\x_1) \text{ is not a }\tau \text{ step for any } 1 \leq j < m_1, h^{m_1}(\x_1) = \v_1, 
\text{In}^{ (m_2 - m_1)}(\v_1) \} \nonumber\\
& \cap 
\{ h^j(\x_1) \text{ is not a }\tau \text{ step for any }  m_1 < j < m_2, 
h^{m_2}(\x_1) = \v_2\} \Bigr ] \bigcap \text{In}(\v_2).
\end{align}
As observed earlier, other than the event $\text{In}(\v_2)$ in Equation 
(\ref{eq:EqualityEvent_3}), rest of the events are ${\cal F}_{m_2}$ measurable.
In fact, Equation (\ref{eq:EqualityEvent_3}) can be further strengthened 
for any $j \geq 1$ as described in the following corollary. 
\begin{corollary}
\label{cor:IncrDistEquality_1}
For any $j \geq 1$ we have the following equality of events
\begin{align}
\label{eq:EventEquality_4}
\bigcap_{l=1}^j \{\tau_l = m_l, h^{m_l}(\x_1) = \v_l \} =  
& \Bigl [ \bigcap_{l=1}^j  \{ h^{m_l}(\x_1) = \v_l\}\cap \bigl \{ \cap_{n = m_{l-1} + 1}^{m_l - 1} 
( \beta(h^n(\x_1)) \leq m_l - n)\bigr \}
 \nonumber\\
&  \bigcap (\cap_{l=1}^{j-1}
\text{In}^{ (m_{(l+1)} - m_l)}(\v_l) )\Bigr ]\bigcap \text{In}(\v_j),
\end{align} 
and other than the event $\text{In}(\v_j)$, rest of the other events 
in the r.h.s. of (\ref{eq:EventEquality_4}) are ${\cal F}_{\v_l(2)} = {\cal F}_{m_l}$ measurable. 
\end{corollary}
For simplicity of notation we write 
$$
E(\v_1, \cdots , \v_j) := 
\Bigl [ \bigcap_{l=1}^j  \{ h^{m_l}(\x_1) = \v_l\}\cap \bigl \{ \cap_{n = m_{l-1} + 1}^{m_l - 1} 
( \beta(h^n(\x_1)) \leq m_l - n)\bigr \}  \bigcap (\cap_{l=1}^{j-1}
\text{In}^{( m_{l+1} - m_l)}(\v_l) )\Bigr ].
$$
Now we are ready to prove Proposition 
\ref{prop:TauExpTail} for $j \geq 1$. 

\noindent {\bf Proof of Proposition \ref{prop:TauExpTail} : }
Fix any $j \geq 1$. Given $(h^{\tau_1}(\x_1), \cdots , h^{\tau_j}(\x_1)) = (\w_1, 
\cdots , \w_j)$, for any $l \geq 1$ we obtain 
\begin{align*}
& \P\bigl( In(h^{\tau_j + l}(\x_1)) \mid (h^{\tau_1}(\x_1), \cdots , h^{\tau_j}(\x_1)) = (\w_1, 
\cdots , \w_j)\bigr) \\
= & \P\bigl( In(h^{l}(\w_j)) \mid ( E(\w_1,\cdots , \w_j)\cap \text{In}(\w_j))\bigr) \\
= & \P \bigl ( In(h^{l}(\w_j)) \cap ( E(\w_1,\cdots , \w_j)\cap \text{In}(\w_j) )\bigr )/ 
\P(E(\w_1,\cdots , \w_j)\cap \text{In}(\w_j) ) \\
= & \P \bigl ( In(h^{l}(\w_j)) \cap ( E(\w_1,\cdots , \w_j)\cap \text{In}^{(l)}(\w_j) )\bigr )/ 
\P(E(\w_1,\cdots , \w_j)\cap \text{In}(\w_j) ) \\
\geq  & \P \bigl ( A_{\text{sp}}(h^{l}(\w_j)) \cap ( E(\w_1,\cdots , \w_j)
\cap \text{In}^{(l)}(\w_j) )\bigr )/ 
\P(E(\w_1,\cdots , \w_j)\cap \text{In}(\w_j) ) \\
=  & \P(A_{\text{sp}}(h^{l}(\w_j)))\Bigl ( \P \bigl ( E(\w_1,\cdots , \w_j)
\cap \text{In}^{(l)}(\w_j) \bigr )/ 
\P(E(\w_1,\cdots , \w_j)\cap \text{In}(\w_j) ) \Bigr) \\
\geq  & \P(A_{\text{sp}}(\mathbf{0})).
\end{align*}
In the last step together with the translation invariance nature of our model, we 
use the fact that 
$$
\bigl (E(\w_1,\cdots , \w_j)\cap \text{In}(\w_j) \bigr )\subset 
\bigl ( E(\w_1,\cdots , \w_j)\cap \text{In}^{(l)}(\w_j) \bigr ).
$$ 
The penultimate step follows from the fact that the event $A_{\text{sp}}(h^{l}(\w_j))$
depends on the collection $\{\Gamma_w : \w(2) > \w_j(2) + l \}$ and 
is independent of the ${\cal F}_{\w_j(2) + l}$ measurable 
event $( E(\w_1,\cdots , \w_j)\cap \text{In}^{(l)}(\w_j) )$.

Given $(h^{\tau_1}(\x_1), \cdots , h^{\tau_j}(\x_1)) = (\w_1, 
\cdots , \w_j)$, using  a similar argument
for any $\v \in \mathbb{H}^+(\w_j)(2)$ and for any $l \geq 1$ 
we have that 
\begin{align}
\label{eq:Special_decay_TauJ}
& \P \bigl( J_{\text{sp}}(\v) > l \mid (h^{\tau_1}(\x_1), \cdots , h^{\tau_j}(\x_1)) = (\w_1, 
\cdots , \w_j) \bigr ) \nonumber \\
& = \P \bigl( J_{\text{sp}}(\v) > l \mid E(\w_1, 
\cdots , \w_j) \cap  \text{In}(\w_j) \bigr ) \nonumber \\
& = \P \bigl( J_{\text{sp}}(\v) > l \cap E(\w_1, 
\cdots , \w_j) \cap  \text{In}(\w_j) \bigr ) /\P(E(\w_1, 
\cdots , \w_j) \cap  \text{In}(\w_j))\nonumber \\
 & \leq \P \bigl( J_{\text{sp}}(\v) > l \cap E(\w_1, 
\cdots , \w_j)  \bigr ) /\P(E(\w_1, 
\cdots , \w_j) \cap  \text{In}(\w_j))\nonumber \\
& \leq \P \bigl( J_{\text{sp}}(\v) > l \cap E(\w_1, 
\cdots , \w_j)  \bigr ) /\P(E(\w_1,\cdots , \w_j) \cap  A_{\text{sp}}(\w_j))\nonumber \\
& = \P ( J_{\text{sp}}(\v) > l )\P  \bigl( E(\w_1,\cdots , \w_j)  \bigr )  
\Bigl( \P(E(\w_1,\cdots , \w_j)) \P(A_{\text{sp}}(\w_j)) \Bigr )^{-1}\nonumber \\
& = \P \bigl( J_{\text{sp}}(\v) > l \bigr ) \P(A_{\text{sp}}(\mathbf{0}))^{-1}
\leq  C_0 \exp{(-C_1 l)},
\end{align}
for some $C_0, C_1 > 0$ which do not depend on $j$. 
In the penultimate step we have used the fact that 
the event $E(\w_1, \cdots , \w_j)$ is ${\cal F}_{\w_j(2)}$ measurable and hence, independent of the events $\{J_{\text{sp}}(\v) > l\}$ and $A_{\text{sp}}(\w_j)$. The last step follows 
from the translation invariance nature of our model. 

Equation (\ref{eq:Special_decay_TauJ}) allows us to repeat the 
calculations  in (\ref{eq:BetaTail}) for general $j \geq 1$.  
The rest of the argument is exactly the same as in the case of $j=0$. 
This completes the proof.
\qed


Fix $j \geq 1$. Given the $\sigma$-field ${\cal G}_j$, for each $1 \leq i \leq k$ the PH path 
starting from $h^{\tau_j}(\x_i)$ given by $\{h^{\tau_j + n}(\x_i) : n \geq 1 \}$ must stay inside the 
region $\nabla(h^{\tau_j}(\x_i))$. Note that for an open vertex $\w \in \mathbb{H}^-(\tau_j)$, it's perturbed version $\tilde{\w}$ can affect the distribution of the joint  process starting from 
$(h^{\tau_j}(\x_1), \cdots , h^{\tau_j}(\x_k))$ only if 
$$
\tilde{\w} \in \cup_{i=1}^k \nabla(h^{\tau_j}(\x_i)) \text{ for some }1 \leq i \leq k.
$$   
This motivates us to define the set 
\begin{align}
\label{def:H_j_set}
H_{\tau_j} = H_{\tau_j}(\x_1, \cdots , \x_k) := I_{\tau_j}(\x_1, \cdots , \x_k)
 \cap \bigl ( \cup_{i=1}^k \nabla(h^{\tau_j}(\x_i))  \bigr), 
\end{align}
for $j \geq 1$, where $I_n$ is defined as in (\ref{def:InformationSet}).

\section{Renewal steps} 
\label{sec:renewal_def}

Below, we define a sequence of what we call as `renewal steps'.
Set $\gamma_0 = 0$ and for $\ell \geq 1$ define 
\begin{equation}
\label{def:RenewalStep}
\gamma_\ell = \gamma_\ell(\x_1, \cdots , \x_k) := \inf\{ j > \gamma_{\ell - 1} : 
H_{\tau_j}(\x_1, \cdots , \x_k) = \emptyset \},
\end{equation}
where $H_{\tau_j}(\x_1, \cdots , \x_k)$ is defined as in (\ref{def:H_j_set}).

We need to show that for all $\ell \geq 1$, the r.v. 
$\gamma_\ell$ is well defined in the sense that $\gamma_\ell$
is finite a.s. A stronger result will be proved in Proposition \ref{prop:Gamma_Tail}. 
For the moment, we  proceed assuming that $\gamma_\ell$ is a.s. finite for all $\ell \geq 1$.
Note that $\gamma_\ell$ denotes the total number of $\tau$ steps, i.e., number 
of `$\text{In}$' steps required for the 
$\ell$-th renewal step and the r.v. $\sigma_\ell := \tau_{\gamma_\ell}$ denotes 
the total number of steps required for the $\ell$-th occurrence of renewal event.

We observe that the condition $H_{\tau_j}(\x_1, \cdots , \x_k) = \emptyset$
implies that no open point from the lower half-plane $\mathbb{H}^-(\tau_j)$
are allowed to perturb to a vertex inside the parabolic regions $\nabla(h^{\tau_j}(\x_i))$
for any $1 \leq i \leq k$. For $\u \in \Z^2$ we define the event $\text{Out}(\u)$ as 
\begin{align}
\label{def:Out_event}
\text{Out}(\u) := \{ V^-(\u(2))\cap \nabla(\u)\} = \emptyset.
\end{align}   
For the joint PH process $\{(h^j(\x_1), \cdots , h^j(\x_k)) : j \geq 1\}$ we say that 
the `Out' event occurs at the $n$-th step if the event $\cap_{i=1}^k \text{Out}(h^n(\x_i))$ 
occurs. In other words, occurrence of a renewal event implies joint occurrence 
of `In' event and `Out' event. For an illustration we refer the reader to 
Figure \ref{fig:InEvent_3}.   

%

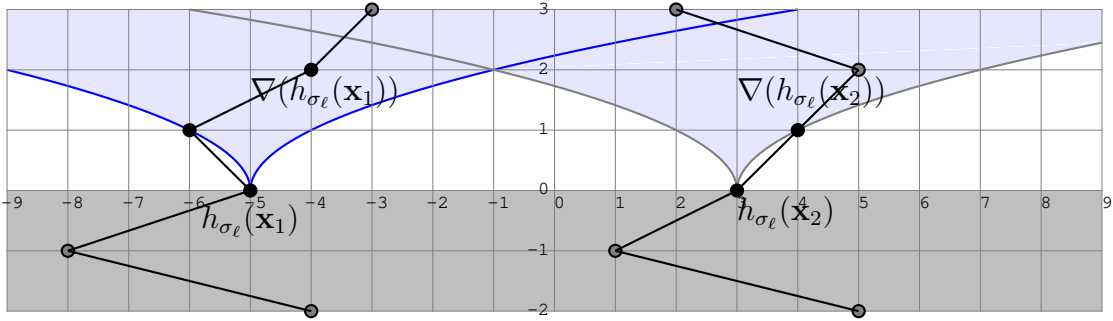
\begin{figure}
\begin{center}
\psset{unit=.8 cm}
\begin{pspicture}(-9.5,-2.5)(9.5,4.5)

\pscustom[linestyle=none,fillstyle=solid,fillcolor=lightblue]{%

      \psline(-9,2)(-9, 3) 
      \psline(-9,3)(-1,3)
      \psline(-1,3)(-1,2) 
      
  \psplot[plotpoints=2000,linecolor=blue]{-9}{-1}{ x 5 add abs sqrt }     
       
        }

\pscustom[linestyle=none,fillstyle=solid,fillcolor=lightblue]{%

    \psline(-1, 2)(-1,3) 
    \psline(-1,3)(9,3)
    
    \psline(9, 3)(!9 dup 3 sub abs sqrt) 
    \psplot[plotpoints=2000,linecolor=red]{-1}{9}{ x 3 sub abs sqrt }
      
        }

  \psplot[plotpoints=2000,linecolor=blue]{-9}{4}{ x 5 add abs sqrt }     
       
    \psplot[plotpoints=2000,linecolor=gray]{-6}{9}{ x 3 sub abs sqrt }
      
\pscustom[linestyle=none,fillstyle=solid,fillcolor=lightgray]{%

      \psline(-9,0)(9, 0) 
      \psline(-9,0)(-9,-2)
      \psline(-9,-2)(9,-2)
      \psline(9, -2)(9,0) 
       
        }
        
\psgrid[gridcolor=gray, gridwidth=0.25pt, subgriddiv=1, gridlabels=7pt](0,0)(-9,-2)(9,3.0)

\pscircle[fillcolor=gray,fillstyle=solid](-4, -2){.1}
\pscircle[fillcolor=gray,fillstyle=solid](-8, -1){.1}
\pscircle[fillcolor=black,fillstyle=solid](-5, 0){.1}
\pscircle[fillcolor=black,fillstyle=solid](-6, 1){.1}
\pscircle[fillcolor=black,fillstyle=solid](-4, 2){.1}
\pscircle[fillcolor=gray,fillstyle=solid](-3, 3){.1}
\psline(-4, -2)(-8, -1)
\psline(-8, -1)(-5, 0)
\psline(-5, 0)(-6,1)
\psline(-6, 1)(-4,2)
\psline(-4, 2)(-3,3)

\pscircle[fillcolor=gray,fillstyle=solid](5, -2){.1}
\pscircle[fillcolor=gray,fillstyle=solid](1, -1){.1}
\pscircle[fillcolor=black,fillstyle=solid](3, 0){.1}
\pscircle[fillcolor=black,fillstyle=solid](4, 1){.1}
\pscircle[fillcolor=gray,fillstyle=solid](5, 2){.1}
\pscircle[fillcolor=gray,fillstyle=solid](2, 3){.1}

 \psline(5, -2)(1, -1)
\psline(1, -1)(3, 0)
\psline(3, 0)(4,1)
\psline(4, 1)(5,2)
\psline(5, 2)(2,3)
\rput(-5,-0.5){$h_{\sigma_\ell}(\mathbf{x}_1)$}
\put(3,-0.5){$h_{\sigma_\ell}(\mathbf{x}_2)$}
\put(-5,1.5){$\nabla(h_{\sigma_\ell}(\mathbf{x}_1))$}
\put(3,1.5){$\nabla(h_{\sigma_\ell}(\mathbf{x}_2))$}
\end{pspicture}
\caption{This figure represents joint renewal for two paths. `In' event condition ensures that PH paths must stay inside the respective $\nabla(\cdot)$ regions. `Out' event condition ensures that perturbed versions of all open points in the lower half-plane (shaded with gray) must be outside the $\nabla(\cdot)$ regions.  }
\label{fig:InEvent_3}
\end{center}
\end{figure}


Note that the event $ \text{Out}(\u)$ is ${\cal F}_{\u(2)}$ measurable 
and, for any $\ell \geq 1$, the r.v. $\gamma_\ell$ is a stopping 
time w.r.t. the filtration $\{{\cal G}_j : j \geq 1\}$ defined as in (\ref{def:G_Sigma_Field}).
This allows us to define the filtration 
\begin{align}
\label{def:S_Filtration}
\{{\cal S}_\ell := {\cal G}_{\gamma_\ell} : \ell \geq 1\}.
\end{align}
For each $\ell \geq 1$, the r.v. $\sigma_\ell$ and the random vector 
$(h^{\sigma_\ell}(\x_1), \cdots, h^{\sigma_\ell}(\x_k))$ are ${\cal S}_\ell$
measurable. 
The next proposition is the main result of this section 
which states that the total number of steps between any two successive renewal steps
has strong uniform exponential tail decay. 
\begin{proposition} 
\label{prop:Sigma_Tail}
Fix any $\ell \geq 0$. For all $n \in \N$ we have 
$$
\P((\sigma_{\ell + 1} - \sigma_{\ell}) > n \mid {\cal S}_\ell)
\leq C_0 \exp{(- C_1 n)}
$$
for some $C_0, C_1 > 0$ which do not depend on $\ell$. 
\end{proposition}
In the next section we will prove Proposition \ref{prop:Sigma_Tail}.

\subsection{Existence of moments for steps between two successive renewals}
\label{subsec:Renewal_Tail_Moments}

Because of Proposition \ref{prop:TauExpTail} and Lemma \ref{lem:ConditionalMarginalTailBound}, 
in order to prove Proposition \ref{prop:Sigma_Tail} 
it is enough to prove strong uniform exponential tail decay type behaviour for
 number of $\tau$ steps required  between two successive renewals
 as mentioned in the next proposition.
\begin{proposition} 
\label{prop:Gamma_Tail}
Fix $\ell \geq 0$. The number of $\tau$ steps between 
$\ell + 1$-th renewal and $\ell$-th renewal is represented by
$(\gamma_{\ell + 1} - \gamma_{\ell})$, 
and for all $n \in \N$ we have 
$$
\P((\gamma_{\ell + 1} - \gamma_{\ell}) > n \mid {\cal S}_\ell)
\leq C_0 \exp{(- C_1 n)}
$$
for some $C_0, C_1 > 0$ which does not depend on $\ell$.
\end{proposition} 
To prove Proposition \ref{prop:Gamma_Tail}, we define a height 
process $\{ L_j : j \geq 1 \}$ which is such that $L_j$ equals 
zero if and only if the information set 
$H_{\tau_j} $  becomes empty.  
We recall that the joint exploration process of $k$ paths starts from points $\x_1, 
\cdots, \x_k$ with $\x_1(2) = \cdots = \x_k(2) = 0$. 
We define our height process as follows.
Set $L_0 = 0$ and define $L_1$ as 
\begin{align*}
L_1 := \sup\{ \w(2) - \tau_1 : \w \in H_{\tau_1} \}\vee 0
= \inf\{n \geq 0 : \mathbb{H}^+(\tau_1 + n) \cap H_{\tau_1} = \emptyset \}. 
\end{align*} 
More generally, for $j \geq 1$ we define the height function $L_j$ as  
\begin{align*}
L_j := \sup\{ \w(2) - \tau_j : \w \in H_{\tau_j} \}\vee 0
= \inf\{n \geq 0 : \mathbb{H}^+(\tau_j + n) \cap H_{\tau_j} = \emptyset \}. 
\end{align*} 
By definition for any $j \geq 1$, the event $L_j = 0$ implies that the corresponding information
set $H_{\tau_j}$ must be empty and vice versa. 
Hence, for any $\ell \geq 1$ we also have the following equivalent representation 
of  the r.v. $\gamma_\ell$ as
\begin{align*}
\gamma_\ell := \inf\{j > \gamma_{\ell-1} : L_j = 0\}.
\end{align*} 

The height process $\{L_j : j \geq 0\}$ is non-negative and not Markov. 
Regarding evolution of this process we obtain the following properties:
\begin{itemize}
\item[(i)] We obtain an important recursion relation (see (\ref{eq:L_j_recursion_rel})) 
in Subsection \ref{subsubsec:BoundLn} regarding evolution of $\{L_j : j \geq 0\}$.
\item[(ii)] Further, we show that the amount of increase in the height process, 
i.e., the random quantity $(L_{j+1} - L_{j})\mathbf{1}_{\{ L_{j+1} > L_{j} \}}$ 
for $j \geq 0$ exhibits strong uniform exponentially decaying tail behaviour. 
\item[(iii)] Lemma \ref{lem:L_decrease_Bound} gives us a strictly positive uniform lower bound 
for the probability that $L_j$ decreases by at least one at the next step.
\end{itemize}
These three properties together allow us to construct 
a non-negative integer valued geometrically ergodic
Markov chain which stochastically dominates evolution of $\{L_j : j \geq 0\}$ 
and prove Proposition \ref{prop:Gamma_Tail}. 
Details are given in Section \ref{subsubsec:PropGammaTail}.

\subsubsection{A (stochastic) bound to control increase in $L_n$}
\label{subsubsec:BoundLn}
 Fix any $j \geq 0$. Given the $\sigma$-field ${\cal G}_j$, 
we consider evolution of the joint process between $\tau_j$-th and $\tau_{j+1}$-th 
step. Given ${\cal G}_j$, the `newly' explored vertices are perturbed versions of 
vertices in the set $\mathbb{H}^-(\tau_{j+1}) \setminus \mathbb{H}^-(\tau_{j})$, i.e., 
vertices those are explored between the $\tau_j$-th and the $\tau_{j+1}$-th steps. 
The contribution of these newly explored vertices to 
$L_{j+1}$ is denoted by a random variable ${\cal N}_{j+1}$ defined as
\begin{align}
\label{def:NewExploredVertices_Contribution}
{\cal N}_{j+1} := 1 +
\sup \{ (\tilde{\w}(2) - \tau_{j+1})\vee 0 : \w \in \mathbb{H}^-(\tau_{j+1}) \setminus 
\mathbb{H}^-(\tau_{j}), \tilde{\w} \in \cup_{i=1}^k \nabla(h^{\tau_{j+1}}(\x_i)) \}.
\end{align}
Actually, ${\cal N}_{j+1}$ represents {\it one excess} to the actual contribution
and the reason behind this modification is to ensure that the recursion relation (\ref{eq:L_j_recursion_rel}) holds with probability $1$. 
The following lemma gives a recursion relation which bounds the amount of 
increase in $L_j$. 
\begin{lemma}
\label{lem:L_Bound}  
For all $j \geq 1$ we have
\begin{equation}
\label{eq:L_j_recursion_rel}
L_{j+1} \leq \max\{ L_j, {\cal N}_{j+1}\} - 1.
\end{equation}
\end{lemma}
\begin{proof}
For any $\w \in \mathbb{H}^-(\tau_j)$ the corresponding random vector $\Gamma_\w$
has been explored already by the $\sigma$-field ${\cal G}_j$ and will be termed as an `older' vertex henceforth. We observe that for any `older' lattice point $\w$, 
which does not contribute to $L_j$, 
two situations can happen:   
\begin{itemize}
\item[(i)] either $\w$ is closed or
\item[(ii)] $\w$ is open but $\tilde{\w} \notin \cup_{i=1}^k \nabla(h^{\tau_j}(\x_i))$.
\end{itemize}
By the nesting property as in Remark \ref{rem:Nesting}, we have that $\nabla(h^{\tau_{j+1}}(\x_i))
\subseteq \nabla(h^{\tau_{j}}(\x_i))$ for all $1 \leq i \leq k$. 
It follows that in both situations (i) and (ii), such an older vertex $\w$, 
which does not contribute to $L_j$, cannot contribute to $L_{j+1}$ as well. 
Since at each time the PH path goes up by one step, for any older 
vertex $\w$ contributing to $L_{j}$, it's contribution to $L_{j+1}$
decreases by {\it at least} one compared to it's contribution to $L_{j}$.    

Further, the above discussion suggests that 
we can have $L_{j+1} \geq L_j$ due to contributions from newly explored 
vertices {\it only}. Hence, the recursion relation (\ref{eq:L_j_recursion_rel}) 
follows from the definition of the random variable ${\cal N}_{j+1}$ as in (\ref{def:NewExploredVertices_Contribution}).
\end{proof}

The next lemma shows that at every $\tau$ step, the probability that 
the height random variable reduces by at least one is uniformly bounded away from zero. 
\begin{lemma}
\label{lem:L_decrease_Bound}
For any $j \geq 1$ given $L_j > 0$, there exists $\tilde{p} > 0$ (which depends only on
parameters of the process and on $k$)
such that we have 
$$
\P( (L_{j+1} - L_{j}) \leq - 1 \mid {\cal G}_j) \geq \P( \NN_{j+1} = 1 \mid {\cal G}_j)
\geq \tilde{p}.
$$
\end{lemma}
\begin{proof}
We prove it for $k = 1$. The proof for general $k$ is similar.
For $\u \in \Z^2$ we define the event 
\begin{align*}
B(\u) = \{\u + (0,1) \in V_{\text{sp}}\}\cap \{\tilde{\v} \notin \nabla(\u + (0,1)) : 
\text{ for all }\v \in \Z^2, \v(2) = \u(2) + 1\}.
\end{align*}
 Given $h^{\tau_j}(\x_1) = \w_j$ and $L_j > 0$, we claim that
 \begin{align}
 \label{eq:Incl_1}
  (B(\w_j)\cap A_{\text{sp}}(\w_j + (0,1)))\subseteq \{(L_{j+1} - L_{j}) \leq - 1 \}.
 \end{align}
We need to justify this event inclusion.   
The occurrence of the event $B(\w_j)$ implies that the vertex $\w_j + (0,1)$ must be a special point
and consequently we must have $h(\w_j) = \w_j + (0,1)$.
Further, occurrence of the event $A_{\text{sp}}(\w_j + (0,1))$ ensures that `In'
event occurs at the next step which makes it  a $\tau$ step as well and gives us 
 $\tau_{j+1} = \tau_j + 1 $. Finally, the event $B(\w_j)$
makes sure that none of the newly explored vertices with $\w(2) = \w_j(2) + 1$
can have their perturbed version in $\nabla(\w_j  + (0,1))$. Since there 
are no new contributions, we have ${\cal N}_{j + 1} = 1$. The recursion relation (\ref{eq:L_j_recursion_rel}) justifies  
 the event inclusion as in (\ref{eq:Incl_1}) and hence, given 
 $(h^{\tau_1}(\x_1), \cdots, h^{\tau_j}(\x_1)) = (\w_1, \cdots, \w_j)$
 we have that
\begin{align*}
& \P\bigl( (L_{j+1} - L_{j}) \leq - 1 \mid 
(h^{\tau_1}(\x_1), \cdots, h^{\tau_j}(\x_1)) = (\w_1, \cdots, \w_j) \bigr)\\ 
& \geq  \P \bigl (  B(\w_j )\cap A_{\text{sp}}(\w_j + (0,1)) \mid E(\w_1, \cdots , \w_j)
\cap \text{In}(\w_j) \bigr )\\
& =  \P \bigl (  B(\w_j)\cap A_{\text{sp}}(\w_j + (0,1)) \cap \text{In}(\w_j)
 \cap E(\w_1, \cdots , \w_j) \bigr )/
\P \bigl (\text{In}(\w_j) \cap E(\w_1, \cdots , \w_j)\bigr )\\
& =  \P \bigl (  B(\w_j )\cap A_{\text{sp}}(\w_j + (0,1)) \cap E(\w_1, \cdots , \w_j) \bigr )/
\P \bigl( \text{In}(\w_j) \cap E(\w_1, \cdots , \w_j) \bigr )\\
& =  \P \bigl (  B(\w_j )\cap A_{\text{sp}}(\w_j + (0,1)) \bigr )\Bigl ( \P( E(\w_1, \cdots , \w_j)) /
\P \bigl( \text{In}(\w_j) \cap  E(\w_1, \cdots , \w_j) \bigr ) \Bigr )\\
& \geq  \P \bigl (  B(\w_j )\cap A_{\text{sp}}(\w_j + (0,1)) \bigr ).
\end{align*}
The penultimate equality follows from the observation that
$$
B(\w_j)\cap A_{\text{sp}}(\w_j + (0,1)) \subseteq \text{In}(\w_j)
$$ 
and the last equality follows from the observation that the events 
$B(\w_j)$ and $A_{\text{sp}}(\w_j + (0,1))$ depend on the collection 
$\{\Gamma_\w : \w \in \mathbb{H}^+(\w_j(2))\}$ 
and hence they are independent of $E(\w_1, \cdots , \w_j)$. 
Further, the events $B(\w_j)$ and $A_{\text{sp}}(\w_j + (0,1))$ depend on
disjoint collection of random vectors which makes them independent.
Therefore, using the translation invariance nature 
of our model we have that   
\begin{align*}
& \P\bigl( (L_{j+1} - L_{j}) \leq - 1 \mid (h^{\tau_1}(\x_1), \cdots, h^{\tau_j}(\x_1)) 
= (\w_1, \cdots, \w_j) \bigr) \\
& \geq   \P \bigl (  B(\w_j )\cap A_{\text{sp}}(\w_j + (0,1)) \bigr )\\
& =  \P (  B(\w_j ) \P(A_{\text{sp}}(\w_j + (0,1))  )\\
& = \P( B(\mathbf{0})) \P(A_{\text{sp}}(\mathbf{0})) > 0.
\end{align*}
In the last step we have used the fact that $\P(B(\mathbf{0})) > 0$
which actually follows from a simple construction. This completes the proof.
\end{proof}

Next, we show that the family of random variables  $\{{\cal N}_j : j \geq 1\}$ 
exhibits strong uniform exponentially decaying tail behaviour. 
It is important to observe that the decay constants do
not depend on the choice of starting points $\x_1, \cdots, \x_k$.  
\begin{lemma}
\label{lem:N_SubExptail}
Fix any $j \geq 1$. For all large $n$ we have 
\begin{align*}
\P({\cal N}_{j} > n \mid {\cal G}_{j-1}) \leq C_0 \exp{(-C_1 n)}, 
\end{align*}
where $C_0, C_1 > 0$ do not depend on $j$.
\end{lemma}
In order to prove Lemma \ref{lem:N_SubExptail} we need to 
introduce some notation. For $\w \in \Z^2$ we define 
the `total' perturbation r.v. as $T_\w := |X_\w| + Y_\w$. 
The distribution of the random vector $\Lambda_\w$
ensures that the tail of the r.v. $T_\w$ decays exponentially. 
For $\w \in \Z^2$ and $m \in \Z$ we consider the 
infinite downward sequence of lattice points starting from $\w + (m,0)$
given by $\{\w + (m, -l) : l \in \N \cup \{0\}\}$. 
A vertical `overshoot' r.v., which represents 
amount of overshoot in the upper half-plane
$\mathbb{H}^+(\w(2))$, based on the collection of total 
perturbation r.v.'s $\{T_{\w + (m,-l)} : l \in \N \cup \{0\}\}$ 
attached to this infinite sequence of lattice points 
$\{\w + (m, -l) : l \in \N \cup \{0\}\}$ is defined by 
\begin{align}
\label{def:VerticalOvershoot}
T^{\uparrow}_m(\w) := \sup\{ T_{\w + (m, -l)} - l : l \in \N \cup\{0\} \}.
\end{align}
It is not difficult to show that the r.v. $T^{\uparrow}_m(\w)$ is non-negative.
Moreover, we have 
\begin{align}
\label{eq:T_Ray_ExpDecay}
\P(T^{\uparrow}_m(\w) > n) \leq \sum_{l = 0}^\infty \P( T_{\w + (m, -l)} > n + l)\leq C_0 \exp{(-C_1 n)},
\end{align}
for some $C_0, C_1 > 0$ which do not depend on $\w$ or $m$.
\begin{remark}
\label{rem:Overshoot_ExpDecay}
In order to obtain (\ref{eq:T_Ray_ExpDecay}) 
we only require uniform exponential tail for the family 
$\{ T_{\w + (m, -l)}  : l \in \N \cup\{0\} \}$ as mentioned 
in Definition \ref{def:Exp_Tail}. For $\w \in \Z^2$ and $m \in \Z$ one can also
consider left, rep.,  horizontal overshoot r.v. as defined below
\begin{align*}
T^{\rightarrow}_m(\w) & := \sup\{ T_{\w + (-l, m)} - l : l \in \N \cup\{0\} \} \text{ and }\\
T^{\leftarrow}_m(\w) & := \sup\{ T_{\w + (l, m)} - l : l \in \N \cup\{0\} \}.
\end{align*}  
The same argument as in (\ref{eq:T_Ray_ExpDecay}) gives us exponentially 
decaying tails for both these r.v.'s, $T^{\leftarrow}_m(\w)$ and $T^{\rightarrow}_m(\w)$. 
We will require these observations later.
\end{remark}
Below we prove Lemma \ref{lem:N_SubExptail} for $j = 1$ and 
later we will indicate the modification required for general $j \geq 2$.

\noindent {\bf Proof of Lemma \ref{lem:N_SubExptail} for $j = 1$:}
We prove it for $k = 1$. The proof for general $k \geq 1$ is exactly the same.
Consider the following events:
\begin{align*}
E^1_n & := \{ \tau_1 \leq n/2 \} \text{ and }\\
E^2_n & := \{ |(h^{j+1}(\x_1) - h^{j}(\x_1))(1)| \leq n \text{ for all }0 \leq j \leq n/2 \}.
\end{align*}
We obtain that 
\begin{align}
\label{eq:}
\P(L_1 \geq n) = \P(\NN_1 \geq n) \leq \P( \NN_1 \geq n, E^1_n\cap E^2_n) + 
\P((E^1_n)^c) + \P((E^2_n)^c).
\end{align}
Proposition \ref{prop:TauExpTail} gives us that $\P((E^1_n)^c)$ goes to zero as $n \to \infty$. 
Also we obtain:
\begin{align*}
& \P((E^2_n)^c) \\
& \leq \P( \cup_{j=0}^{\lfloor n/2 \rfloor } |(h^{j+1}(\x_1)-h^{j}(\x_1))(1)| > n )\\
& \leq \P \bigl( \cup_{j=0}^{\lfloor n/2 \rfloor } 
(J_{\text{sp}}(h^{j}(\x_1)) > n) \bigr  )\\
& \leq (\lfloor n/2 \rfloor + 1) \P( J_{\text{sp}}(\mathbf{0}) > n  ) \\
& \leq (\lfloor n/2 \rfloor + 1)(1 - p_0)^{n}.
\end{align*}  
We observe that on the event $E^1_n \cap E^2_n$ we must have 
$h^{\tau_1}(\x_1) \in [-n^2, n^2]\times [0, n/2]$.
Consider the projection of the starting point $\x_1$ on the line $y = - n$ 
denoted as $\x^{n,\downarrow}_1 = \x_1 + (0, - n)$.
The notation $\x^{n,\uparrow}_1$ denotes the lattice point $\x_1 + (0, \lfloor n/2 \rfloor)$. 
The nesting property as mentioned in Remark \ref{rem:Nesting} ensures that 
on the event $E^1_n \cap E^2_n$ we have 
$$
\nabla(h^{\tau_1}(\x_1)) \subseteq \nabla(\x^{n,\downarrow}_1). 
$$
For an illustration we refer the reader to Figure \ref{fig:InEvent_5}.
Therefore, on the event $E^1_n \cap E^2_n\cap \{ \NN_1 \geq n\}$, the r.v. $\NN_1$
is dominated by the maximum of the overshoot random variables 
$\{ T_m(\x^{n,\uparrow}_1) : m \in \Z \}$ when overshoot is restricted to the region
$\nabla(\x^{n,\downarrow}_1)$.  

%

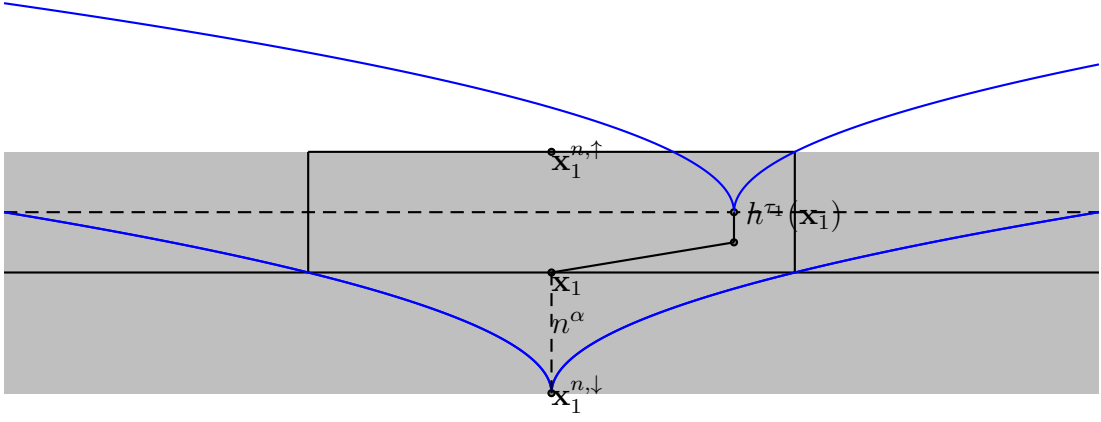
\begin{figure}
\begin{center}
\psset{unit=.8 cm}
\begin{pspicture}(-9,0)(9,4)

\pscustom[linestyle=none,linecolor = lightgray, fillstyle=solid,fillcolor=lightgray]{
\psline(-9,4)(9, 4)
\psline(9,4)(9,0)
\psline(9, 0)(-9,0)
\psline(-9, 0)(-9,4)}
%

\pscircle[fillcolor=gray,fillstyle=solid](0,0){.05}
\pscircle[fillcolor=gray,fillstyle=solid](0,2){.05}
\pscircle[fillcolor=gray,fillstyle=solid](0,4){.05}
\psplot[plotpoints=2000,linecolor=blue]{-9}{9}{ x abs sqrt }

\psline[linestyle=dashed](-9,3)(9,3)

\pscircle[fillcolor=gray,fillstyle=solid](3,2.5){.05}
\pscircle[fillcolor=gray,fillstyle=solid](3,3){.05}

\psline(-9,2)(9,2)
\psline(-4,2)(-4,4)
\psline(4,2)(4,4)
\psline(-4,4)(4,4)
\psplot[plotpoints=2000,linecolor=blue]{-9}{9}{ x abs sqrt } 

\psline(0,2)(3,2.5)
\psline(3,2.5)(3,3)

\psline[linestyle = dashed](0,1.1)(0,2)
\put(0,1){$n^\alpha$}
\psline[linestyle = dashed](0,0)(0,0.9)

\psplot[plotpoints=2000,linecolor=blue]{-9}{9}{ x 3 sub abs sqrt 3 add  } 

%

\put(0,1.7){$\mathbf{x}_1$}
\put(0,3.7){$\mathbf{x}^{n,\uparrow}_1$}
\put(0,-0.2){$\mathbf{x}^{n,\downarrow}_1$}
\put(3.2,2.8){$h^{\tau_1}(\mathbf{x}_1)$}

\end{pspicture}
\caption{ If the $\tau_1$ step occurred within the rectangular box then $\nabla(h^{\tau_1}(\x_1))$ is contained in
$\nabla(\x^{n,\downarrow}_1)$. Therefore, on this event the maximum overshoot amount due to
the total perturbation 
r.v.'s  attached to lattice points in the shaded region (which extends to the infinite lower half-plane)
dominates $L_1$ as well as the r.v. $\NN_1$.}
\label{fig:InEvent_5}
\end{center}
\end{figure}


Therefore, on the event $ E^1_n\cap E^2_n$
in order to have $\NN_1 > n$, we must have either
\begin{itemize}
\item[(i)] $T^\uparrow_m(\x^{n,\uparrow}_1) > n/2$ for some $m \in 
[- 4n^2, 4n^2]\cap \Z$ or
\item[(ii)] $T^\uparrow_m(\x^{n,\uparrow}_1) > n/2 + l$ for some $m$ 
in either $[-(2n + l)^2,-(2n + (l-1))^2]\cap \Z$ or in $[(2n + (l-1))^2,(2n + l)^2]\cap \Z$
for some $l \geq 1$.
\end{itemize}
Hence, using the translation invariance nature of our model 
and applying the union bound we obtain 
\begin{align*}
& \P(\NN_1 > n)\\
& \leq 2\P \Bigl [ \bigl( \cup_{m = 0}^{(2n)^2 }
 T^\uparrow_m(\x^{n,\uparrow}_1) > n/2 \bigr ) 
 \bigcup_{l\geq 1} \bigl ( \bigcup_{m  = (2n+l-1)^2 + 1}^{(2n + l)^2 \rfloor} 
 T^\uparrow_m(\x^{n,\uparrow}_1) > n/2 + l \bigr )\Bigr ]\\
& \leq 2\Bigl ( 4n^2\P( T^{\uparrow}_0(\mathbf{0}) > n/2 ) + 
\sum_{l=1}^\infty 2(2n+l)\P( T^{\uparrow}_0(\mathbf{0}) > n/2  + l) \Bigr).
\end{align*}
Exponential tail decay of the overshoot r.v. $T^{\uparrow}_0(\mathbf{0})$ 
as observed in (\ref{eq:T_Ray_ExpDecay}) completes the proof. We emphasize the 
fact that the choice of decay constants does not depend on $\x_1$. 
\qed

Before proving Lemma \ref{lem:N_SubExptail} for general $j \geq 1$ 
we recall that given ${\cal G}_j$, 
distribution of the random vectors in the upper half-plane $\{\Gamma_\w : 
\w \in \mathbb{H}^+(\tau_j)\}$ no longer remain  i.i.d. 
However, for $\w \in \mathbb{H}^+(\tau_j)$, irrespective of the location of $\w$
 and for any $m \in \Z$, we still have uniform exponential tail decay for the r.v. $T_m(\w)$ . 
\begin{lemma}
\label{lem:ConditionalMarginalTailBound}
Given $(h^{\tau_j}(\x_1), \cdots, h^{\tau_j}(\x_k)) = (\w_1, \cdots , \w_k)$, for any 
$\w \in \mathbb{H}^+(\w_1(2))$, $\alpha > \beta > 0$ and,  for all $n \in \N$, 
there exist $C_0, C_1 > 0$ uniformly, such that we have the following:
\begin{itemize}
\item[(i)]  $\P(T_\w > n \mid {\cal G}_j) \leq C_0 \exp{(-C_1 n)}$;
\item[(ii)] $\P(J_{\text{sp}}(\w) > n \mid {\cal G}_j) \leq  C_0 \exp{(-C_1 n)}$;
\item[(iii)] $\P \Bigl ( \max \bigl \{ T_{\w +(m,-l)} - l : 
0 \leq l \leq \w(2) - \w_1(2) \bigr \} > n \mid {\cal G}_j \Bigr ) 
\leq C_0 \exp{(-C_1 n)}$  for all $m \in \Z$; 
\item[(iv)] $\P \bigl (  \max\{ |(h^l(\w_i) - \w_i)(1)| : 1 \leq l \leq n^\beta, 
1 \leq i \leq k \} >  n^\alpha \mid {\cal G}_j \bigr ) \leq  C_0 \exp{(-C_1 n^{\alpha - \beta})}$.
\end{itemize} 
\end{lemma}
\begin{proof}
We prove Lemma \ref{lem:ConditionalMarginalTailBound} for $k = 1$. 
The argument is same for general $k \geq 1$. Given 
$(h^{\tau_1}(\x_1), \cdots, h^{\tau_j}(\x_1)) = (\w^1_1, \cdots, \w_1^j)$
for any $\v \in \mathbb{H}^+(\w_1^j(2))$  we obtain
\begin{align*}
& \P \bigl ( T_\v  > n \mid (h^{\tau_1}(\x_1), \cdots, h^{\tau_j}(\x_1)) = 
(\w^1_1, \cdots, \w_1^j)\bigr ) \\
& =  \P \bigl ( (T_\v  > n )\cap \text{In}(\w_1^j)\cap E(\w^1_1, \cdots, \w_1^j))/
\P(\text{In}(\w_1^j)\cap E(\w^1_1, \cdots, \w_1^j) \bigr ) \\
& \leq  \P( (T_\v  > n) \cap E(\w^1_1, \cdots, \w_1^j))/\P(\text{In}(\w_1^j)\cap 
E(\w^1_1, \cdots, \w_1^j)) \\
& \leq  \P( T_\v  > n )\P(E(\w^1_1, \cdots, \w_1^j))/\P(A_{\text{sp}}(\w_1^j)
\cap E(\w^1_1, \cdots, \w_1^j)) \\
& =  \P(T_\v  > n )/\P(A_{\text{sp}}(\mathbf{0})) \\
& \leq  C_0 \exp{(- C_1 n)}.
\end{align*}
Clearly, the values of the positive constants $C_0, C_1$ do not depend 
on the point $\v$ or on $(\w^1_1, \cdots , \w_1^j)$. 

The argument for (ii) is exactly the same as that of (i) and hence, we skip it.

For (iii) we first observe that given $(h^{\tau_1}(\x_1), \cdots, h^{\tau_j}(\x_1)) 
= (\w^1_1, \cdots, \w_1^j)$, the r.v. $\max \{ T_{\w +(m,-l)} - l : 
0 \leq l \leq \w(2) - \w_1^j(2)\}$ is dominated by the overshoot r.v. 
$\sup \{ T_{\w +(m,-l)} - l : l \geq 0\}$. By (i) we have that the family  
of r.v.'s $\{ T_{\w +(m,-l)} : 0 \leq l \leq \w(2) - \w_1^j(2)\}$ has uniform exponential tail decay. 
Hence, the same argument for exponential tail decay of overshoot random variable,  
as given in Remark \ref{rem:Overshoot_ExpDecay}, proves (iii). 

The proof of (iv) follows from (ii) together with application of union bound.  
\end{proof}
We now prove Lemma \ref{lem:N_SubExptail} for general $j \geq 2$.

\noindent {\bf Proof of Lemma \ref{lem:N_SubExptail} for $j \geq 2$:}
The proof for general $j \geq 2$ is very similar to that of $j =1$ and
we mention here the required modifications only. In the proof 
for $j =1$, we did not really use the fact that $\{T_\w : \w \in \Z^2\}$ is an 
i.i.d. collection. Our proof essentially uses the union bound and uniform exponential
tail decay of marginal distributions of total perturbation r.v.'s.   
Given ${\cal G}_j$, though the collection $\{ \Gamma_\w : \w \in \mathbb{H}^+(\tau_j)\}$ 
is no longer an i.i.d. collection of random vectors, 
Lemma \ref{lem:ConditionalMarginalTailBound} (i) provides uniform 
exponential tail decay for the family of total perturbation random vectors 
$\{ T_\w : \w \in \mathbb{H}^+(\tau_j)\}$. We observe that given ${\cal G}_j$,
Item (ii) of Lemma \ref{lem:N_SubExptail} ensures uniform 
exponential tail decay for the r.v. $J_{\text{sp}}(\w)$
for $\w \in \mathbb{H}^+(\tau_j)$. On the other hand, Item (iv) of Lemma \ref{lem:N_SubExptail}
gives us that the probability of the corresponding version of the event $(E^2_n)^c$
exponentially decays to $0$ as $n \to \infty$. Finally, the `Out' event condition 
ensures that the family of previously explored perturbation random vectors 
$\{ \Lambda_\w : \w \in \mathbb{H}^-(\tau_j)\}$ does not contribute 
to the r.v. $\NN_{j+1}$. Therefore, Item (iii) of Lemma \ref{lem:N_SubExptail} allows
 us to apply the same argument for general $j \geq 1$ and completes the proof.   
\qed

We are now ready to prove Proposition \ref{prop:Gamma_Tail}. 

\subsubsection{Proof of Proposition \ref{prop:Gamma_Tail}}
\label{subsubsec:PropGammaTail}

Proof of Proposition \ref{prop:Gamma_Tail} is motivated from the proof of Lemma 
2.6 of \cite{RSS16A}. We recall the family of r.v.'s $\{\NN_j : j \geq 1\}$ as in
the recursion relation (\ref{eq:L_j_recursion_rel}). 
We observe that the strong uniform exponential tail decay behaviour (see Definition
\ref{def:Exp_Tail}) of this family, as shown in Lemma \ref{lem:N_SubExptail}, allows us to assume that 
there exists an i.i.d. family of non-negative integer valued r.v.'s $\{ R_j : j \geq 1\}$
with exponential tail decay such that for all $j , m \geq 1$ we have   
\begin{align}
\label{eq:NN_dom_iid}
\P(\NN_j \geq m \mid \NN_{j-1}, \cdots, \NN_1) \leq \P(R_j \geq m) .
\end{align} 
Using such an i.i.d. family $\{ R_j : j \geq 1\}$ we construct a non-negative
integer valued Markov chain $\{M_j : j \geq 0\}$ which stochastically dominates 
the height function $\{ L_j : j \geq 0\}$. 

Set $M_0 = 0$ and for $j \geq 1$ define $M_j$ as $M_j := \max \{M_{j-1} , R_j\} - 1$. 
The i.i.d. nature of the collection $\{ R_j : j \geq 1\}$ ensure that 
the process $\{M_j : j \geq 0\}$ is a time homogeneous Markov chain. Let 
$\tau^M := \inf\{ j \geq 1 : M_j = 0\}$ denote the return time to state zero for this 
Markov chain. The recursion relation (\ref{eq:L_j_recursion_rel}) 
together with the stochastic domination as 
observed in the inequality (\ref{eq:NN_dom_iid}) ensure that $\tau^M$ stochastically 
dominates $\gamma$. Hence, in order to prove Proposition \ref{prop:Gamma_Tail}
it is enough to show that there exist $C_0, C_1 > 0$ such that 
$$
\P(\tau^M > n ) \leq C_0 \exp{(-C_1 n)}.
$$
Lemma \ref{lem:L_decrease_Bound} ensures that we have 
$$
\P(M_{j+1} \leq M_{j} - 1 \mid M_j = m) \geq \tilde{p} \text{ for all }m \geq 1.
$$ 
Therefore, it is not difficult to see that the Markov chain $M_j$ is irreducible.   
It suffices to prove that, for some $\alpha > 0$, we have $\E(\exp{(\alpha \tau M))} <\infty$.

Towards that using Proposition 5.5, Chapter 1 of Asmussen \cite{A03}, 
it suffices to show that there exist a non-negative
function $f : \N \cup \{ 0 \} \mapsto \R^+$, $n_0 \in \N$ and 
$r > 1$ such that $f(j)>\nu$ for some $\nu >0$ and $\E[f (M_1) \mid M_0 = j ] <
\infty $ for all $j \leq n_0$, while for $j > n_0$, $\E[ f(M_1) \mid M_0 = j] \leq  f(j)/r$. 
 
Taking $ f : \{ 0, 1, 2, \cdots \} \mapsto \R $ to be 
$f(i) = \exp{(\alpha i)}$, where $\alpha > 0$ is small enough so that 
$\E [\exp{(\alpha R_1)}] < \infty $ and $\exp{(-\alpha)} < 1/r$, we have
\begin{align*}
 & \E[\alpha(M_{j+1} - M_j) \mid M_j = m] \\
 & =  \exp{(-\alpha)}\P(R_1 \leq m) + 
 \exp{(-\alpha m)}\E[\mathbf{1}_{\{ R_{j+1} > m\}}\exp{(\alpha R_{j+1})}]\\
 & < (1/r) + \exp{(-\alpha m)}\E[\mathbf{1}_{\{ R_{j+1} > m\}}\exp{(\alpha R_{j+1})}]\\
 & \leq (1/r) \text{ for }m \text{ sufficiently large}.
 \end{align*}
The last inequality follows because $\E[\exp(\alpha R_1)] <\infty $ 
guarantees 
$$
\exp{(-\alpha m)}\E[\mathbf{1}_{\{ R_1 > m\}}\exp{(\alpha R_1)}] \to 0 \text{ as }m \to \infty.
$$
\qed

\section{Renewal properties for a single path}
\label{sec:Renewal_k_1}

In this section, we consider $k= 1$ and we explain renewal properties of the marginal process 
$\{h^{\sigma_\ell}(\x_1) : \ell \geq 0\}$.
 More precisely, we will show that 
 the sequence of renewal steps gives rise to a random walk process with i.i.d. 
increments. We define the sequence of successive renewal steps as 
\begin{align}
\label{def:SinglePt_Rwalk_Onepath}
\{ Y_{\ell + 1} = Y_{\ell + 1}(\x_1) := h^{\sigma_{\ell+1}}(\x_1) :  \ell \geq 0\} .
\end{align}
We observe that $\sigma_{\ell + 1}(\x_1) - \sigma_{\ell}(\x_1) = 
(Y_{\ell + 1}(\x_1) - Y_{\ell}(\x_1))(2)$ represents the number of steps elapsed, i.e.,  
the total time taken between $\ell + 1$-th and $\ell$-th renewal. 
Our next proposition explains the renewal structure observed at these random steps. 

\begin{proposition}
\label{prop:SinglePtRwalk}
The sequence $\{ (Y_{\ell+1}-Y_{\ell}) : \ell \geq 1\}$ gives a collection of 
\textit{i.i.d.} random vectors taking values in the space $\Z\times \N$ 
whose distribution does not depend on the choice of the 
starting point $\x_1$.
\end{proposition}
We observe that the above proposition gives i.i.d. increments with a lag. 
Since, the starting condition is not the same as the renewal conditions, 
the first increment random vector given by $(Y_1 - \x_1)$ has a different distribution. We emphasize that  
the above proposition also implies that $\{ \sigma_{\ell + 1}(\x_1) -
 \sigma_{\ell}(\x_1) : \ell \geq 1\}$ forms an i.i.d. 
sequence whose distribution does not depend on $\x_1$. 
This provides a strong uniform exponential tail decay (see Definition \ref{def:Exp_Tail}) 
for the family  $\{ \sigma_{\ell}(\x_1) : \ell \geq 1 \}$.

Proposition \ref{prop:SinglePtRwalk} will be proved through a sequence of lemmas.
We need to introduce few notations first.
In what follows, together with the usual Howard step $h(\u) = h(\u, V)$ we will 
also consider $h(\u, V^+_{\u(2)})$, i.e., the step taken considering only the point set $V^+_{\u(2)}$.  For simplicity of notation, 
$h(\u, V^+_{\u(2)})$ will be simply denoted as $h(\u, V^+)$ and for any $j \geq 1$, 
the $j$-th step $h^j(\u, V^+_{\u(2)})$ (taken considering only the point set $V^+_{\u(2)}$)
 will also be denoted as $h^j(\u, V^+)$. In general this modified step 
 $h^j(\u, V^+_{\u(2)})$ need not be equal 
to the usual $j$-th Howard step $h^j(\u, V) = h^j(\u)$. 
For $\u \in \Z^2$ we recall the definition of the event $ \text{Out}(\u)$ 
which ensures that points from the set $V^-_{\u(2)}$
must be out of the parabolic region $\nabla(\u)$. 

On the event $\text{In}(\u) \cap \text{Out}(\u)$ we must have 
\begin{equation}
\label{eq:EqualityOldpath_Newpath}
h^j(\u, V^+_{\u(2)}) = h^j(\u, V^+) = h^j(\u, V) = h^j(\u) \in \nabla(\u)
\text{ for all }j \geq 1.
\end{equation}

This motivates us to define another event $\text{In}^+(\u)$ similar to 
the event $\text{In}(\u)$ involving the point set $V^+_{\u(2)}$ only  as
\begin{align}
\label{def:EventA_in_New}
\text{In}^+(\u) := \{ h^j(\u,V^+) \in \nabla(\u) \text{ for all }j \geq 1\}. 
\end{align}
Observation (\ref{eq:EqualityOldpath_Newpath}) allows us to have the following equality of events: 
\begin{equation}
\label{eq:EqualityOldpath_NewpathEventEquality}
\text{In}(\u) \cap \text{Out}(\u) = \text{In}^+(\u) \cap \text{Out}(\u).
\end{equation}
Similarly, the occurrence of the renewal event at the $n$-th step $h^n(\x_1)$
can be equivalently represented as the occurrence of the event 
$\text{In}^+(h^n(\x_1))\cap \text{Out}(h^n(\x_1))$.
The usefulness of the representation is that the events $\text{In}^+(\u)$  and 
$ \text{Out}(\u)$ are supported on disjoint sets of random vectors and therefore, 
these two events are independent.

In order to have a better understanding of distribution of renewal increments, 
we decouple our renewal event and express it as a joint occurrence of two {\it independent} events.   
The same argument as that of Corollary \ref{cor:IncrDistEquality_1} gives us the
next Corollary. 
\begin{corollary}
\label{cor:Rwalk_IncrDistEquality_2}
For any $j \geq 1$ we have the following equality of events
\begin{align}
\label{eq:EventEquality_4}
\bigcap_{l=1}^j \{\sigma_l = m_l, h^{m_l}(\x_1) = \v_l \} =  & \Bigl [ \bigcap_{l=1}^j  \{ h^{m_l}(\x_1) = \v_l, \text{Out}(\v_l)\}\cap \bigl \{ \cap_{n = m_{l-1} + 1}^{m_l - 1} ( \gamma(h^n(\x_1)) \leq m_l - n)
 \nonumber\\
& \cup (\text{Out}(h^n(\x_1)))^c   \bigr \} \bigcap (\cap_{l=1}^{j-1}
\text{In}^{ m_{(l+1)} - m_l} )\Bigr ]\bigcap \text{In}^+(\v_j),
\end{align} 
and the event $\text{In}^+(\v_j)$ is independent of rest of the other events 
in the r.h.s. of (\ref{eq:EventEquality_4}), all of which are 
${\cal F}_{\v_j(2)} = {\cal F}_{m_j}$ measurable. 
\end{corollary}

Corollary \ref{cor:Rwalk_IncrDistEquality_2} allows us 
to obtain the following distributional equality:
\begin{align}
\label{eq:InformationUpperPlaneatRenewal}
\{\Lambda_\w : \w \in \mathbb{H}^+(\v_\ell(2))\} \mid (h^{\sigma_\ell}(\x_1) = \v_\ell)
& \stackrel{d}{=} \{\Lambda_\w : \w \in \mathbb{H}^+(\v_\ell(2))\} \mid \text{In}^+(\v_\ell) \nonumber \\ 
& \stackrel{d}{=} \{\Lambda_\w : \w \in \mathbb{H}^+(0)\} \mid \text{In}^+(\mathbf{0}).
\end{align}
Equation \ref{eq:EqualityOldpath_NewpathEventEquality}
together with Corollary \ref{cor:Rwalk_IncrDistEquality_2} give us the first equality. 
The last equality follows from the translation invariance nature of our model. 
Equation (\ref{eq:InformationUpperPlaneatRenewal}) allows us to obtain 
the next proposition which proves equality of increment distributions
between successive renewal steps for the marginal process $\{h^j(\x_1) : j \geq 1\}$. 

\begin{proposition}
\label{prop:HalfPlaneDistEquality}
Fix any $\ell \geq 1$ and we have the following equality of distributions: 
\begin{align*}
 \{ (h^{\sigma_{\ell} + m}(\x_1) - \v_\ell ) : m \geq 1 \} \mid ( 
 h^{\sigma_{\ell}}(\x_1) = \v_{\ell})
 \stackrel{d}{=} \{ h^{m}(\mathbf{0}, V^+) : m \geq 1 \} \mid (\text{In}^+(\mathbf{0})).
\end{align*}
\end{proposition}
\begin{proof}
For any $k \geq 1$, fix $n_1, \cdots, n_k \geq 1$. Given that
 $h^{\sigma_{\ell}}(\x_1) = \v_{\ell}$, we consider joint 
distribution of the random vector 
$(h^{\sigma_{\ell} + n_1}(\x_1), \cdots, h^{\sigma_{\ell} + n_k}(\x_1))$. 
Let ${\cal B}$ be an arbitrary Borel set in an appropriate space. It suffices to show that
\begin{align}
\label{eq:EqDist_1}
 & \P \Bigl ( \bigl( h^{\sigma_{\ell} + n_1}(\x_1) - \v_\ell, \cdots, 
 h^{\sigma_{\ell} + n_k}(\x_1) - \v_\ell ) \in {\cal B} 
 \mid (h^{\sigma_{\ell}}(\x_1) = \v_{\ell}) \Bigr ) \nonumber \\
&  =    \P \bigl ( \bigl(  h^{n_1}(\v_{\ell}, V^+) - \v_\ell, \cdots, 
 h^{n_k}(\v_{\ell}, V^+)- \v_\ell) \in {\cal B} 
 \mid \text{In}^+(\v_{\ell}) \bigr )\nonumber \\
 & =    \P \bigl ( \bigl(  h^{n_1}(\mathbf{0}, V^+), \cdots, h^{n_k}(\mathbf{0}, V^+)) \in {\cal B} 
 \mid \text{In}^+(\mathbf{0}) \bigr ). 
\end{align}
The last equality follows from the translation invariance nature of our model. 
For the first equality in (\ref{eq:EqDist_1}), we observe that 
\begin{align*}
& \P \Bigl ( \bigl( h^{\sigma_\ell + n_1}(\x_1) - \v_\ell, \cdots, h^{\sigma_\ell + n_k}(\x_1) 
- \v_\ell \bigr ) \in {\cal B} 
 \mid ( h^{\sigma_\ell}(\x_1) = \v_\ell) \Bigr )  \\
 & =   \P \Bigl ( \bigl( h^{n_1}(\v_\ell) - \v_\ell, \cdots, h^{n_k}(\v_\ell) - \v_\ell \bigr ) 
  \in {\cal B}  \mid ( h^{\sigma_\ell}(\x_1) = \v_\ell) \Bigr )  \\
 & =  \P \Bigl ( \bigl( h^{ n_1}(\v_\ell, V^+) - \v_\ell, \cdots, h^{n_k}(\v_\ell, V^+) - \v_\ell \bigr ) \in {\cal B} \mid (h^{\sigma_\ell}(\x_1) = \v_\ell) \Bigr )  \\
 & =   \P \Bigl ( \bigl( h^{ n_1}(\v_\ell, V^+) - \v_\ell, \cdots, h^{n_k}(\v_\ell, V^+) - \v_\ell \bigr ) \in {\cal B}  \mid \text{In}^+(\v_\ell) \Bigr ).
\end{align*} 
The last equality follows from (\ref{eq:InformationUpperPlaneatRenewal}). 
This completes the proof. 
\end{proof}

We are now ready to prove Proposition \ref{prop:SinglePtRwalk}. 

\noindent {\bf Proof of Proposition \ref{prop:SinglePtRwalk} :} Proposition \ref{prop:HalfPlaneDistEquality} gives us that the distribution of the increment random vector
$(Y_{\ell + 1} - Y_{\ell})$ for any $\ell \geq 1$ can be reconstructed 
 as follows. Consider an i.i.d. copy of $V$ as $V^{\text{ind}}$. 
 As $V^+_0$, the set $V^{\text{ind},+}_0$ is defined similarly. 
 Conditional to  the event $\text{In}^+({\bf 0})$ w.r.t. the point set 
 $V^{\text{ind}}$, start a PH path from ${\bf 0}$ 
 using  the point set $V^{\text{ind},+}_0$  only
 until occurrence of the next renewal event. 
Let $\mathbf{Z}_0$ be the position of the above path at the next renewal step. 
Then, Proposition \ref{prop:HalfPlaneDistEquality} confirms that for any $\ell \geq 1$ 
we must have 
\begin{equation}
\label{eqn:IncrementDistrn}
(Y_{\ell + 1} - Y_\ell) \stackrel{d}{=} \mathbf{Z}_0 .
\end{equation}  
This proves that the increment random vectors are identically distributed. Next we show that 
the increments are independent as well. 

Recall that the random vector $Y_\ell = h^{\sigma_\ell}(\x_1)$ 
is ${\cal S}_{\ell}$ measurable where ${\cal S}_{\ell}$ is defined as 
in $(\ref{def:S_Filtration})$. Fix $ m \geq 1 $ and Borel 
subsets $ B_2, \dotsc, B_{m+1} $ of $ \Z\times \N$. Let
$ I_{\ell+1} ( B_{\ell+1} ) $ be the indicator random variable of the event 
$ (Y_{\ell+1} - Y_{\ell})  \in B_{\ell } $. Then, we have
\begin{align*}
& \P \bigl(  (Y_{\ell + 1} - Y_{\ell }) \in B_{\ell + 1} \text{ for } \ell =  1, \dotsc, m \bigr )
= \E (  \prod_{ \ell = 1}^{m } I_{\ell+1} ( B_{\ell+1} )) \\
&  = \E \Bigl( \E  \bigl( \prod_{ \ell = 1}^{m } I_{\ell+1} ( B_{\ell+1} )
\mid {\cal S}_{m} \bigl) \Bigr) = \E \Bigl( \prod_{ \ell = 1}^{m-1 } I_{\ell+1} ( B_{\ell+1} )
\E  \bigl(  I_{m+1} ( B_{m+1} ) \mid {\cal S}_{m} \bigl) \Bigr)
\end{align*}
as the random variables $ I_{\ell+1} ( B_{\ell+1} )  $ are measurable w.r.t. $ {\cal S}_{m} $
for $ \ell =  1, \dotsc, m-1 $. 

By the earlier discussion, we have that the conditional distribution of
 $ Y_{m+1} - Y_m$ given $ {\cal S}_{m}$ is given by $\mathbf{Z}_0 $. Therefore, we have
\begin{align*}
& \P (  (Y_{\ell+1} - Y_{\ell}) \in B_{\ell+1 } \text{ for } \ell =  1, \dotsc, m )\\
& =  \E \Bigl( \prod_{ \ell = 1}^{m-1 } I_{\ell+1} ( B_{\ell+1} )
\E  \bigl(  I_{m+1} ( B_{m+1} ) \mid {\cal S}_{m} \bigl) \Bigr) \\
& = \P ( (\mathbf{Z}_0 \in B_{m+1} ) \E \Bigl( \prod_{ \ell = 1}^{m - 1 } I_{\ell+1} ( B_{\ell+1} ) \Bigr).
\end{align*}
Now, induction on $m$ completes the proof.
\qed

The next lemma proves that the distribution of the increment random variable
 $(Y_2(1)- Y_1(1))$ is symmetric about zero.   
\begin{lemma}
\label{lem:SinglePtRwalkMeanZero}
For all $m \in \N$, we have  
$$
\P(Y_2(1)- Y_1(1) = m) = \P( Y_2(1)- Y_1(1) = -m).
$$
\end{lemma}
Given $h^{\sigma_1}(\x_1) = \v_1$, proof of Lemma \ref{lem:SinglePtRwalkMeanZero} 
follows from the observation that distribution of the point set $V^+_{\v_1(2)} $ 
 remains invariant with respect to reflection about the line $x = \v_1(1)$.
Details of this argument is given in the appendix section. 

For any $\ell \geq 1$, the `In' event condition at renewal step ensures that 
$$
|Y_{\ell + 1}(1) - Y_{\ell}(1)| \leq (Y_{\ell + 1}(2) - Y_{\ell}(2))^2
\text{ with probability }1.
$$
 Hence, Proposition \ref{prop:Sigma_Tail} and Lemma 
\ref{lem:SinglePtRwalkMeanZero} readily give us Corollary \ref{cor:SingleRW_Moments}. 
\begin{corollary}
\label{cor:SingleRW_Moments}
The increment r.v. $(Y_{2}(1) - Y_{1}(1))$ has moments of all 
orders and it's mean is zero. 
\end{corollary} 


\section{Properties for the joint process at renewal steps}
\label{sec:JtRenewal}

In this section we consider renewal steps for the joint process of PH paths 
starting from two vertices $\x_1$ and $\x_2$ with $\x_1(2) = \x_2(2)=0$. 
The main objective of this section is to show that 
if vertices $h^{\sigma_{\ell}}(\x_1)$ and $h^{\sigma_{\ell}}(\x_2)$ 
are far away, then the successive increment behaves like a mean zero random walk
 on an event with high probability. 
This statement has been made 
precise in Proposition \ref{prop:ZprocessRwalkProperties}.

W.l.o.g. we assume that $\x_1 (1) < \x_2(1)$. Non-crossing property of 
PH paths ensures that we have 
$h^m(\x_1) (1) \leq h^m(\x_2)(1)$ for all $m \geq 1$. For $\ell \geq 0$
we define 
\begin{equation}
\label{eq:Z_process}
Z_\ell =  Z_\ell(\x_1, \x_2) := h^{\sigma_{\ell}}(\x_2)(1) - h^{\sigma_{\ell}}(\x_1)(1).
\end{equation}
Given ${\cal G}_\ell$, let us now focus on the case where $Z_{\ell}$ is large, and 
we will show that there exists an event $F_\ell$ which occurs with {\it high probability} such that 
on this event, the increment r.v. $Z_{\ell+1}-Z_{\ell}$ is symmetric 
about zero. Details are given below in Proposition \ref{prop:ZprocessRwalkProperties} 
which obtains some additional properties of the increment r.v.'s. 
This result will be crucially used to obtain the tail decay of the coalescing time for PH paths. 
For details see Section \ref{sec:Tail_CoalTime}.

The next result says that, far from the origin, 
the process $\{Z_\ell : \ell \geq 0\}$ behaves like a mean zero 
random walk satisfying certain moment bounds. 

\begin{prop}
\label{prop:ZprocessRwalkProperties}
Fix $\x_1,\x_2\in \Z^2$ with $\x_1(1) < \x_2(1)$
 and consider the joint process of perturbed Howard 
paths starting from these two points till the $\ell$-th (joint) renewal step. 
Given the $\sigma$-field ${\cal G}_{\ell}$, 
there exist positive constants $M_0, C_0, C_1, C_2$ and $C_3$ 
and an event $F_\ell$ such that:
\begin{itemize}
\item[(i)] On the event $Z_\ell > M_0 $ we have 
$$
\P(F^c_\ell \mid {\cal G}_{\ell} ) \leq C_3/(Z_\ell)^3\text{ and }
\E \big[ (Z_{\ell +1} - Z_{\ell}) \mathbf{1}_{F_\ell} \mid {\cal G}_{\ell} \big] = 0.
$$
\item[(ii)] On the event $\{Z_\ell \leq  M_0 \}$ we have
 $$
 \E \big[ ( Z_{\ell +1} - Z_{\ell} )  \mid {\cal G}_\ell \big] \leq C_0 .
 $$

\item[(iii)] For any $\ell\geq 0$ and $m>0$, there exists $c_m>0$ such that, 
on the event $Z_\ell \leq m$,
\begin{equation*}
\P \big( Z_{\ell+1} = 0 \mid {\cal G}_\ell \big) \geq c_m ~.
\end{equation*}
\item[(iv)] On the event $Z_\ell > M_0 $, we have
\begin{equation*}
\E \bigl[ ( Z_{\ell+1} - Z_{\ell} )^2 \mid {\cal  G}_\ell \bigr] \geq C_1 \; \text{ and } \;
\E \bigl[ |Z_{\ell+1} - Z_{\ell} |^3 \mid {\cal  G}_\ell  \bigr] \leq C_2 ~.
\end{equation*}
\end{itemize}
\end{prop}

\begin{proof}
The required event $F_\ell$ will be defined as an intersection of several other events. 
We first describe the heuristics. Consider two disjoint large rectangles 
centred around the points $h^{\sigma_\ell}(\x_1)$ and $h^{\sigma_\ell}(\x_2)$. 
Firstly we will define two events in such a way so that their intersection ensures that 
the next (joint) renewal happens inside these rectangles and till the next  renewal step,  
both the paths use perturbed open points from these two rectangles {\it only}. 
We then consider a `new' set of perturbed open points where perturbed open points 
within the two rectangular regions are interchanged and perturbed open points outside 
these two regions remain unchanged. Since both the (original) 
paths use perturbed open points from these two rectangles only till the next renewal step,
 trajectories of the concerned paths constructed  using `new' (transformed) point process   
 gets interchanged till the next renewal step (w.r.t. the original point process). 
 We need to define additional events to
ensure that the transformed step actually gives the next (joint) renewal step
w.r.t. the transformed point process.   

We need to introduce some notations. For simplicity of notations we first set  
$$
l_x = \lfloor  Z_{\ell}/40 \rfloor \text{ and } l_y = \lfloor \sqrt{Z_{\ell}/40}\rfloor .
$$
Given $h^{\sigma_\ell}(\x_i) = \w_i$ for $i = 1,2$, we define the 
following three rectangular regions both centred at $\w_i$:
\begin{align*}
{\cal R}^1_i & := \w_i \oplus [- l_x , l_x] \times [0, l_y] \text{ and } \\
{\cal R}^2_i & := \w_i \oplus [- 9 l_x , 9 l_x] \times [0, 3 l_y]\\.
{\cal R}^3_i & := \w_i \oplus [- 18 l_x , 18 l_x] \times [0, 3 l_y]. 
\end{align*}
Clearly, we have ${\cal R}^1_i \subseteq {\cal R}^2_i  \subseteq {\cal R}^3_i$ and the outer 
rectangles ${\cal R}^3_1$ and ${\cal R}^3_2$ are disjoint.   
As $(\w_1, \w_2)$ gives a (joint) renewal step, the `In' event condition ensures that
for $i = 1,2$ we have     
\begin{align*}
h^{ m}(\w_i) & = h^{ m}(\w_i, V) \in {\cal R}^1_i \text{ for all }1 \leq m \leq  
l_y \text{ as well as }\\
h^{ m}(\w_i) &  = h^{ m}(\w_i, V)  \in {\cal R}^2_i \text{ for all }1 \leq m \leq  
3 l_y.
\end{align*}
The required event $F_\ell$ will be defined as intersection of several events. 
For an illustration of the event $F_\ell$ we refer the reader to Figure \ref{fig:InEvent_6}.  
Towards that we first define the following two events :
\begin{align*}
F^1_\ell & := \{\sigma_{ \ell + 1}(\x_1, \x_2) - \sigma_{ \ell}(\x_1, \x_2) \leq l_y \} \text{ and }\\
F^2_\ell & := \{ \tilde{\w} \notin {\cal R}^2_1 \cup {\cal R}^2_2 \text{ for all }
\w \in \mathbb{H}^+(\w_1(2)) \setminus ( \cup_{i =1}^2 {\cal R}^3_i )\}.
\end{align*}
We observe that on the event $F^1_\ell\cap F^2_\ell$, at 
the next joint renewal step, i.e., the $\sigma_{\ell + 1}$-th step, 
the PH path starting from $\w_i$ stays  inside the  innermost rectangle 
${\cal R}^1_i$  for $i = 1,2$. Moreover, till 
the $ l_y$-th step, which includes the next renewal step, 
$i$-th PH path uses perturbed open vertices from the outer rectangle ${\cal R}^3_i$ only.

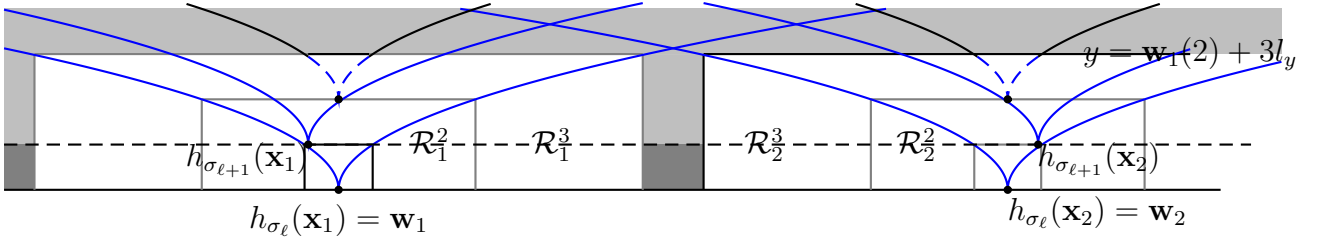
\begin{figure}
\begin{center}
\psset{unit=.8 cm}
\begin{pspicture}(-10,-2)(9,2.5)

\pscustom[linestyle=none,linecolor = lightgray, fillstyle=solid,fillcolor=lightgray]{
\psline(-10.5,0)(-10.5, 2.25)
\psline(-10.5,2.25)(-10,2.25)
\psline(-10,2.25)(-10,0)
\psline(-10,0)(-10.5,0)}

\pscustom[linestyle=none,linecolor = lightgray, fillstyle=solid,fillcolor=lightgray]{
\psline(0,0)(0, 2.25)
\psline(0,2.25)(1,2.25)
\psline(1,2.25)(1,0)
\psline(1,0)(0.5,0)}
    
\pscustom[linestyle=none,linecolor = lightgray, fillstyle=solid,fillcolor=lightgray]{
\psline(-10.5,2.25)(10.5, 2.25)
\psline(10.5,2.25)(10.5,3)
\psline(10.5,3)(-10.5,3)
\psline(-10.5,3)(-10.5,2.25)}

\pscustom[linestyle=none,linecolor = gray, fillstyle=solid,fillcolor=gray]{
\psline(-10.5,0)(-10.5, 0.75)
\psline(-10.5,0.75)(-10,0.75)
\psline(-10,0.75)(-10,0)
\psline(-10,0)(-10.5,0)}

\pscustom[linestyle=none,linecolor = gray, fillstyle=solid,fillcolor=gray]{
\psline(0,0)(0, 0.75)
\psline(1,0.75)(1,0.75)
\psline(1,0.75)(1,0)
\psline(1,0)(0,0)}

 \psline(-10.5, 0)(9.5,0)
     \psline(-5.56,.75)(-4.44, .75)
      \psline(-5.56,.75)(-5.56,0)
      \psline(-4.44,.75)(-4.44,0)

     \psline[linecolor = gray](-7.25,1.5)(-2.75, 1.5)
      \psline[linecolor = gray](-7.25,1.5)(-7.25,0)
      \psline[linecolor = gray](-2.75,1.5)(-2.75,0)

     \psline[linecolor = gray](-10,2.25)(0, 2.25)
      \psline[linecolor = gray](-10,2.25)(-10,0)
      \psline[linecolor = gray](0,2.25)(0,0)

\psline[linecolor = gray](5.45,.75)(6.55, .75)
      \psline[linecolor = gray](5.45,.75)(5.45,0)
      \psline[linecolor = gray](6.55,.75)(6.55,0)

  \psline[linecolor = gray](3.75,1.5)(8.25, 1.5)
      \psline[linecolor = gray](3.75,1.5)(3.75,0)
      \psline[linecolor = gray](8.25,1.5)(8.25,0)

\psline(1,2.25)(9, 2.25)
      \psline(1,2.25)(1,0)
    
\psplot[plotpoints=2000,linecolor=blue]{-10.5}{4}{ x 5 add abs sqrt }
\psplot[plotpoints=2000,linecolor=blue]{-10.5}{0}{ x 5.5 add abs sqrt .75 add }
\psplot[plotpoints=2000,linecolor=blue]{-3}{10.5}{ x 6 sub abs sqrt }  
\psplot[plotpoints=2000,linecolor=blue]{1}{9}{ x 6.5 sub abs sqrt .75 add }

\psplot[plotpoints=2000,linecolor=blue, linestyle = dashed]{-5.5}{-4.5}{ x 5 add abs sqrt 1.5 add}
\psplot[plotpoints=2000]{-7.5}{-5.5}{ x 5 add abs sqrt 1.5 add}
\psplot[plotpoints=2000]{-4.5}{-2.5}{ x 5 add abs sqrt 1.5 add}

\psplot[plotpoints=2000,linecolor=blue, linestyle = dashed]{5.5}{6.5}{ x 6 sub abs sqrt 1.5  add}
\psplot[plotpoints=2000]{3.5}{5.5}{ x 6 sub abs sqrt 1.5  add}
\psplot[plotpoints=2000]{6.5}{8.5}{ x 6 sub abs sqrt 1.5  add}

\pscircle[fillcolor=black,fillstyle=solid](-5, 0){.05}
\pscircle[fillcolor=black,fillstyle=solid](-5.5, .75){.05}
\pscircle[fillcolor=black,fillstyle=solid](6, 0){.05}
\pscircle[fillcolor=black,fillstyle=solid](6.5, 0.75){.05}

\pscircle[fillcolor=black,fillstyle=solid](-5, 1.5){.05}
\pscircle[fillcolor=black,fillstyle=solid](6, 1.5){.05}

\psline[linestyle = dashed](-10.5, 0.75)(10, 0.75)

\psline(-5.5, 2.25)(-4.5, 2.25)

\rput(-5,-0.5){$h_{\sigma_\ell}(\mathbf{x}_1) =\w_1$}

\rput(-3.5,0.75){${\cal R}^2_1$}
\rput(-1.5,0.75){${\cal R}^3_1$}
\rput(2,0.75){${\cal R}^3_2$}
\rput(4.5,0.75){${\cal R}^2_2$}

\rput(-6.5,0.5){$h_{\sigma_{\ell+1}}(\mathbf{x}_1) $}
\put(6,-0.5){$h_{\sigma_\ell}(\mathbf{x}_2) = \w_2$}
\rput(7.5,0.5){$h_{\sigma_{\ell+1}}(\mathbf{x}_2)$}
\rput(9,2.25){$y = \w_1(2) + 3l_y$}
\end{pspicture}
\caption{This figure represents occurrence of the $F_\ell$ event. 
Points from the shaded (light as well as deep) region are not allowed to perturb 
to the middle rectangles ${\cal R}^2_1$ and ${\cal R}^2_2$. Points from the 
 deeply shaded region are not allowed to perturb to either of the regions $\nabla(\w_1)$ 
 (which contains $\nabla(h^{\sigma_{\ell + 1}}(\x_1))$) and 
$\nabla(\w_2)$ (which contains $\nabla(h^{\sigma_{\ell + 1}}(\x_2))$). The occurrence of the events
$A^{(l_y)}_{\text{sp}}(\w_1 + (0,2l_y))$ and $ A^{(l_y)}_{\text{sp}}(\w_2 + (0,2l_y))$
create a shield using special points in the upper half-plane $\mathbb{H}^+(\w_1(2) + 3l_y)$. 
While crossing the top boundary of the outer rectangle on the line $y =\w_1(2) + 3l_y$, 
if the PH paths stay inside the respective intervals (marked in black) 
then the so-called shields prevent the PH paths to cross respective parabolic curves 
(marked in black). This ensures the occurrence of the `In' event 
for the transformed point process.}
\label{fig:InEvent_6}
\end{center}
\end{figure}

We interchange the perturbation random vectors inside these bigger rectangles 
${\cal R}^3_1$ and ${\cal R}^3_2$ to obtain a new point process. Consider 
the PH paths starting from $\w_1$ and $\w_2$ constructed using the resultant point process. 
We make it precise as follows. For $\w \in \Z^2$ we define the transformation $\overline{\w}$ as 
\begin{align*}
\overline{\w} := 
\begin{cases}
\w & \text{ if }\w \notin {\cal R}^3_1\cup {\cal R}^3_2 \\
\w_1 + (s,t)& \text{ if }\w \in {\cal R}^3_2 \text{ with } \w = \w_2 + (s,t)\\
\w_2 + (s,t) & \text{ if }\w \in {\cal R}^3_1 \text{ with } \w = \w_1 + (s,t).
\end{cases}
\end{align*}
The collection $\{ (B_{\overline{\w}}, R_{\overline{\w}}, 
\Lambda_{\overline{\w}} = ( X_{\overline{\w}}, Y_{\overline{\w}})) 
: \w \in \Z^2 \}$ gives rise to the point process 
$$
V^\prime := \{ \w  + ( X_{\overline{\w}}, Y_{\overline{\w}}): 
\w \in \Z^2, B_{\overline{\w}} = 1 \}.
$$ 
It is not difficult to observe that, on the lower half-plane $\mathbb{H}^-(\w_1(2))$, 
both $V$ and $V^\prime$ give us the same collection.
Considering evolution w.r.t. to the transformed
point set $V^\prime$ (where only the perturbed open vertices from the outer 
rectangles are interchanged),  for all $1 \leq m \leq l_y$ we have 
\begin{align}
\label{eq:Symmetry_1}
h^{ m}(\w_1, V^\prime) & = \w_1 + (h^{ m}(\w_2, V) - \w_2) \in \nabla(\w_1)\cap {\cal R}^1_1
\text{ and } \nonumber \\
h^{ m}(\w_2, V^\prime) & = \w_2 + (h^{ m}(\w_1, V) - \w_1) \in \nabla(\w_2)\cap {\cal R}^1_2.
\end{align}
Equation (\ref{eq:Symmetry_1}) further ensures that 
 on the event $F^1_\ell \cap F^2_\ell$  we have
\begin{align}
\label{eq:Symmetry_2}
\Bigl [ \sum_{i = 1}^2 \bigl ( h^{ m}(\w_i, V^\prime) - \w_i \bigr ) (1) \Bigr ]
= - \Bigl [ \sum_{i = 1}^2 \bigl ( h^{ m}(\w_i, V) - \w_i \bigr ) (1) \Bigr ] ,
\end{align}
for all  $1 \leq m \leq l_y$.
Ideally, Equation (\ref{eq:Symmetry_2}) should have given us 
$$
\E \bigl [ (Z_{\ell + 1} - Z_\ell)\mathbf{1}_{F^1_\ell \cap F^2_\ell } \mid {\cal G}_\ell
 \bigr ] =0.
$$  
Unfortunately, we need to work harder.
The issue is that starting from $\w_1$ and $\w_2$, if the  $m_0 (\leq l_y)$-th step gives 
the next, i.e., the $(\ell + 1)$-th joint renewal step 
w.r.t. the point process $V$, the corresponding step w.r.t. $V^\prime$ given by 
$(h^{m_0}(\w_1, V^\prime), h^{m_0}(\w_2, V^\prime))$ need {\it not} give 
the next  renewal step w.r.t. the transformed point process $V^\prime$. 
To ensure that the step $(h^{m_0}(\w_1, V^\prime), h^{m_0}(\w_2, V^\prime))$ is 
the next  renewal step w.r.t. $V^\prime$, we need to consider three more events. Let $\nabla_i := \nabla(\w_i)$ denote the parabolic region 
centred at $\w_i$  and let $F^3_\ell$ denote the event that
\begin{align*}
F^3_\ell := \{ & \tilde{\w} \notin \nabla_1 \cup \nabla_2 \text{ for all }\w 
\text{ with }\w(2) \in [\w_1(2) + 1, \w_1(2) + l_y ] \\
& \text{ and }
\w \notin {\cal R}^3_1 \cup {\cal R}^3_2 \}.
\end{align*}
In Figure \ref{fig:InEvent_6} the event $F^3_\ell$ means that the open 
points from the gray shaded (light or deep) region are not allowed to perturb 
to the middle rectangular boxes ${\cal R}^2_1$ and ${\cal R}^2_2$.  
The nesting property (as mentioned in Remark \ref{rem:Nesting}) together with (\ref{eq:Symmetry_1})
ensure that on the event $F^1_\ell \cap F^2_\ell$ we have (refer to Figure \ref{fig:InEvent_6})
$$
\nabla( h^{m_0}(\w_i, V^\prime) ) \subset \nabla_i \text{ for } i =1,2.
$$ 
Hence, we  observe that the `Out' event condition is 
automatically satisfied for $V^\prime$ perturbed version of all lattice points in the 
lower half-plane $\mathbb{H}^-(\w_1(2))$ as well as for lattice points in the region 
$({\cal R}^3_1\cup {\cal R}^3_2) \cap \mathbb{H}^-(\w_1(2) + m_0)$. In addition to this, 
the event $F^3_\ell$ ensures that the same holds for all lattice points in the set  
$\bigl [ \mathbb{H}^-(\w_1(2) + l_y)\setminus \mathbb{H}^-(\w_1(2)) \bigr ] 
\setminus ({\cal R}^3_1\cup {\cal R}^3_2)$ (shown as a gray region in Figure 
\ref{fig:InEvent_6}) which contains all lattice points in the set 
$[\mathbb{H}^-(m_0)\setminus \mathbb{H}^-(\w_1(2))] 
\setminus ({\cal R}^3_1\cup {\cal R}^3_2)$.  This ensures that on the event $\cap_{j=1}^3 F^j_\ell$,
the `Out' event occurs at $(h^{m_0}(\w_1, V^\prime), h^{m_0}(\w_2, V^\prime))$ w.r.t. 
the point set $V^\prime$.

To complete the proof, we need to ensure that the occurrence of the `In' event w.r.t. 
$V^\prime$ at $(h^{m_0}(\w_1, V^\prime), h^{m_0}(\w_2, V^\prime))$. 
It is important to observe that the event $\cap_{i = 1}^3 F^i_\ell$ ensures occurrence of the `In'
event at $h^{m_0}(\w_i, V^\prime)$ w.r.t. $V^\prime$ till the time  
the concerned  PH paths cross the line $y = \w_1(2) + 3 l_y$. We note 
that on the event $F^1_\ell \cap F^2_\ell $ we have 
\begin{align}
\label{eq:InEvent_1}
\nabla \bigl( \w_i + (0,2l_y)\bigr ) \subseteq \nabla \bigl( h^{m_0}(\w_i, V^\prime)\bigr) 
\text{ for all }i = 1,2.
\end{align}
We recall the definition of the event $A_{\text{sp}} = A_{\text{sp}}(\mathbf{0})$ as in 
(\ref{def:A_sp}) and for $l \in \N$ we define 
\begin{align*}
A_{\text{sp}}^{(l)}(\mathbf{0}) := \bigcap_{m=l}^\infty  \{ I^R_m\cap V^{\text{sp}}\neq \emptyset \}\cap \{ I^L_m\cap V^{\text{sp}}\neq \emptyset \}. 
\end{align*} 
For $\w \in \Z^2$ the notation $A_{\text{sp}}^{(l)}(\w)$ denotes the translated version of 
the event $A_{\text{sp}}^{(l)}(\mathbf{0})$ translated at $\w$. 
Finally, (\ref{eq:InEvent_1}) motivates us to construct our 
last two events $F^4_\ell $ and $F^5_\ell $ respectively as 
\begin{align*}
F^4_\ell & := \cap_{i=1}^2 A_{\text{sp}}^{(l_y)} \bigl( \w_i + (0,2l_y) \bigr ) \text{ and }\\
F^5_\ell & := \cap_{i=1}^2 \{ h^{3l_y}(\w_i)(1) - \w_i(1) \in [-l_x , l_x] \}.
\end{align*}
We observe that on the event $\cap_{i=1}^5 F^i_\ell$, at the $3l_y$-th step 
the $V^\prime$ driven PH paths don't deviate too much and this ensures that $h^{3l_y}(\w_i, V^\prime) \in \nabla(\w_i + (0, 2l_y))$ (marked as a black interval in Figure \ref{fig:InEvent_6}). 
The other event $F^4_\ell$ makes a shield of special points over the 
time interval $[3l_y, \infty)$. As we have 
$h^{3l_y}(\w_i, V^\prime) \in \nabla(\w_i + (0, 2l_y))$, the truncated shield
created by the event $F^4_\ell$ ensures occurrence
of a truncated version of the `In' event at
$\w_1 + (0,2l_y)$ and $\w_2 + (0,2l_y))$ over the time interval $[\w_1(2) + 3l_y, \infty)$. 
Finally, the nesting property as observed in (\ref{eq:InEvent_1}) makes sure that 
on the event $\cap_{i=1}^5 F^i_\ell$, the `In' event occurs at 
points $h^{m_0}(\w_1, V^\prime)$ and $h^{m_0}(\w_2, V^\prime)$ w.r.t. $V^\prime$ as well. 
We refer to Figure \ref{fig:InEvent_6} for an illustration of the event 
$F_\ell$ defined as
\begin{equation}
\label{eq:EventFell}
F_\ell := \cap_{i =1}^5 F^i_\ell
\end{equation}
Hence, on the event $F_\ell$  the next (joint) renewal with respect to 
$V^\prime$ occurs at $h^{m_0}(\w_1, V^\prime)$ and $h^{m_0}(\w_2, V^\prime)$.   
Lemma \ref{lem:F_ell_Decay} shows that the probability of $(F_\ell)^c$ 
satisfies the required tail bound and thereby completes the proof of item $(i)$.   

Item $(ii)$ in Proposition \ref{prop:ZprocessRwalkProperties}
follows readily from the fact that
\begin{align*}
\E[(Z_{\ell + 1} - Z_\ell) \mid {\cal G}_\ell ] & \leq 
\E[\sum_{i=1}^2|(h^{\sigma_{\ell + 1}}(\x_i) - h^{\sigma_{\ell}}(\x_i))(1)| \mid {\cal G}_\ell ] \\
& = \E[|(h^{\sigma_{\ell + 1}}(\x_1) - h^{\sigma_{\ell}}(\x_1))(1)| \mid {\cal G}_\ell ] 
+ \E[|(h^{\sigma_{\ell + 1}}(\x_2) - h^{\sigma_{\ell}}(\x_2))(1)| \mid {\cal G}_\ell ] \\
& \leq 2 \E((\sigma_{\ell + 1} - \sigma_\ell)^2 \mid {\cal G}_\ell) < \infty.
\end{align*}
Finiteness of the expectation follows from Proposition \ref{prop:Sigma_Tail} and 
the penultimate inequality follows from the `In' event condition at the renewal step.

For Item $(iii)$, it is not difficult to convince oneself that the  
conditional probability $\P(Z_{\ell + 1} = 0 \mid {\cal G}_\ell)$ is strictly positive (suitable configurations are easy to build).

It then remains to check Item $(iv)$. We observe that
\begin{align*}
\E \bigr [ (Z_{\ell + 1} - Z_\ell)^2 \mid {\cal G}_\ell \bigl]  =
\E \bigr [ (Z_{\ell + 1} - Z_\ell)^2\mathbf{1}_{ \{ |Z_{\ell + 1} - Z_\ell |^2 \geq 1 \}}
 \mid {\cal G}_\ell \bigl] \geq \P((Z_{\ell + 1} - Z_\ell )^2 \geq 1 \mid {\cal G}_\ell) ~.
\end{align*}
Again on the event $\{Z_\ell > M_0\}$, a simple construction gives us 
that the probability $\P((Z_{\ell + 1} - Z_\ell )^2 \geq 1 \mid {\cal G}_\ell)$ 
is  strictly positive. For the third moment, we have an uniform bound 
\begin{align*}
\E[(Z_{\ell + 1} - Z_\ell)^3 \mid {\cal  G}_\ell ] 
\leq 8 \E[ (\sigma_{\ell + 1} - \sigma_{\ell})^6 \mid {\cal G}_\ell ] \leq 8 C^\prime,
\end{align*}
by Proposition \ref{prop:Sigma_Tail}.
\end{proof}
The next lemma shows that as for large $Z_\ell$, the event $F_\ell$ has required tail decay.
\begin{lemma}
\label{lem:F_ell_Decay}
Consider the event $F_\ell$ defined as in (\ref{eq:EventFell}). Then for all 
large $Z_\ell$ there exists $C^\prime > 0$, which does not depend on $Z_\ell$, 
such that we have
$ \P( F^c_\ell \mid {\cal G}_\ell)  \leq C^\prime/( Z_\ell )^3 $.
\end{lemma}
\begin{proof}
It suffices to show that for all $1 \leq i \leq 5$ we have  
$$
\P( (F^i_\ell)^c \mid {\cal G}_\ell)  \leq C_0/( Z_\ell )^3,
$$ 
for some $C_0 > 0$. In this paper we proved similar arguments multiple times and we present only 
a sketch here. The required bound for $\P( (F^1_\ell)^c \mid {\cal G}_\ell)$ follows 
from Proposition \ref{prop:Sigma_Tail}. Regarding  the probability 
$\P( (F^2_\ell)^c \mid {\cal G}_\ell)$ we observe that there are $3 l_y$ many horizontal
lines of lattice points and the event $(F^1_\ell)^c $ may occur due to large amount of $x$ coordinate perturbation from one of the vertices on these horizontal lines only. 
We recall the horizontal right overshoot r.v.  and  horizontal left overshoot r.v.
as defined in Remark \ref{rem:Overshoot_ExpDecay}. Given the $\sigma$-field ${\cal G}_\ell$, 
the same argument as in Lemma \ref{lem:ConditionalMarginalTailBound} 
gives exponential tail decay for these horizontal overshoot r.v.'s. 
This allows us to obtain the required tail decay estimate
for $\P( (F^2_\ell)^c \mid {\cal G}_\ell)$.

Considering the event $F^3_\ell$, we observe that outside the boxes ${\cal R}^1_1$ and 
${\cal R}^1_2$, boundaries of the parabolic regions, $\nabla_1$
and $\nabla_2$, both are at least at a height of $2l_y$ from the line $x = l_y$.
Hence, similar argument as in Lemma \ref{lem:N_SubExptail} gives us the required estimate. 
Regarding the event $F^4_\ell$, we need to bound the probability of complement of 
truncated version of the event $A_{\text{sp}}(\cdot)$. We observe
that for any $ i = 1,2$ we have
\begin{align}
\label{eq:Bd_In_Truncated}
& \P(A_{\text{sp}}^{l_y}(\w_i + (0, 2l_y))^c \mid {\cal G}_\ell) \nonumber\\
 & =  \P(A_{\text{sp}}^{l_y}(\w_i + (0, 2l_y))^c \mid \text{In}^+(\w_1)\cap \text{In}^+(\w_2) )\nonumber\\
 & =   \P(A_{\text{sp}}^{l_y}(\w_i + (0, 2l_y))^c \cap \text{In}^+(\w_1)\cap \text{In}^+(\w_2) )/
\P(\text{In}^+(\w_1)\cap \text{In}^+(\w_2))\nonumber\\
& \leq  \P(A_{\text{sp}}^{l_y}(\w_i + (0, 2l_y))^c )/
\P(\text{In}^+(\w_1)\cap \text{In}^+(\w_2))\nonumber\\
& \leq  \P(A_{\text{sp}}^{l_y}(\w_i + (0, 2l_y))^c )/
\P \bigl( A^{\text{sp}}(\w_1)\cap A^{\text{sp}}(\w_2) \bigr )\nonumber\\
& \leq  \P(A_{\text{sp}}^{l_y}(\w_i + (0, 2l_y))^c )/
\P \bigl(  A^{\text{sp}}(\mathbf{0})\bigr)^2 \nonumber \\
& =  C_0\P(A_{\text{sp}}^{l_y}(\w_i + (0, 2l_y))^c )
\end{align} 
The first equality in (\ref{eq:Bd_In_Truncated}) follows from the fact 
that the occurrence of the event $A_{\text{sp}}^{l_y}(\w_i + (0, 2l_y))$
depends on the collection of random vectors $\{\Gamma_\w : \w \in \mathbb{H}^+(\w_1(2))\}$.
In order to bound the probability $\P(A_{\text{sp}}^{l_y}(\w_i + (0, 2l_y))^c )$,
we observe that there must be an interval, $I^R_m(\w_i + (0, 2l_y))$ or 
$I^R_m(\w_i + (0, 2l_y))$ for some $m > l_y$ which does not contain any special
vertices. Therefore, from (\ref{eq:Bd_In_Truncated}) we obtain that 
$$
\P(A_{\text{sp}}^{l_y}(\w_i + (0, 2l_y))^c \mid {\cal G}_\ell)
\leq  C^\prime_0 \exp{(- C^\prime_1 l_y)} ,
$$
for some $C^\prime_0, C^\prime_1 > 0$. This gives us the required 
tail estimate for $\P((F^4_\ell)^c)$.

Finally,  Item (iv) of Lemma \ref{lem:ConditionalMarginalTailBound} 
ensures that $\P((F^5_\ell)^c)$ decays sub-exponentially in $Z_\ell$ 
and thereby completes the proof of Lemma \ref{lem:F_ell_Decay}. 
\end{proof}

\begin{remark}
\label{rem:RenewalTimeIID_domination}
Proposition \ref{prop:Sigma_Tail} gives us that for $k= 2$, the family
$\{ \sigma_{\ell + 1}(\x_1, \x_2) - \sigma_\ell(\x_1, \x_2) : \ell \geq 1\}$ gives 
a collection of non-negative integer valued random variables with strong uniform exponential decay
(see Definition \ref{def:Exp_Tail}) such that the decay constants do not 
depend on the choice of the starting points $\x_1, \x_2$.  
\end{remark}

\section{Tail distribution for the coalescence time of two PH paths and the proof of Theorem \ref{thm:PerturbedHoward_Tree}}
\label{sec:Tail_CoalTime}

In this section we start with two points $\x_1,\x_2$  in $\Z^{2}$ such that $\x_1(1)<\x_2(1)$ and $\x_1(2)=\x_2(2)=0$. A key result for proving the convergence of the PH 
network to the BW, lies in a precise estimate for the tail distribution of the 
coalescence time of two PH paths:
 \begin{equation}
 \label{CoalTime:x1x2}
 T(\x_1,\x_2) := \inf \{ t \geq 0 : \pi^{\x_1}(t) = \pi^{\x_2}(t) \}
 \end{equation}
 where $\pi^{\x_i}=(\pi^{\x_i}(t))_{t\geq 0}$ denotes the parametrization of the path $\pi^{\x_i}$. 
  In this section we prove the following theorem on tail decay of coalescing 
  time $T(\x_1, \x_2)$ of two perturbed Howard  paths $\pi^{\x_1}$ and $\pi^{\x_2}$.

\begin{theorem}
\label{thm:CoalescingTimetail}
For the above mentioned choice of $\x_1, \x_2$, 
there exists a constant $C_{0}>0$ which does not depend on $\x_1,\x_2$ such that,
 for any $t>0$,
$$
\P(T(\x_1,\x_2) > t ) \leq  \frac{C_{0}(\x_2(1)-\x_1(1))}{\sqrt{t}}  .
$$
\end{theorem}
In order to prove Theorem \ref{thm:CoalescingTimetail} we follow a robust technique 
developed by Coupier et. al. in \cite{CSST20}. This technique is applicable for a general class
of processes which need not be Markov but behave like mean zero random 
walks away from origin and satisfy certain moment bounds (see Corollary \ref{corol:LaplaceCoalescingTimeTailNew3}). 
Proposition \ref{prop:ZprocessRwalkProperties} ensures that the difference 
between two PH paths observed at renewal steps satisfy these properties and allows us 
to apply this technique to get a suitable tail decay in terms of number of (joint)
{\it renewal} steps. With some additional work, we obtain the tail estimate 
for coalescing time in terms of total number of steps.

For completeness we first quote the following corollary taken from \cite{CSST20} regarding tail decay 
of the coalescing time for a suitable class of processes. 

\begin{corol}[Corollary 5.6 of \cite{CSST20}]
\label{corol:LaplaceCoalescingTimeTailNew3}
Let $\{Y_t : t \geq 0\}$ be a $\{{\cal G}_t : t \geq 0\}$ adapted stochastic process taking values in $\R_+$. Let $\nu^Y := \inf\{t \geq 1 : Y_t = 0\}$ be the first hitting time to $0$. Suppose for any $t \geq 0$ there exist positive constants $ M_0, C_0, C_1, C_2, C_3 $ such that:
\begin{itemize}
\item[(i)] There exists an event $F_t$ such that, on the event $\{Y_t> M_0\}$, we have $\P(F^c_t \mid {\cal G}_t) \leq C_0 / Y^3_t$ and
\begin{equation*}
\E \big[ (Y_{t+1}-Y_t) \mathbf{1}_{F_t} \mid {\cal G}_t\big] = 0 ~.
\end{equation*}
\item[(ii)] For any $t\geq 0$, on the event $\{Y_t \leq M_0\}$,
\begin{equation*}
\E \big[ (Y_{t+1}-Y_t)  \mid {\cal G}_t\big] \leq C_1 ~.
\end{equation*}
\item[(iii)] For any $t\geq 0$ and $ m > 0 $, there exists $c_m > 0 $ such that, on the event $\{Y_t \in (0, m]\}$,
\begin{equation*}
\P \big( Y_{t+1} = 0 \mid {\cal G}_t\big) \geq c_m ~.
\end{equation*}
\item[(iv)] For any $t \geq 0$, on the event $\{Y_t> M_0\}$, we have
\begin{equation*}
\E \bigl[ (Y_{t+1} - Y_t)^2 \mid {\cal  G}_t \bigr] \geq C_2 \; \text{ and } \; \E \bigl[ |Y_{t+1} - Y_t|^3 \mid {\cal  G}_t \bigr] \leq C_3 ~.
\end{equation*}
\end{itemize}
Then, $\nu^Y<\infty $ almost surely. Further, there exist positive constants $C_4, C_5$ such that for any $y>0$ and any integer $n$,
\begin{equation*}
\P( \nu^Y > n \mid Y_0 = y) \leq \frac{C_4 + C_5 y}{\sqrt{n}} ~.
\end{equation*}
\end{corol}

By Proposition \ref{prop:ZprocessRwalkProperties}, the four hypotheses $(i)$-$(iv)$ of Corollary \ref{corol:LaplaceCoalescingTimeTailNew3} are satisfied by the process $\{Z_\ell : \ell \geq 0\}$ defined in (\ref{eq:Z_process}) and this gives us the required decay estimate in terms of number of 
(joint) renewal steps before coalescing. Using this estimate, we proceed to establish Theorem \ref{thm:CoalescingTimetail}.
%
%

\begin{proof}[Proof of Theorem \ref{thm:CoalescingTimetail}]
It is easy to observe that $h^{\sigma_{\ell}}(\x_1) = h^{\sigma_{\ell}}(\x_2)$
 implies that $ h^{m}(\x_1) = h^{m}(\x_2)$ for some $m$ such that $m\leq  \sigma_{\ell}$.
In other words
$$
T_\nu := \min\{ \sigma_{\ell} : h^{\sigma_{\ell}}(\x_1) = h^{\sigma_{\ell}}(\x_2)\}
$$
dominates the actual coalescing time $T(\x_1, \x_2)$ of the two paths. 
Consider an i.i.d. sequence $\{ W_\ell : \ell \geq 1\}$ with exponentially decaying tail 
such that for each $\ell \geq 1$ the conditional distribution of 
$(\sigma_{\ell + 1}(\x_1, \x_2) -\sigma_{\ell}(\x_1, \x_2)) \mid {\cal S}_{\ell}$
is stochastically dominated by $W_{\ell+1}$. Since the family $\{(\sigma_{\ell + 1}(\x_1, \x_2) -
\sigma_{\ell}(\x_1, \x_2)) : \ell \geq 1\}$ has strong uniform exponential tail decay 
(see Remark \ref{rem:RenewalTimeIID_domination}), we can always construct 
such an  i.i.d. sequence $\{ W_\ell : \ell \geq 1\}$.
Choose $c = 1/\E(2W_1)$ and we have,
\begin{align*}
\P(T_{\nu} > t)  & \leq \P \Big( \sum_{\ell = 1}^{\lfloor ct \rfloor + 1} W_{\ell} \geq t \Big) + \P ( \nu > ct ) \\
& \leq \P \Big( \sum_{\ell = 1}^{\lfloor ct \rfloor + 1} (W_{\ell} - \E(W_1)) \geq t(1- c\E(W_1)) \Big) + \frac{C_{0}}{\sqrt{ct}}  (\x_2(1)-\x_1(1)) \\
& \leq \frac{\text{Var} \Big( \sum_{\ell = 1}^{\lfloor ct \rfloor + 1} W_{\ell} \Big)}
 { (t(1- c\E(W_1))^2} + \frac{C_{0}}{\sqrt{ct}} (\x_2(1)-\x_1(1))\\
& \leq \frac{(\lfloor ct \rfloor + 1)\text{Var}(W_1)}{(t/2 )^2} + 
\frac{C_{0}}{\sqrt{ct}} (\x_2(1)-\x_1(1)) \\
& \leq \frac{C_{1}}{\sqrt{t}} (\x_2(1)-\x_1(1)),
\end{align*}
for a suitable choice of constant $C_1 > 0$. This completes the proof.
\end{proof} 

We observe that Theorem \ref{thm:CoalescingTimetail} is applicable 
for PH paths starting from $\x_1, \x_2 \in \Z^2$ with $\x_1(2) = \x_2(2)$ and with no information
about the  collection of random vectors $\{\Gamma_\w : \w \in \mathbb{H}^-(\x_1(2))\}$.
On the other hand, while applying this result we need to estimate coalescing time of two PH paths 
after they come close enough. By that time we gather enough information about 
a large set of explored  random vectors. Therefore, we need to develop a modified version of Theorem 
\ref{thm:CoalescingTimetail} which would be applicable in such situations. 
Before ending this section, we state such a version of Theorem \ref{thm:CoalescingTimetail}
that would be more useful while proving finite dimensional convergence to coalescing Brownian motions.   
In order to do that we need to introduce some notations.

Fix any $\beta \in (0, 1/2)$ and lattice points 
 $\w_1, \w_2 \in \Z^2$ with $\w_1(2) = \w_2(2) = 0$
and $\w_1(1) < \w_2(1)$. Let ${\cal R} \subset \mathbb{H}^-(0)$ be a bounded region.
Proposition \ref{prop:CoalescingTimetail_General} gives us that given any
realisations of the collection of random vectors $\{ \Gamma_\w : \w \in {\cal R}_n\}$
satisfying that $\tilde{\w} \in \mathbb{H}^-(n^\beta)$ for all $\w \in {\cal R}$, 
the same tail estimate for coalescing time of PH paths
as in Theorem \ref{thm:CoalescingTimetail} holds. Proposition \ref{prop:CoalescingTimetail_General} 
would be more useful for proving convergence to coalescing Brownian motions. 
\begin{proposition}  
\label{prop:CoalescingTimetail_General}
For $\w_1, \w_2$ chosen as above, there exists $C_0 > 0$ (not depending on $\w_1, \w_2$)
such that for all large $n$ we have 
\begin{align*}
\P \Bigl( T(\w_1, \w_2) >  n \mid \bigl ( \{ \Gamma_\w : \w \in {\cal R}\} & \text{ such that } 
\tilde{\w} \in \mathbb{H}^-(n^\beta) \text{ for all }\w \in {\cal R} \bigr ) \Bigr )\\  
& \leq \frac{C_0(\w_2(1) - \w_1(1))}{\sqrt{n}}.
\end{align*}
\end{proposition}
\begin{proof}
Choose $\alpha \in (0, 1/2)$ such that $\alpha > \beta$ and we 
define the events $B^1_n$ and $B^2_n$ respectively as 
\begin{align*}
B^1_n & := \{ \sigma_1 = \sigma_1(\w_1, \w_2) < 2 n^\beta \} \text{ and }\\
B^2_n & := \cap_{i=1}^2 \bigl \{ \max \{|(h^{j}(\w_i) - \w_i)(1)| : 1 \leq j
 \leq \lfloor 2n^\beta \rfloor\}   \leq n^\alpha \bigr \}. 
\end{align*}
The event $B^1_n$ says that first (joint) renewal step $\sigma_1(\w_1, \w_2)$ 
occurs within the next $\lfloor 2n^\beta \rfloor$ many steps. 
Recall that we have achieved strong uniform exponential tail decay for the family $\{
\tau_{j+1}(\w_1, \w_2) - \tau_j(\w_1, \w_2) : j \geq 0\}$ through repeated occurrence of 
$A_{\text{sp}}(\quad)$ events and these events depend only on the collection
 $\{\Gamma_\w : \w \in \mathbb{H}^+(0)\}$. Further, given any realisation 
of the collection $\{ \Gamma_\w : \w \in {\cal R}\}$ satisfying 
$\tilde{\w} \in \mathbb{H}^-(n^\beta)$  for all $\w \in {\cal R}$, 
probability of the occurrence of `Out' event at a $\tau_j$ step with $\tau_j \geq n^\beta$ 
becomes higher than having no information 
on $\{ \Gamma_\w : \w \in {\cal R}\}$. Therefore, 
given any realisation of the collection $\{ \Gamma_\w : \w \in {\cal R}\}$ satisfying 
$\tilde{\w} \in \mathbb{H}^-(n^\beta)$  for all $\w \in {\cal R}$,
the argument of Proposition \ref{prop:Sigma_Tail} still ensures that 
$\P((B^1_n)^c)$ goes to zero sub-exponentially.
On the other hand, the proof of Item (iv) of Lemma \ref{lem:ConditionalMarginalTailBound} 
uses `special' points in the upper half-plane $\mathbb{H}^+(0)$. Hence, 
given any realisation of the collection $\{ \Gamma_\w : \w \in {\cal R}_n\}$, 
the probability $\P((B^2_n)^c)$ goes to zero sub-exponentially fast too. 

Let $\sigma_\ell$ be the first joint renewal step after crossing the line $y = n^\beta$. 
On the event $B^1_n \cap B^2_n$, the coalescing time 
$T(\w_1, \w_2)$ is dominated by $T(h^{\sigma_\ell}(\w_1), h^{\sigma_\ell}(\w_2)) + 2n^\beta$. 
We further observe that the coalescing time $T(h^{\sigma_\ell}(\w_1), h^{\sigma_\ell}(\w_2))$
depends on the collection of random vectors $
\{ \Gamma_\w : \w \in \mathbb{H}^+(\sigma_\ell)\}$ which is contained in 
the collection $\{ \Gamma_\w : \w \in \mathbb{H}^+(n^\beta)\}$.
Therefore, for all large $n$ we obtain 
\begin{align*}
& \P \Bigl( T(\w_1, \w_2) >  n \mid \bigl ( \{ \Gamma_\w : \w \in {\cal R}\} \text{ such that } 
\tilde{\w} \in \mathbb{H}^-(n^\beta) \text{ for all }\w \in {\cal R} \bigr ) \Bigr )\\
& \leq \P \Bigl( (T(\w_1, \w_2) >  n)\cap B^1_n \cap B^2_n 
\mid \bigl ( \{ \Gamma_\w : \w \in {\cal R}\} \text{ such that } 
\tilde{\w} \in \mathbb{H}^-(n^\beta) \text{ for all }\w \in {\cal R} \bigr ) \Bigr ) \\
& \qquad \qquad \qquad + \P((B^1_n)^c) + \P((B^2_n)^c)\\
& \leq \P \Bigl( (T(h^{\sigma_1}(\w_1), h^{\sigma_1}(\w_2)) + 2n^\beta > n)\cap B^1_n \cap B^2_n 
\mid \bigl ( \{ \Gamma_\w : \w \in {\cal R}\} \text{ such that }\\ 
& \qquad \qquad \tilde{\w} \in \mathbb{H}^-(n^\beta) \text{ for all }\w \in {\cal R} \bigr ) \Bigr ) 
 +  \P((B^1_n)^c) + \P((B^2_n)^c)\\
 & \leq \P \Bigl( (T(h^{\sigma_1}(\w_1), h^{\sigma_1}(\w_2))  > n/2)\cap B^1_n \cap B^2_n 
\mid \bigl ( \{ \Gamma_\w : \w \in {\cal R}\} \text{ such that }\\ 
& \qquad \qquad \tilde{\w} \in \mathbb{H}^-(n^\beta) \text{ for all }\w \in {\cal R} \bigr ) \Bigr ) 
 +  \P((B^1_n)^c) + \P((B^2_n)^c)\\
& = \P \Bigl( (T(h^{\sigma_1}(\w_1), h^{\sigma_1}(\w_2)) > n/2)\cap B^1_n \cap B^2_n \Bigr ) + \P((B^1_n)^c) + \P((B^2_n)^c)\\
& \leq C_0 (\w_2(1) - \w_1(1) + 2n^\alpha)/\sqrt{n} \\
& \leq C^\prime_0 (\w_2(1) - \w_1(1))/\sqrt{n}, 
\end{align*} 
for some $C^\prime_0 > 0$. 
This completes the proof.  
\end{proof}
We are ready to prove Theorem \ref{thm:PerturbedHoward_Tree} now.

\noindent {\bf Proof of Theorem \ref{thm:PerturbedHoward_Tree}:}
We prove that the PH network $G$ is a.s. connected. We prove the second part of the
Theorem \ref{thm:PerturbedHoward_Tree}, i.e., there is no bi-infinite path a.s. later 
in Remark \ref{rem:Biinfinite}.

Theorem \ref{thm:CoalescingTimetail} gives us that the coalescing time 
of two PH paths starting from $\w_1, \w_2 \in \Z^2$ with $\w_1(2) =\w_2(2)$  s finite 
a.s. Therefore, we have that
\begin{align}
\label{eq:ConnectedSameLevel}
\P \bigl [  \bigcap_{\w_1, \w_2 \in \Z^2, \w_1(2) = \w_2(2)} \{\text{
the paths }\pi^{\w_1}, \pi^{\w_2} \text{ coalesce eventually}\}\bigr ] = 1.
\end{align}  
Because of the non-crossing nature of our model, it 
is straightforward to observe that (\ref{eq:ConnectedSameLevel}) proves
the first part of Theorem \ref{thm:PerturbedHoward_Tree}.
\qed 

\section{Convergence to the Brownian web}
\label{sec:cvBW}

This section is devoted to the proof of Theorem \ref{thm:PerturbedHoward_Bweb},
 i.e,  the collection of scaled PH paths converges to the 
Brownian web (BW). In fact we prove a stronger version of the theorem in the sense that
we construct a dual process and show that under diffusive scaling the original process together 
with the dual process jointly converge to the BW and its dual.
Towards this we will apply a robust technique that was developed in \cite{CSST20}
to study convergence to the BW for non-crossing path models.
We recall here that the PH paths are non-crossing in the sense of (\ref{eq:NonCrossing}).

We recall that the BW and its dual denoted 
by $({\mathcal W},\widehat{{\mathcal W}})$ is a $({\mathcal H}\times \widehat{{\mathcal H}}, {\mathcal B}_{{\mathcal H}\times \widehat{{\mathcal H}}})$-valued random variable such that:
\begin{itemize}
\item[$(i)$] $\widehat{{\mathcal W}}$ is distributed as $-{\mathcal W}$, the BW rotated $180^0$ about the origin;
\item[$(ii)$] ${\mathcal W}$ and $\widehat{{\mathcal W}}$ uniquely determine each other
in the sense that the paths of ${\cal W}$ a.s. do not cross with (backward) paths in 
$\widehat{{\cal W}}$. See Theorem 2.4 of Schertzer et al. \cite{SSS19}.
 The interaction between paths in ${\mathcal W}$ and $\widehat{\mathcal W}$ 
 is that of Skorohod reflection (see \cite{STW00}).
\end{itemize}

It is time to specify a dual graph $\widehat{G}$ for the PH network $G$.
The construction of the dual graph is not unique and our construction is 
inspired from that of \cite{RSS16B}. For our dual graph $\widehat{G}$, the dual vertices 
are precisely the mid-points between two consecutive vertices in $V$ 
on each horizontal line $y = m$ for $m \in \Z$. It is further ensured that each 
dual vertex has a unique offspring dual vertex in the negative direction of the y-axis.
Before giving a formal definition, we direct the attention of the reader to Figure \ref{GRSDual}.

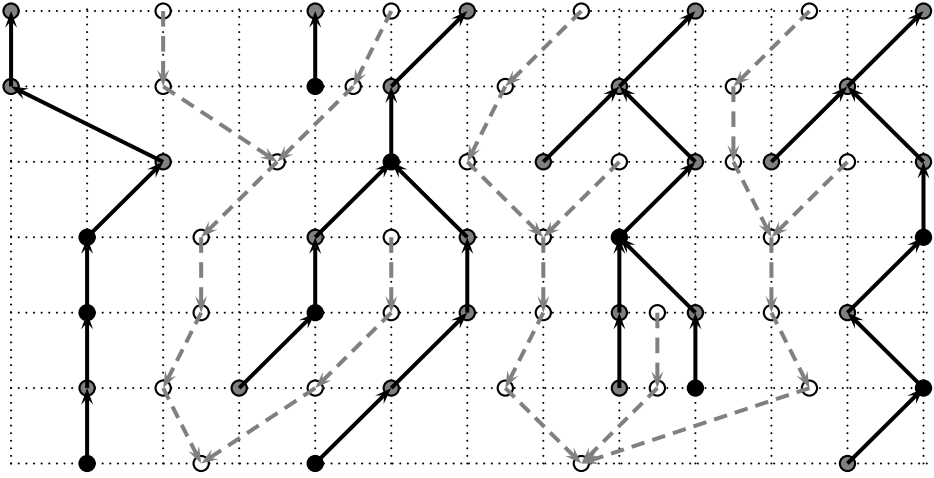
\begin{figure}[htb]
\begin{center}
\begin{pspicture}(0,0)(12,6)
\psgrid[subgriddiv=0,griddots=10,gridlabels=0]
\pscircle[fillcolor=black,fillstyle=solid](1,0){.1}
\pscircle[fillcolor=gray,fillstyle=solid](1,1){.1}
\pscircle[fillcolor=black,fillstyle=solid](1,2){.1}
\pscircle[fillcolor=black,fillstyle=solid](1,3){.1}
\pscircle[fillcolor=gray,fillstyle=solid](2,4){.1}
\pscircle[fillcolor=gray,fillstyle=solid](0,5){.1}
\pscircle[fillcolor=gray,fillstyle=solid](0,6){.1}

\psline[linewidth=1.5pt]{->}(1,0)(1,1)
\psline[linewidth=1.5pt]{->}(1,1)(1,2)
\psline[linewidth=1.5pt]{->}(1,2)(1,3)
\psline[linewidth=1.5pt]{->}(1,3)(2,4)
\psline[linewidth=1.5pt]{->}(2,4)(0,5)
\psline[linewidth=1.5pt]{->}(0,5)(0,6)

\pscircle[fillcolor=gray,fillstyle=solid](3,1){.1}
\pscircle[fillcolor=black,fillstyle=solid](4,2){.1}
\pscircle[fillcolor=gray,fillstyle=solid](4,3){.1}
\psline[linewidth=1.5pt]{->}(3,1)(4,2)
\psline[linewidth=1.5pt]{->}(4,2)(4,3)
\psline[linewidth=1.5pt]{->}(4,3)(5,4)

\pscircle[fillcolor=black,fillstyle=solid](4,5){.1}
\pscircle[fillcolor=gray,fillstyle=solid](4,6){.1}
\psline[linewidth=1.5pt]{->}(4,5)(4,6)

\pscircle[fillcolor=black,fillstyle=solid](4,0){.1}
\pscircle[fillcolor=gray,fillstyle=solid](5,1){.1}
\pscircle[fillcolor=gray,fillstyle=solid](6,2){.1}
\pscircle[fillcolor=gray,fillstyle=solid](6,3){.1}
\pscircle[fillcolor=black,fillstyle=solid](5,4){.1}
\pscircle[fillcolor=gray,fillstyle=solid](5,5){.1}
\pscircle[fillcolor=gray,fillstyle=solid](6,6){.1}

\psline[linewidth=1.5pt]{->}(4,0)(5,1)
\psline[linewidth=1.5pt]{->}(5,1)(6,2)
\psline[linewidth=1.5pt]{->}(6,2)(6,3)
\psline[linewidth=1.5pt]{->}(6,3)(5,4)
\psline[linewidth=1.5pt]{->}(5,4)(5,5)
\psline[linewidth=1.5pt]{->}(5,5)(6,6)

\pscircle[fillcolor=black,fillstyle=solid](9,1){.1}
\pscircle[fillcolor=gray,fillstyle=solid](9,2){.1}
\pscircle[fillcolor=black,fillstyle=solid](8,3){.1}
\pscircle[fillcolor=gray,fillstyle=solid](9,4){.1}
\pscircle[fillcolor=gray,fillstyle=solid](8,5){.1}
\pscircle[fillcolor=gray,fillstyle=solid](9,6){.1}

\pscircle[fillcolor=gray,fillstyle=solid](7,4){.1}
\pscircle[fillcolor=gray,fillstyle=solid](8,1){.1}
\pscircle[fillcolor=gray,fillstyle=solid](8,2){.1}
\psline[linewidth=1.5pt]{->}(8,1)(8,2)
\psline[linewidth=1.5pt]{->}(8,2)(8,3)
\psline[linewidth=1.5pt]{->}(7,4)(8,5)
\psline[linewidth=1.5pt]{->}(9,1)(9,2)
\psline[linewidth=1.5pt]{->}(9,2)(8,3)
\psline[linewidth=1.5pt]{->}(8,3)(9,4)
\psline[linewidth=1.5pt]{->}(9,4)(8,5)
\psline[linewidth=1.5pt]{->}(8,5)(9,6)

\pscircle[fillcolor=gray,fillstyle=solid](11,0){.1}
\pscircle[fillcolor=black,fillstyle=solid](12,1){.1}
\pscircle[fillcolor=gray,fillstyle=solid](11,2){.1}
\pscircle[fillcolor=black,fillstyle=solid](12,3){.1}
\pscircle[fillcolor=gray,fillstyle=solid](12,4){.1}
\pscircle[fillcolor=gray,fillstyle=solid](11,5){.1}
\pscircle[fillcolor=gray,fillstyle=solid](12,6){.1}
\pscircle[fillcolor=gray,fillstyle=solid](10,4){.1}
\psline[linewidth=1.5pt]{->}(10,4)(11,5)
\psline[linewidth=1.5pt]{->}(11,0)(12,1)
\psline[linewidth=1.5pt]{->}(12,1)(11,2)
\psline[linewidth=1.5pt]{->}(11,2)(12,3)
\psline[linewidth=1.5pt]{->}(12,3)(12,4)
\psline[linewidth=1.5pt]{->}(12,4)(11,5)
\psline[linewidth=1.5pt]{->}(11,5)(12,6)

\pscircle[fillcolor=white,fillstyle=solid](2.5,0){.1}
\pscircle[fillcolor=white,fillstyle=solid](7.5,0){.1}
\pscircle[fillcolor=white,fillstyle=solid](2,1){.1}
\pscircle[fillcolor=white,fillstyle=solid](4,1){.1}
\pscircle[fillcolor=white,fillstyle=solid](6.5,1){.1}
\pscircle[fillcolor=white,fillstyle=solid](8.5,1){.1}
\pscircle[fillcolor=white,fillstyle=solid](10.5,1){.1}

\pscircle[fillcolor=white,fillstyle=solid](2.5,2){.1}
\pscircle[fillcolor=white,fillstyle=solid](5,2){.1}
\pscircle[fillcolor=white,fillstyle=solid](7,2){.1}
\pscircle[fillcolor=white,fillstyle=solid](8.5,2){.1}
\pscircle[fillcolor=white,fillstyle=solid](10,2){.1}

\pscircle[fillcolor=white,fillstyle=solid](2.5,3){.1}
\pscircle[fillcolor=white,fillstyle=solid](5,3){.1}
\pscircle[fillcolor=white,fillstyle=solid](7,3){.1}
\pscircle[fillcolor=white,fillstyle=solid](10,3){.1}

\pscircle[fillcolor=white,fillstyle=solid](3.5,4){.1}
\pscircle[fillcolor=white,fillstyle=solid](6,4){.1}
\pscircle[fillcolor=white,fillstyle=solid](8,4){.1}
\pscircle[fillcolor=white,fillstyle=solid](9.5,4){.1}
\pscircle[fillcolor=white,fillstyle=solid](11,4){.1}

\pscircle[fillcolor=white,fillstyle=solid](2,5){.1}
\pscircle[fillcolor=white,fillstyle=solid](4.5,5){.1}
\pscircle[fillcolor=white,fillstyle=solid](6.5,5){.1}
\pscircle[fillcolor=white,fillstyle=solid](9.5,5){.1}

\pscircle[fillcolor=white,fillstyle=solid](2,6){.1}
\pscircle[fillcolor=white,fillstyle=solid](5,6){.1}
\pscircle[fillcolor=white,fillstyle=solid](7.5,6){.1}
\pscircle[fillcolor=white,fillstyle=solid](10.5,6){.1}

\psline[linecolor=gray,linewidth=1.5pt,linestyle=dashed]{->}(2,6)(2,5)
\psline[linecolor=gray,linewidth=1.5pt,linestyle=dashed]{->}(2,5)(3.5,4)
\psline[linecolor=gray,linewidth=1.5pt,linestyle=dashed]{->}(3.5,4)(2.5,3)
\psline[linecolor=gray,linewidth=1.5pt,linestyle=dashed]{->}(2.5,3)(2.5,2)
\psline[linecolor=gray,linewidth=1.5pt,linestyle=dashed]{->}(2.5,2)(2,1)
\psline[linecolor=gray,linewidth=1.5pt,linestyle=dashed]{->}(2,1)(2.5,0)

\psline[linecolor=gray,linewidth=1.5pt,linestyle=dashed]{->}(5,6)(4.5,5)
\psline[linecolor=gray,linewidth=1.5pt,linestyle=dashed]{->}(4.5,5)(3.5,4)

\psline[linecolor=gray,linewidth=1.5pt,linestyle=dashed]{->}(5,3)(5,2)
\psline[linecolor=gray,linewidth=1.5pt,linestyle=dashed]{->}(5,2)(4,1)
\psline[linecolor=gray,linewidth=1.5pt,linestyle=dashed]{->}(4,1)(2.5,0)

\psline[linecolor=gray,linewidth=1.5pt,linestyle=dashed]{->}(7.5,6)(6.5,5)
\psline[linecolor=gray,linewidth=1.5pt,linestyle=dashed]{->}(6.5,5)(6,4)
\psline[linecolor=gray,linewidth=1.5pt,linestyle=dashed]{->}(6,4)(7,3)
\psline[linecolor=gray,linewidth=1.5pt,linestyle=dashed]{->}(7,3)(7,2)
\psline[linecolor=gray,linewidth=1.5pt,linestyle=dashed]{->}(7,2)(6.5,1)
\psline[linecolor=gray,linewidth=1.5pt,linestyle=dashed]{->}(6.5,1)(7.5,0)

\psline[linecolor=gray,linewidth=1.5pt,linestyle=dashed]{->}(8,4)(7,3)
\psline[linecolor=gray,linewidth=1.5pt,linestyle=dashed]{->}(8.5,2)(8.5,1)
\psline[linecolor=gray,linewidth=1.5pt,linestyle=dashed]{->}(8.5,1)(7.5,0)

\psline[linecolor=gray,linewidth=1.5pt,linestyle=dashed]{->}(10.5,6)(9.5,5)
\psline[linecolor=gray,linewidth=1.5pt,linestyle=dashed]{->}(9.5,5)(9.5,4)
\psline[linecolor=gray,linewidth=1.5pt,linestyle=dashed]{->}(9.5,4)(10,3)
\psline[linecolor=gray,linewidth=1.5pt,linestyle=dashed]{->}(10,3)(10,2)
\psline[linecolor=gray,linewidth=1.5pt,linestyle=dashed]{->}(10,2)(10.5,1)
\psline[linecolor=gray,linewidth=1.5pt,linestyle=dashed]{->}(10.5,1)(7.5,0)
\psline[linecolor=gray,linewidth=1.5pt,linestyle=dashed]{->}(11,4)(10,3)
\end{pspicture}
\caption{The shaded (black or gray) points are points of $V$ and the black arrows represent
the PH paths. Black points are special points whereas gray points general points of $V$, i.e.,
they are perturbed versions of some other open lattice points.
 The white circled points are the points of the dual process 
and the gray (dashed) paths are the dual paths.}
\label{GRSDual}
\end{center}
\end{figure}

For $ (x,t)  \in \Z^2$, we define,
\begin{equation}
\label{DualIncrement}
\begin{split}
J^{+}_{(x,t)} & := \inf\{k: k \geq 1, ( x + k, t ) \in V\} \\
J^{-}_{(x,t)} & := \inf \{k: k\geq 1, ( x -k, t) \in V\}.
\end{split}
\end{equation}
Next, we define $r(x,t):= ( x + J^{+}_{(x,t)},t)$ and $l(x,t):= ( x - J^{+}_{(x,t)},t)$, 
as  the first open point to the right ({\em open right neighbour}) and  
the first open point to the left ({\em open left neighbour})
of $(x,t)$ at the same level $y = t$ respectively. 
For $(x,t)\in V$, let $\widehat{r}(x,t) := (x + J^{+}_{(x,t)}/2,t)$ 
and $\widehat{l}(x,t) := (x -J^{-}_{(x,t)}/2,t)$ respectively denote 
 the right dual neighbour and the left dual neighbour of $(x,t)$ in the dual vertex set.
Finally, the dual vertex set is given by
\begin{equation*}
\widehat{V} := \{\widehat{r}(x,t),\widehat{l}(x,t):(x,t) \in V \}.
\end{equation*}
For a vertex $ (u,s)\in \widehat{V} $, let $ (v,s-1) \in \widehat{V} $ be such that the straight line segment 
joining $ (u,s)$ and $(v,s-1)$ does not cross any edge in $G$. The dual edges are edges joining all 
such $ (u,s)$ and $ (v,s-1) $.  Formally, for $(u,s)\in \widehat{V}$, we define
\begin{equation}
\label{eqn:def_al_ar}
\begin{split}
a^{l} (u,s) & := \sup \{ z : (z,s-1)\in V, h(z,s-1)(1) < u \} \\
a^{r} (u,s) & := \inf\{z : (z,s-1) \in V, h(z,s-1)(1) > u \}
\end{split}
\end{equation}
and set $ \widehat{h} ( u, s) := ( (a^{l} (u,s)+a^{r} (u,s))/2, s- 1) $. 
Note that $ ( a^{r} (u,s), s-1) $ and
$ (a^{l} (u,s), s-1) $ are the nearest vertices in $V$ to the right and left respectively
of the dual vertex $ \widehat{h} ( u, s)  $.
Finally the edge set of the dual graph $\widehat{G} := (\widehat{V}, \widehat{E})$
is given by
\begin{equation*}
\widehat{E} := \{  \langle (u,s), \widehat{h}(u,s) \rangle : (u,s) \in \widehat{V}\}.
\end{equation*}
Clearly, each dual vertex has exactly one outgoing edge which goes in the downward direction. Hence, the dual graph $\widehat{G}:= (\widehat{V},\widehat{E})$ does not contain any cycle or loop. This forest $\widehat{G}$ is entirely determined from $G$ without any extra randomness. 

The dual (or backward) path $\widehat{\pi}^{(y,s)}\in\widehat{\Pi}$ starting at $(y,s)$ is
constructed  by linearly joining the successive $\widehat{h}(\cdot)$ steps. 
Thus, $\widehat{{\cal X}}:=\{\widehat{\pi}^{(y,s)} : (y,s) \in \widehat{V}\}$ denotes the collection of all dual paths obtained from $\widehat{G}$.

Let us recall that ${\cal X}_n={\cal X}_n(\gamma,\sigma)$ for $\gamma,\sigma>0$ and $n\geq 1$, is the collection of $n$-th order diffusively scaled paths. In the same way, we define $\widehat{{\cal X}}_n=\widehat{{\cal X}}_n(\gamma,\sigma)$ as the collection of diffusively scaled dual paths. 
For any dual path $\widehat{\pi}$ with starting time $
\sigma_{\widehat{\pi}}$, the scaled dual path $\widehat{\pi}_n(\gamma,\sigma) : [-\infty , \sigma_{\widehat{\pi}}/n\gamma] \to [-\infty, \infty]$ is given by
\begin{equation}
\label{defi:ScaledDualPath}
\widehat{\pi}_n(\gamma,\sigma) (t) := \widehat{\pi}(n\gamma t)/\sqrt{n }\sigma ~.
\end{equation}
For each $n\geq 1$, the closure of $\widehat{{\cal X}}_n$ in $(\widehat{\Pi},d_{\widehat{\Pi}})$
denoted as $\overline{\widehat{{\cal X}}}_n$ is a 
$(\widehat{{\cal H}}, {\cal B}_{\widehat{{\cal H}}})$-valued random variable.
We are now ready to state our result regarding joint convergence for PH model:
\begin{theorem}
\label{thm:BW_Joint}
There exist $\sigma=\sigma(p, \theta_x, \theta_y)>0$ and 
$\gamma=\gamma((p, \theta_x, \theta_y))>0$ such that the sequence
$$
\big\{ \big( \overline{{\cal X}}_n(\gamma,\sigma) , \overline{\widehat{\cal X}}_n(\gamma,\sigma) \big) : \, n \geq 1 \big\}
$$
converges in distribution to $({\cal W}, \widehat{{\cal W}})$ as $({\cal H}\times \widehat{{\cal H}},{\cal B}_{{\cal H}\times \widehat{{\cal H}}})$-valued random variables as $n\rightarrow\infty$.
\end{theorem}

The convergence criteria to the BW  for  non-crossing path models are provided by Fontes et al. \cite{FINR04}. Schertzer et al. \cite{SSS19} provides a very complete overview on this topic.
Let $\Xi\subset\Pi$. For $t>0$ and $t_0,a,b\in\R$ with $a<b$, let $\eta_\Xi(t_0,t;a,b)$ denote 
the counting random variable defined as
\begin{equation}
\label{def:EtaPtSet}
\eta_{\Xi}(t_0,t;a,b) := \# \big\{ \pi(t_0+t) :\, \pi \in \Xi , \, \sigma_{\pi}\leq t_0 \text{ and } \pi(t_0) \in [a,b] \big\}.
\end{equation}
In other words, $\eta_{\Xi}(t_0,t;a,b)$ considers all paths in $\Xi$, born before $t_0$, 
that intersect $[a,b]$ at time $t_0$ and counts the number of different positions these paths occupy at time $t_0+t$. In Theorem 2.2 of \cite{FINR04}, Fontes et al. provided the following convergence criteria.

\begin{theorem}[Theorem 2.2 of \cite{FINR04}]
\label{thm:BwebConvergenceNoncrossing1}
Let $\{ \Xi_n : n \in \N\}$ be a sequence of $({\mathcal H},B_{{\mathcal H}})$ valued random variables with non-crossing paths. Assume that the following conditions hold:
\begin{itemize}
\item[$(I_1)$] Fix a deterministic countable dense set ${\mathcal D}$  of $\R^2$. For each $\x \in {\mathcal D}$, there exists $\pi_n^{\x} \in \Xi_n$ such that for any finite set of points $\x^1, \dotsc, \x^k \in {\mathcal D}$, as $n \to \infty$, we have
$(\pi^{\x^1}_n, \dotsc, \pi^{\x^k}_n)$ converges in distribution to $(W^{\x^1}, \dotsc, W^{\x^k} )$, where $(W^{\x^1}, \dotsc, W^{\x^k} )$ denotes coalescing Brownian motions starting from the points $\x_1, \ldots, \x_k$.
\item[$(B_1)$] For all $t>0$, $\limsup_{n\rightarrow \infty}\sup_{(a,t_0)\in\R^2}\P(\eta_{\Xi_n}(t_0,t;a,a+\epsilon)\geq 2)\rightarrow 0$  as $\epsilon\downarrow 0$.
\item[$(B_2)$] For all $t>0$, $\frac{1}{\epsilon}\limsup_{n\rightarrow\infty}\sup_{(a,t_0)\in \R^2}\P(\eta_{\Xi_n}(t_0,t;a,a+\epsilon)\geq 3)\rightarrow 0$ as $\epsilon\downarrow 0$.
\end{itemize}
Then $\Xi_n$ converges in distribution to the standard Brownian web ${\mathcal W}$ as $n \to \infty$.
\end{theorem}

Let us first mention that for a sequence of $({\cal H}, {\cal B}_{{\cal H}})$-valued random variables $\{\Xi_n: n \in \N\}$ with non-crossing paths, Criterion $(I_1)$ implies 
tightness (see Proposition B.2 in the Appendix of \cite{FINR04} or Proposition 6.4 in \cite{SSS19}) and hence sub-sequential limit(s) always exists. Moreover, Criterion $(B_1)$ has been shown to be redundant with $(I_1)$ for non-crossing path models (see Theorem 6.5 of \cite{SSS19}). Actually Condition $(I_1)$
implies that subsequential limit contains coalescing Brownian motions 
starting from all points with rational coordinates 
and hence contain a copy of the standard BW ${\mathcal W}$.
Through Condition $(B_2)$, we ensure that the limiting random variable does not 
have extra paths other than the BW.

Criterion $(B_2)$ is often verified by applying an FKG type correlation inequality
 together with an estimate on   the distribution of the coalescence time between two paths. However, 
FKG is a strong property and difficult to apply for models  with complicate dependencies.
We will follow a more robust technique developed in \cite{CSST20} and applicable only for non-crossing path models. The following theorem is taken from \cite{CSST20} to 
obtain joint convergence for the PH and its dual to the BW and its dual.
\begin{theorem}[Theorem 6.3 of \cite{CSST20}]
\label{thm:JtConvBwebGenPath}
Let $\{(\Xi_n, \widehat{\Xi}_n) : n\geq 1\}$ be a sequence of $({\cal H}\times\widehat{{\cal H}}, {\cal B}_{{\cal H}\times\widehat{{\cal H}}})$-valued random variables with non-crossing paths only, satisfying the following assumptions:
\begin{itemize}
\item[(i)] For each $n \geq 1$, paths in $\Xi_n$ do not cross (backward) paths in $\widehat{\Xi}_n$ almost surely, i.e., there does not exist any $\pi\in\Xi_n$, $\widehat{\pi}\in\widehat{\Xi}_n$ and $t_1,t_2\in (\sigma_{\pi},\sigma_{\widehat{\pi}})$ such that $(\widehat{\pi}(t_1) - \pi(t_1))(\widehat{\pi}(t_2) - \pi(t_2))<0$ almost surely.
\item[(ii)] $\{\Xi_n : n \in \N\}$ satisfies $(I_1)$.
\item[(iii)]  $\{(\widehat{\pi}_n(\sigma_{\widehat{\pi}_n}),\sigma_{\widehat{\pi}_n}) : \widehat{\pi}_n \in \widehat{\Xi}_n \}$, the collection  of starting points of all the backward paths in $\widehat{\Xi}_n$, as $n \to \infty$, becomes dense in $\R^2$.
\item[(iv)] For any sub sequential limit $({\cal Z},\widehat{{\cal Z}})$ of $\{(\Xi_n,\widehat{\Xi}_n) : n \in \N\}$, paths of ${\cal Z}$ do not spend positive Lebesgue measure time together with paths of $\widehat{{\cal Z}}$, i.e.,  almost surely there is no $\pi\in{\cal Z}$ and $\widehat{\pi}\in\widehat{{\cal Z}}$ such that $\int_{\sigma_\pi}^{\sigma_{\widehat{\pi}}} \mathbf{1}_{\pi(t)=\widehat{\pi}(t)} dt > 0$.
\end{itemize}
Then $({\cal X}_n,\widehat{{\cal X}}_n)$ converges in distribution 
$({\cal W}, \widehat{{\cal W}})$ as $n \to \infty$.
\end{theorem}

It is useful to mention here that there are several other approaches to replace Criterion $(B_2)$. 
Long before, Criterion $(E)$ was proposed by Newman et al \cite{NRS05} which is applicable even for  models with crossing paths as well. Schertzer et al. \cite{SSS19} provided  a new criterion in Theorem 6.6 replacing $(B_2)$, called the \textit{wedge condition}.  Theorem \ref{thm:JtConvBwebGenPath} 
 appears as a slight generalization of Theorem 6.6 of \cite{SSS19} 
 by considering the joint convergence to the BW and it's dual. It replaces the wedge condition by the fact that no limiting primal and dual paths can spend positive Lebesgue time together.  
The next subsection is devoted to verification of the conditions of Theorem \ref{thm:JtConvBwebGenPath}
for the diffusively scaled PH and its dual  $\{(\overline{{\cal X}}_n,\overline{\widehat{{\cal X}}}_n) : n \in \N \}$. 
 
Before proceeding further, we make the following remark regarding existence of bi-infinite path for 
the PH network and this completes the proof of Theorem \ref{thm:PerturbedHoward_Tree}. 
We comment here that it is possible to obtain the same result by following a Burton-Keane argument.
\begin{remark}
\label{rem:Biinfinite}
From the construction of the dual graph it is evident that 
the PH network has a bi-infinite path if and only if the dual graph is not connected. 
If there are scaled dual paths which do not coalesce but converge 
to coalescing Brownian motions then there must be
scaled forward paths entrapped between these scaled dual paths. 
Further, joint convergence to the double Brownian web $(\WW, \widehat{\WW})$ 
forces that there must be a limiting forward Brownian path approximating this sequence of entrapped forward scaled paths and this limiting forward Brownian path must spend positive Lebesgue measure time together with a backward (dual) Brownian path. 
This leads to a contradiction and gives us that there is no bi-infinite path in the PH network 
a.s.
\end{remark}

\subsection{Verification of conditions of Theorem \ref{thm:JtConvBwebGenPath}}
\label{sebsec:DSFBwebFinalVerification}

In this section, we show that the sequence of diffusively scaled path families 
$\{(\overline{{\cal X}}_n,\overline{\widehat{{\cal X}}}_n) : n \geq 1\}$ obtained 
from the PH network and its dual satisfies the conditions in Theorem \ref{thm:JtConvBwebGenPath}.

Conditions $(i)$ and $(iii)$ of Theorem \ref{thm:JtConvBwebGenPath} hold by construction. Indeed, paths of ${\cal X}$ do not cross (backward) paths of $\widehat{{\cal X}}$ with probability $1$ and 
the same holds for $\overline{{\cal X}}_n$ and $\overline{\widehat{{\cal X}}}_n$ for any $n \geq 1$. Clearly the set of all starting points of the scaled backward paths
 in $\widehat{\Xi}_n$ becomes dense in $\R^2$ as $n \to \infty$.
We prove Condition (ii) in the following subsection.

\subsubsection{Verification of Condition $(I_1)$}
\label{subsubsec:I_1Verification}
 
In this section we show that the condition $(ii)$ 
holds for the sequence $\{\overline{{\cal X}}_n : n\geq 1\}$, 
i.e., Criterion $(I_1)$ of Theorem \ref{thm:BwebConvergenceNoncrossing1} holds.
 
We first focus on a single path, $\pi^{\mathbf{0}}$ starting at the origin.
The main ingredient here is the construction of i.i.d. pieces through (marginal) renewal steps.    
As shown in Proposition \ref{prop:SinglePtRwalk}, the sequence of renewal steps
 $\{ h^{\sigma_\ell}(\mathbf{0})(1) : \ell \geq 1\}$ breaks down the path
 $\pi^{\mathbf{0}}$ into independent pieces. Let us scale the 
 PH path $\pi^{\mathbf{0}}$  starting from $\mathbf{0}$
 into $\pi^{\mathbf{0}}_n$ as in (\ref{eq:PathScale}) with the following choices of $\sigma$
 and $\gamma$
\begin{align}
\label{def:Sigma_Gamma}
\sigma^2 & := \text{Var}(Y_2(1) - Y_1(1)) =  \text{Var} \bigl( (h^{\sigma_2}(\mathbf{0}) - 
h^{\sigma_1}(\mathbf{0}))(1)\bigr) \text{ and } \nonumber \\ 
\gamma & := \E(Y_2(2) - Y_1(2)) = \E \bigl ( h^{\sigma_2}(\mathbf{0})(2) -
 h^{\sigma_1}(\mathbf{0})(2) \bigr).
\end{align}
From now on, the diffusively scaled sequence $\{\overline{{\cal X}}_n : n\geq 1\}$ is considered w.r.t. these parameters, but for ease of writing, we drop $(\gamma,\sigma)$ from our notation. Proposition \ref{prop:SinglePtRwalk} together with Corollary \ref{cor:SingleRW_Moments} 
 allow us to apply Donsker's invariance principle 
 to show that $\pi^{\mathbf{0}}_n$ converges in distribution in $(\Pi,d_{\Pi})$ to $B^{\mathbf{0}}$, a standard Brownian motion starting at $\mathbf{0}$.

The above argument proves Criterion $(I_1)$ for $k=1$. To prove Criterion $(I_1)$ 
for general $k \geq 1$ we follow the method of induction. We proceed to prove it for $k \geq 2$
assuming it is true for $k -1$. The strategy that we adopt is to show that until the time 
when the $k$-th PH path comes close to one of the other $(k-1)$ PH paths, 
it can be approximated by an independent path with the same distribution as itself, 
and after that time, it quickly coalesces with the path which is close to it
and both of them converge to the same Brownian motion. Proposition 
\ref{prop:CoalescingTimetail_General}, which has been derived from 
Theorem \ref{thm:CoalescingTimetail}, allows us to show that 
when two PH paths come close enough, after that time they coalesce quickly. 
This strategy was first developed by Ferrari et al in \cite{FFW05} 
to deal with dependent paths with bounded range interactions, and later 
modified in a series of papers (\cite{CFD09}, \cite{RSS16A}, \cite{CSST20})
 to deal with long range interactions. Since, the essential idea of the proof here 
 is the same as that of \cite{CSST20}, we do not provide the full details and 
 refer the reader to Section 6.2.1 of \cite{CSST20}.
Still, because of complicate dependencies of our model, 
we need to make some adjustments.

Thus we obtain that, for $\v \in \R^2$ and for any sequence 
$\v^n \in \Z^2$ such that $(\v^n(1)/\sqrt{n}\sigma, 
\v^n(2)/n\gamma) \to \v$ as $n \to \infty$, the scaled PH path $\pi^{\v^n}_n$
converges in distribution to $B^\v$, Brownian motion starting from $\v$. 
In fact, non-crossing property of paths imply that for sequences $\v^n $ and $\w_n$ in $\Z^2$ 
such that $\v^n(2) = \w^n(2) = 0$, $\w^n(1) < 0 < \v^n(1)$ with 
$(\v^n(1) - \w^n(1))/\sqrt{n} \to 0$, 
the couple $(\pi^{\w^n}_n , \pi^{\v^n}_n )$ converges in distribution to
$(B^{\mathbf{0}},B^{\mathbf{0}})$, i.e., to the same Brownian motion starting from origin.
This is also implied by the estimated
on the coalescing time that we have established in Theorem \ref{thm:CoalescingTimetail}.
This completes the proof of Criterion $(I_1)$ for $k=1$.

We now define a sequence of subsets of $ \Pi^k$ where the $k$ th path {\em comes
close to one of the $k-1$ paths}. We fix $ \alpha \in (0,1/2)$ for the rest of this section.
For $n \geq 1$, define
\begin{align}
 \label{Analpha}
A_n^\alpha = & \Bigl\{  (\pi_1,\dotsc,\pi_k) \in \Pi^k :  \pi_i \text{'s  satisfy the following conditions } \nonumber \\
& a)  \quad \pi_i(\sigma_{\pi_i}) \neq \pi_j(\sigma_{\pi_j}) \text{ for all }
1 \leq i < j \leq k;  \nonumber \\
& b) \quad t^{k}_n : = \inf\{t\geq 0: |\pi_i(t)-\pi_k(t)|\leq 3n^{\alpha - 1/2} \text{  for some }
1 \leq i\leq k-1\} < \infty \Bigr\}.
\end{align}
Next we define the `{\em $\alpha$-coalescence map}'
$f_n^{(\alpha)}:  \Pi^k \to \Pi^k$, which is actually a modification of the coalescence map
$f_n$ introduced in Ferrari {\it et al.} \cite{FFW05}, as follows:
\begin{equation*}
f_n^{(\alpha)} (\pi_1,\dotsc,\pi_{k}) :=
\begin{cases}
(\pi_1,\dotsc,\pi_{k-1},\overline{\pi}_{k}) & \text{for }(\pi_1,\dotsc,\pi_{k}) \in A_n^\alpha \\
(\pi_1,\dotsc,\pi_{k}) & \text{ otherwise }
\end{cases}
\end{equation*}
with
\begin{align*}
\overline{\pi}_{k}(t) :=
\begin{cases}
 \pi_{k}(t) & \text{ for } t \leq t^{k}_n\\
 \pi_{k}(t^{k}_n) +
n ( t - t^{k}_n) \bigl[ \pi_{i}( t^{k}_n
+ \frac{1}{n} ) -  \pi_{k}( t^{k}_n) \bigr] &  \text{ for } t^{k}_n < t \leq t^{k}_n
+ \frac{1}{n} \\
 \overline{\pi}_{i}(t) & \text{ for } t >  t^{k}_n + \frac{1}{n}
\end{cases}
\end{align*}
where $i$ is the index such that  $|\pi_{i}( t^{k}_n) - \pi_{k}( t^{k}_n)| 
\leq 3 n^{\alpha - 1/2}$ and
$|\pi_{j}( t^{k}_n) - \pi_{k}( t^{k}_n)| > 3 n^{\alpha - 1/2}$ for all $ 1 \leq j < i$.

Fix any set of $k$ points $\y_1, \cdots , \y_k$ in $\R^2$ and let 
$(W^{\y_1}, \dotsc, W^{\y_k})$ denote coalescing Brownian motions starting from 
these points. Let $\{(\y^n_1, \cdots , \y^n_k)
: n \in \N\}$ be such that for all $1 \leq i \leq k$ we have
$$
(\y^n_i(1)/\sqrt{n}\sigma, \y^n_i(1)/n\gamma) \to \y_i \text{ as } n \to \infty.
$$  
Let $\pi^i$ denote the PH path starting from $\y^n_i$ and $\pi^i_n$
represents the $n$-th diffusively scaled version.
The following proposition completes the verification of condition $ (I_1)$.
\begin{prop}
\label{prop:fddbweb2}
We have, as $ n \to \infty $,
 \begin{align*}
(a)~ & f^{(\alpha)}_n (\pi^{1}_n, \dotsc, \pi^{k-1}_n, \pi^{k}_n)
\Rightarrow  (W^{\y_1}, \dotsc, W^{\y_k}); \\
(b)~ &  (\pi^{1}_n, \dotsc, \pi^{k-1}_n, \pi^{k}_n)
\Rightarrow  (W^{\y_1}, \dotsc, W^{\y_k}).
 \end{align*}
\end{prop}
%
\begin{proof} 
For $\pi \in \Pi$ and $t > \sigma_\pi$ let $\pi_{[\sigma_\pi, t]} = \pi_{(-\infty, t]}$
denotes the restriction of $\pi$ over the time interval $[\sigma_\pi, t]$.
For $(\pi_1, \cdots , \pi_k) \in \Pi^k$ and for $t > 
\max\{\sigma_{\pi_i} : 1\leq i \leq k\}$ the notation $(\pi_1, \cdots , \pi_k)_{(-\infty, t]}$
denotes the tuple of respective restrictions of individual paths. Fix $t > \max\{\y_i(2) : 1\leq i \leq k\}$ and we observe that for Item (a) it is enough to show that 
 \begin{align}
 \label{eq:K_tuple_Weak_Conv}
 f^{(\alpha)}_n ((\pi^{1}_n, \dotsc, \pi^{k-1}_n, \pi^{k}_n)_{(-\infty, t]})
\Rightarrow  (W^{\y_1}, \dotsc, W^{\y_k})_{(-\infty, t]}.
 \end{align}
  Further, in order to prove (\ref{eq:K_tuple_Weak_Conv})
  it suffices to show that 
  \begin{align}
 \label{eq:K_tuple_Weak_Conv_1}
 \E \bigl [ g \bigl ( f^{(\alpha)}_n ((\pi^{1}_n, \dotsc, \pi^{k-1}_n, \pi^{k}_n)_{(-\infty, t]})
  \bigr ) \bigr ]
\Rightarrow  \E \bigl [ g \bigl ( (W^{\y_1}, \dotsc, W^{\y_k})_{(-\infty, t]} \bigr ) \bigr ],
 \end{align}
 for bounded continuous function $g$.
 We need to introduce some notations. 
 Our motivation is to show that with high probability a PH path uses perturbed 
 versions of `close' enough lattice points only and therefore, as long as two PH paths 
 are far away, they evolve independently. 
 
 Fix $ \beta \in (0,\alpha)$ where $\alpha$ is as in (\ref{eq:}). We consider 
 `Tubes' of width $n^\alpha$ and $n^\beta$ around each of the $k$ (unscaled) PH paths 
 as  defined below: 
 \begin{align*}
 \mathbb{T}^\alpha_{n,i} & := \{ (y,s) \in \Z^2 : ||(y,s) - h^m(\y^n_i)||_1 \leq n^\alpha
 \text{ for some } 0 \leq m  \leq \lfloor n\gamma t - \y^n_i(2) \rfloor\}\text{ and }\\
 \mathbb{T}^\beta_{n,i} & := \{ (y,s) \in \Z^2 : ||(y,s) - h^m(\y^n_i)||_1 \leq n^\beta
 \text{ for some } 0 \leq m  \leq \lfloor n\gamma t - \y^n_i(2) \rfloor\} 
\end{align*}  
  where $t$ is as in (\ref{eq:K_tuple_Weak_Conv}).
  Let the event $E_n$ is defined as bellow
  \begin{align*}
 E_n := \bigcap_{i=1}^k & 
 \bigl\{ \text{There does not exist any } \w \text{ in }\Z^2 \setminus 
 \mathbb{T}^\alpha_{n,i} \text{ with }
 \tilde{\w} \in \mathbb{T}^\beta_{n,i} \bigr \}. 
\end{align*}
Exponential tail decay of perturbation random vectors gives us 
that $\P(E_n)$ converges to $1$ as $n \to \infty$.
Therefore, we can modify (\ref{eq:K_tuple_Weak_Conv}) as 
\begin{align}
\label{eq:K_tuple_Weak_Conv_1}
\E \bigl [ g \bigl ( f^{(\alpha)}_n ((\pi^{1}_n, \dotsc, \pi^{k-1}_n, \pi^{k}_n)_{(-\infty, t]})
 \bigr )\mathbf{1}_{E_n} \bigr ] 
 \to \E \bigl [ g \bigl ( (W^{\y_1}, \dotsc, W^{\y_k})_{(-\infty, t]} \bigr ) \bigr ]
 \text{ as }n \to \infty.
\end{align}
For each $n \geq 1$ let $\overline{\pi}^k_n$ denote a scaled PH path 
which has the same marginal distribution as that of 
$\pi^k_n$ and evolves independently of $(\pi^1_n, \cdots, \pi^{k-1}_n)$. A simple 
construction using another i.i.d. collection of random vectors $\{\Gamma^{\text{ind}}_w : 
\w \in \Z^2 \}$ independent of the collection $\{\Gamma_w : \w \in \Z^2 \}$ ensures that 
such a random path $\overline{\pi}^k_n$ exists. 

We observe that on the event $E_n$, till time $t \wedge t^{k}_n$, 
the path $\pi^k_n$ is supported on disjoint sets 
of random vectors and hence, it's evolution is independent of all of the $k-1$ paths
$\pi^1_n, \cdots, \pi^{k-1}_n$. Therefore, we have  
$$
\E \bigl [ g \bigl ( f^{(\alpha)}_n ((\pi^{1}_n, \dotsc, \pi^{k-1}_n, \pi^{k}_n)_{(-\infty, t]})
 \bigr )\mathbf{1}_{E_n} \bigr ]  
 \stackrel{d}{=}
 \E \bigl [ g \bigl ( f^{(\alpha)}_n ((\pi^{1}_n, \dotsc, \pi^{k-1}_n, 
 \overline{\pi}^{k}_n)_{(-\infty, t]}) \bigr )\mathbf{1}_{E_n} \bigr ].
$$  
The same argument as in Item (a) of Proposition 5.6 in \cite{RSS16A} gives us 
$$
\E \bigl [ g \bigl ( f^{(\alpha)}_n ((\pi^{1}_n, \dotsc, \pi^{k-1}_n, 
 \overline{\pi}^{k-1}_n)_{(-\infty, t]}) \bigr ) \bigr ] \to 
 \E \bigl [ g \bigl ( (W^{\y_1}, \dotsc, W^{\y_k})_{(-\infty, t]} \bigr ) \bigr ]
 \text{ as }n \to \infty.
$$ 
As $\P(E_n)$ converges to $1$ as $n \to \infty$, 
Equation (\ref{eq:K_tuple_Weak_Conv_1}) follows from the
observation that 
$$
\lim_{n \to \infty}
\E \bigl [ g \bigl ( f^{(\alpha)}_n (\pi^{1}_n, \dotsc, \pi^{k-1}_n, 
 \overline{\pi}^{k}_n)_{(-\infty, t]} ) \bigr ) \mathbf{1}_{E_n} \bigr ] =  
 \E \bigl [ g \bigl ( f^{(\alpha)}_n ((\pi^{1}_n, \dotsc, \pi^{k-1}_n, 
 \overline{\pi}^{k}_n)_{(-\infty, t]}) \bigr ) \mathbf{1}_{E_n} \bigr ].
$$
This proves Item (a).   

For (b) we first consider the situation $t^{k}_n \geq t$. 
On the event $t^{k}_n \geq t$, Item (b) follows trivially as we have 
$$
(\pi^{1}_n, \dotsc, \pi^{k-1}_n,  \pi^{k}_n)_{(-\infty, t]} =
f^{(\alpha)}_n (\pi^{1}_n, \dotsc, \pi^{k-1}_n, 
 \pi^{k}_n)_{(-\infty, t]} \text{ a.s.}
$$
Next, we consider the situation $ t > t^{k}_n$. W.l.o.g. we assume 
that at time $t^{k}_n$, paths $\pi^i_n$ and $\pi^k_n$
come close enough. We need to show that after time $t^{k}_n$, these two paths 
quickly coalesce. Proposition \ref{prop:CoalescingTimetail_General} 
would help us to achieve this. 
Towards that we define the event
  \begin{align*}
 F_n := \bigcap_{i=1}^k & 
 \bigl\{ \tilde{\w} \in \mathbb{H}^-(n\gamma t^{k}_n + n^\beta) \text{ for all }
 \w \in \mathbb{T}^\beta_{n,i} \cap \mathbb{H}^-(n\gamma t^{k}_n) \bigr \}. 
\end{align*}
We observe that $\P(F_n)$ converges to $1$ as $n \to \infty$. Finally,
on the event $F_n$, Proposition \ref{prop:CoalescingTimetail_General} gives us the 
required tail estimate for coalescing time of $\pi^k_n$ and $\pi^i_n$ and completes the proof.       
\end{proof}
 
\subsubsection{Verification of condition (iv)}
\label{subsubsec:LastConditionVerification}

To show condition $(iv)$, we mainly follow Section 6.2.2 of \cite{CSST20} 
and  again,  the coalescence time estimate given in Proposition \ref{prop:CoalescingTimetail_General} 
serves as a key ingredient. Let $({\cal Z}, \widehat{\cal Z})$ be any sub-sequential 
limit of $\{(\overline{{\cal X}}_n, \widehat{\overline{{\cal X}}}_n): n \geq 1\}$. 
By Skorohod's representation theorem we may assume that 
we are working on a probability space such that the convergence happens almost surely. 
With slight abuse of notation we continue to denote that subsequence also  
by $\{(\overline{{\cal X}}_n, \widehat{\overline{{\cal X}}}_n): n\geq 1\}$.

We have to prove that, with probability $1$, paths in ${\cal Z}$ do not spend positive Lebesgue measure time together with the dual paths in $\widehat{\cal Z}$. This means that for any $\delta>0$ and any integer $m\geq 1$, the probability of the event
\begin{equation*}
A(\delta, m) := \left\lbrace
\begin{array}{c}
\text{$\exists$ paths $\pi\in {\cal Z}, \widehat{\pi}\in \widehat{\cal Z}$ and $t_0\in\R$ s.t. $-m<\sigma_{\pi}<t_0<t_0 + \delta<\sigma_{\widehat{\pi}}<m$} \\
\text{and $-m<\pi(t) = \widehat{\pi}(t)<m$ for all $t\in [t_0, t_0+\delta]$}
\end{array}
\right\rbrace
\end{equation*}
has to be $0$.

To show that $\P(A(\delta,m))=0$, we introduce a generic event $B^{\epsilon}_n(\delta, m)$ defined as follows. Given an integer $m\geq 1$ and $\delta,\epsilon>0$,
\begin{align*}
B^{\epsilon}_n(\delta, m) := \bigl \{ & 
\exists \text{ paths }\pi_1^n,\pi_2^n,\pi_3^n  \in {\cal X}_n \text{ s.t. }\sigma_{\pi_1^n},\sigma_{\pi_2^n} = 0, \sigma_{\pi_3^n} \leq \delta \text{ and }
\pi_1^n(0),\pi_1^n(\delta)\in [-m,m] \\
& \text{with }|\pi_1^n(0)-\pi_2^n(0)|<\epsilon \text{ but }\pi_1^n(\delta) \not= \pi_2^n(\delta) \\
& \text{and with }|\pi_1^n(\delta)-\pi_3^n(\delta)|<\epsilon \text{ but }
\pi_1^n(2\delta) \not= \pi_3^n(2\delta) \bigr \} 
\end{align*}
The event $B^{\epsilon}_n(\delta, m)$ means that there exists a path $\pi_1^n$ localized in $[-m,m]$ at time $0$ as well as at time $\delta$ which is approached (within distance $\epsilon$) by two path $\pi_2^n$ and $\pi_3^n$ respectively at times $0$ and $\delta$ while still being different from them respectively at time $\delta$ and $2\delta$. 

It was shown in Section 6.2.2 of \cite{CSST20} that to show $\P(A(\delta, m) = 0)$ it suffices 
to prove the following lemma.
\begin{lemma}
\label{lem:coalescenceDSF}
For any integer $m\geq 1$, real numbers $\epsilon,\delta>0$, there exists a constant $C_0(\delta,m)>0$ (only depending on $\delta$ and $m$) such that for all large $n$,
\begin{equation*}
\P( B^{\epsilon}_n (\delta, m) ) \leq C_0(\delta, m) \, \epsilon ~.
\end{equation*}
\end{lemma}

For the proof of Lemma \ref{lem:coalescenceDSF} we essentially follow \cite{CSST20}. Note 
that the discrete nature of the perturbed point process $V$ requires some modifications. 
We need to deal with non-Markovian nature of the PH paths.

\begin{proof}[Proof of Lemma \ref{lem:coalescenceDSF}]
We define the event $ D^{\epsilon}_n $ as the unscaled version 
of the event $ B^{\epsilon}_n $ in the following way: 
\begin{align*}
D^{\epsilon}_n  & := \bigl\{ \text{there exist }x, y, z \in \Z 
\text{ such that }  x \in [ - m \sqrt{n} \sigma,  m \sqrt{n} \sigma ], 
| x - y | < \sqrt{n} \epsilon \sigma \text{ and } \\ 
& \pi^{(x,0)}(\lfloor n \gamma \delta \rfloor) 
 \neq \pi^{(y,0)}(\lfloor n \gamma \delta \rfloor),
  | \pi^{(x,0)}(\lfloor n \gamma \delta \rfloor) - z | < \sqrt{n} \epsilon \sigma,
 \pi^{(x,0)}(2\lfloor n\gamma \delta \rfloor) \neq 
\pi^{(z,0)}(2\lfloor n \gamma \delta \rfloor)\bigr\}.
\end{align*}
For $ \omega \in D^{\epsilon}_n$, suppose $ x , y $ are as in the definition above and assume that 
$ x < y $. Set $ l  = \max \{ x + j : \pi^{(x,0)}(\lfloor n\gamma \delta \rfloor)  = 
\pi^{(x+j,0)}(\lfloor n \gamma \delta \rfloor) \}$.  
Clearly, $ -  m \sqrt{n} \sigma  \leq  x \leq l < y \leq  ( m+\epsilon) \sqrt{n} \sigma $
 and $ \pi^{(x,0)}(\lfloor n \gamma \delta \rfloor)
 = \pi^{(l,0)}(\lfloor n \gamma \delta \rfloor)
  < \pi^{(l+1,0)}(\lfloor n \gamma \delta \rfloor)
  \leq \pi^{(y, 0)}(\lfloor n \gamma \delta \rfloor) $.
    Assume that $ \pi^{(x, -\delta)}(\lfloor n \gamma \delta \rfloor) = k$ for some $k \in \Z $.
    Then, $z$ in the definition above satisfies 
    $ z \in (k - \sqrt{n} \epsilon \sigma, k + \sqrt{n} \epsilon \sigma)$ and 
    $ \pi^{( k,\lfloor n \gamma \delta \rfloor)}(2\lfloor n \gamma \delta \rfloor) \neq 
\pi^{(z, \lfloor n \gamma \delta \rfloor)}(2\lfloor n \gamma \delta \rfloor)$. So, 
by non-crossing property of paths,  it must be the case that 
$$
\pi^{( k -  \lfloor \sqrt{n} \epsilon \sigma  \rfloor- 1,\lfloor n \gamma \delta \rfloor)}(2\lfloor n\gamma \delta \rfloor) \neq \pi^{(k +  \lfloor \sqrt{n} \epsilon \sigma  \rfloor+ 1, 
\lfloor n \gamma \delta \rfloor)}(2\lfloor n \gamma \delta \rfloor)
$$
 Thus, we must have $ \omega \in H^{(L)} (n, \delta, \epsilon) $ where  for $ l \in \Z $, 
\begin{align*}
  H_{l, k}^{(L)} (n, \epsilon) := & \bigl\{ \pi^{(l, 0)}(\lfloor n \gamma \delta \rfloor) = 
  k \neq \pi^{(l+1,0)}(\lfloor n \gamma \delta \rfloor) \text{ and } \\
& \pi^{( k -  \lfloor \sqrt{n} \epsilon \sigma  \rfloor- 1,\lfloor n \gamma \delta \rfloor)}(2\lfloor n \gamma \delta \rfloor) \neq \pi^{(k +  \lfloor \sqrt{n} \epsilon \sigma  \rfloor+ 1,
 \lfloor n \gamma \delta \rfloor)}(2\lfloor n \gamma \delta \rfloor)  \bigr\}; \\
H^{(L)} (n, \delta, \epsilon) := &  \cup_{ l =  - \lfloor  2 m \sqrt{n} \sigma \rfloor  }
^{ \lfloor  2 m \sqrt{n} \sigma \rfloor} \cup_{k \in \Z } H_{l, k}^{(L)} (n, \delta, \epsilon)  .  
\end{align*}

Similarly for $ \omega \in D^{\epsilon}_n$ such that  $ x > y $, set $ r  = \min \{ x - j : \pi^{(x, -\delta)}(\lfloor n \gamma \delta \rfloor)   =  \pi^{(x-j, -\delta)}(\lfloor n \gamma \delta \rfloor)
   \} $. 
As earlier, $ \omega \in H^{(R)} (n, \delta, \epsilon) $ where 
for $ r \in \Z $, 
\begin{align*}
H_{r, k}^{(R)} (n, \delta, \epsilon) := & \bigl\{ \pi^{(r, 0)}(\lfloor n \gamma \delta \rfloor)
= k \neq \pi^{(r-1, 0)}(\lfloor n \gamma \delta \rfloor) \text{ and } \\
& \pi^{( k -  \lfloor \sqrt{n} \epsilon \sigma  \rfloor- 1,
\lfloor n \gamma \delta \rfloor )}(2\lfloor n \gamma \delta \rfloor) \neq 
\pi^{(k +  \lfloor \sqrt{n} \epsilon \sigma  \rfloor+ 1, \lfloor n \gamma \delta \rfloor )}
(2\lfloor n \gamma \delta \rfloor)   \bigr\}; \\
H^{(R)} (n, \delta, \epsilon) := & \cup_{ r =  - \lfloor 2 m \sqrt{n} \sigma \rfloor  }
^{ \lfloor  2 m \sqrt{n} \sigma \rfloor}  \cup_{k \in \Z } H_{r, k}^{(R)} (n, \delta, \epsilon) .  
\end{align*}

Thus, $ D^{\epsilon}_n \subseteq H^{(L)} (n, \delta, \epsilon) \cup H^{(R)} (n, \delta, \epsilon) $.

The argument is very similar to Proposition \ref{prop:fddbweb2} and we only give 
a sketch here. We consider tubes of widths $n^\alpha$ and $n^\beta$ around the 
PH path $\pi^{(l,0)}$ and $\pi^{(l+1,0)}$ over the time interval $[0, \lfloor n\gamma\delta \rfloor]$. 
We observe that the probability of the event 
that these two paths do not use random vectors associated to lattice points 
outside these `$n^\alpha$' tubes to evolve and 
perturbed versions of all the lattice points inside these two tubes 
are confined to the lower half-plane $\mathbb{H}^-(\lfloor n\gamma\delta \rfloor + n^\beta)$ 
converges to $1$ as $n \to \infty$. Therefore for all large $n$, 
on the said event we can apply Proposition \ref{prop:CoalescingTimetail_General} to obtain
\begin{align*}
 & \P \{ \pi^{( k -  \lfloor \sqrt{n} \epsilon \sigma  \rfloor - 1,
 \lfloor n \gamma \delta \rfloor )}(2\lfloor n \gamma \delta \rfloor) \neq 
\pi^{(k +  \lfloor \sqrt{n} \epsilon \sigma \rfloor + 1, 
\lfloor n \gamma \delta \rfloor )}(2\lfloor n \gamma \delta \rfloor) 
\mid \pi^{(l, 0)}(\lfloor n \gamma \delta \rfloor) = k 
\neq \pi^{(l+1,0)}(\lfloor n \gamma \delta \rfloor) \} \\
& \leq  \frac{C_2 (2\lfloor \sqrt{n} \sigma \epsilon \rfloor + 3) }
{\sqrt{ \lfloor n \gamma \delta \rfloor}}
  \leq C_3 (\delta) \epsilon 
\end{align*}
where $ C_2, C_3 (\delta)  > 0 $ are constants. Hence, 
\begin{align*}
 \P ( H_{k}^{(L),1} (n, \delta, \epsilon)) \leq C_3 (\delta) \epsilon  ~  
 \P  \{  \pi^{(l,0)}(\lfloor n \gamma \delta \rfloor) = k \neq \pi^{(l+1, 0)}
 (\lfloor n \gamma \delta \rfloor)  \} . 
\end{align*}

Now, the events $ \{ \pi^{(l,0)}(\lfloor n \gamma \delta \rfloor) = k 
\neq \pi^{(l+1, 0)}(\lfloor n \gamma \delta \rfloor)  \} $ 
are disjoint for distinct values of $ k $. Hence, 
\begin{align*}
\P   \bigl( \cup_{k \in \Z } H_{ k}^{(L),1} (n, \delta, \epsilon) 
\bigr) & \leq \sum_{ k \in \Z} \P  \bigl( H_{k}^{(L),1} (n, \delta, \epsilon) \bigr) \\
& \leq C_3 (\delta) \epsilon  \sum_{ k \in \Z}  \P  
\{  \pi^{(l,0)}(\lfloor n \gamma \delta \rfloor) = k \neq \pi^{(l+1,0)}
(\lfloor n \gamma \delta \rfloor) \}  \\
& = C_3 (\delta) \epsilon ~  \P  \{  \pi^{(l,0)}(\lfloor n \gamma \delta \rfloor) 
\neq \pi^{(l+1,0)}(\lfloor n \gamma \delta \rfloor) \}.
\end{align*}
The above argument also holds for $  \cup_{k \in \Z } H_{r, k}^{(R)} (n, \delta, \epsilon)$. 
Thus, combining the above terms and applying Proposition \ref{prop:CoalescingTimetail_General}
\begin{align*}
\P ( D^{\epsilon}_n \cap F_n(k) ) & \leq  \P   \bigl( H^{(L)} (n, \delta, \epsilon)  \bigr)  +  \P   \bigl( H^{(R)} (n, \delta, \epsilon)  \bigr) \\
 & \leq  \sum_{ l =  - \lfloor 2 m \sqrt{n} \sigma \rfloor  }^{ \lfloor  2 m \sqrt{n} \sigma \rfloor} 
\P   \bigl( \cup_{k \in \Z } H_{ k}^{(L)} (n, \delta, \epsilon) \bigr) +  
\sum_{ r =  - \lfloor 2 m \sqrt{n} \sigma_0 \rfloor  }^{ \lfloor  2 m \sqrt{n} \sigma \rfloor} 
\P   \bigl( \cup_{k \in \Z } H_{r, k}^{(R)} (n, \delta, \epsilon) \bigr)   \\
& \leq 16 m \sqrt{n} \sigma  C_3 ( \delta ) \epsilon C_2 / \sqrt{\lfloor n \gamma \delta \rfloor }
 \leq C_1 ( \delta, m ) \epsilon 
\end{align*}
for a proper choice of $ C_1 ( \delta, m) $. 
This completes the proof.
\end{proof} 

\section{Appendix}

In this section we present proofs of some basic results that we have used before.

\noindent {\bf Proof of Lemma \ref{lem:SinglePtRwalkMeanZero}:} 
Fix $j \geq 1$. Given $h^{\sigma_j}(\x_1) = \v_j$,
we recall that the path $\{ h^n(\v_j) : n \geq 1\}$ uses 
the random vectors $\{\Lambda_\w : \w \in \mathbb{H}^+(\v_j(2))\}$ only and always
 stays within the region $\nabla(\v_j)$. 
 
The main idea is to show that distribution of the point set $V^+_{\v_j(2)} $ 
 remains invariant with respect to reflection
about the line $x = \v_j(1)$. For $\w = \v_j + (s,t) \in \mathbb{H}^+(\v_j(2))$ we define
its reflected copy reflected about the line $X = \v_j(1)$ as 
$$
\overline{\w} := \v_j + (-s, t).
$$ 

Now, using the collection $\{\Gamma_\w = (B_\w, R_\w, \Lambda_\w): \w \in \Z^2 \}$, 
we define a new collection $\{ \Gamma^\prime_\w := B_{\w}^\prime, R_{\w}^\prime,
 \Lambda^\prime_\w : \w \in \Z^2\}$ given by 
\begin{align*}
\bigl (  B_{\w}^\prime, R_{\w}^\prime,
 \Lambda^\prime_\v := (X^\prime_{\w}, Y^\prime_{\w}) \bigr ):= 
\begin{cases}
\bigl( B_{\overline{\w}}, - R_{\overline{\w}},  
 (- X_{\overline{\w}}, Y_{\overline{\w}})\bigr ) 
& \text{ if }\w \in \mathbb{H}^+(\v_j(2))\\
\bigl( B_{\w},  R_{\w},  
 (X_{\w}, Y_{\w})\bigr ) 
& \text{ if }\w \in \mathbb{H}^-(\v_j(2)).
\end{cases}
\end{align*}
The resultant point process generated from this collection 
is defined as 
$$
V^\prime := \{ \w + \Lambda^\prime_\w : B_\w^\prime = 1\}.
$$ 
We observe that the newly created point process $V^\prime$ on $\mathbb{H}^+(\v_j(2))$
gives a reflected copy of the set $V^+_{\v_j(2)}$
reflected about the line $x = \v_j(1)$. 
More precisely, a point $\w$ belongs to $ (V^\prime)^+_{\v_j(2)}$ if and only 
if $\overline{\w} \in V^+_{\v_j(2)}$. 

It is also important to observe that for $\w \in \mathbb{H}^+(\v_j(2))$ 
we have $R_{\w}^\prime = - R_{\overline{\w}}$. This ensures that in case of `tie'
(w.r.t. nearest member from $V^\prime$ at the next level) the choice of outgoing edge gets reversed 
appropriately. This way the newly constructed collection 
$\{\Gamma^\prime_\w : \w \in \Z^2\}$ ensures that 
\begin{align}
\label{eq:SymmetryRWIncr}
\overline{h^n}(\v_j, V) = \overline{h^n}(\v_j, V^+_{\v_j(2)}) = h^n(\v_j, (V^\prime)^+_{\v_j(2)})
\text{ for all }n \geq 1.
\end{align}
The first equality follows as $\v_j$ is a renewal step.
 As the region $\nabla(\v_j)$ is symmetric about the line $X = \v_j(1)$, 
 (\ref{eq:SymmetryRWIncr}) ensures that 
the new point process satisfies the renewal conditions as well: 
\begin{itemize}
\item[(i)] $h^n(\v_j, V^\prime) = h^n(\v_j, (V^\prime)^+_{\v_j(2)}) \in \nabla(\v_j)$ for all $n\geq 1$.
\item[(ii)]The event $\text{Out}(\v_j)$ occurs w.r.t. the point set $V^\prime$ also as 
$(V^\prime)^+_{\v_j(2)} = V^-_{\v_j(2)}$.
\end{itemize}
Hence, given $h^{\sigma_j}(\x_1) = \v_j$, 
distribution of the process $\{h^n(\v_j) : n \geq 1\}$ starting from $\v_j$ remains the same 
when it progresses using the collection $\{ \Lambda^\prime_\v : \v \in \mathbb{H}^+(\v_j(2))\}$. 
 
Let $\v_{j+1} = \v_j + (s_0, t_0)$ be the point of the next renewal w.r.t. the point process $V$.
 It suffices to show that the point $\overline{\v}_{j+1} $
 is the position of the next renewal with respect to the point set $V^\prime$ as well.
 Now, (\ref{eq:SymmetryRWIncr}) ensures that we have $
h^{t_0}(\v_j, V^\prime) = \overline{\v}_{j+1}$. 
From (\ref{eq:SymmetryRWIncr}) we also obtain that for any 
$ 1  \leq n_1 \leq n_2 $ 
\begin{align}
\label{eq:SymmetryRWIncr_1}
\bigl( h^{n_2}(\v_j, V) - h^{n_1}(\v_j, V) \bigr )(1) = 
- \bigl( h^{n_2}(\v_j, V^\prime ) - h^{n_1}(\v_j, V^\prime ) \bigr )(1).
\end{align}
By taking $n_1 = t_0$ and $n_2 \geq t_0 + 1$ in Equation (\ref{eq:SymmetryRWIncr_1})
we obtain 
$$
h^{n}(\overline{\v}_{j+1}, V^\prime) = h^{n}(\overline{\v}_{j+1}, (V^\prime)^+_{\v_{j+1}(2)}) \in \nabla( \overline{\v}_{j+1} )
 \text{ for all 
} n \geq 1.
$$
This implies occurrence of the event $\text{In}^+(\overline{\v}_{j+1})$ w.r.t. the point process 
$V^\prime$.  Given that the event $\text{Out}(\v_{j})$ has occurred, 
occurrence of the event $\text{Out}(\overline{\v}_{j+1})$
 w.r.t. the point set $V^\prime$ depends only on the random vectors 
 $\{\Lambda^\prime_\w : \w(2) \in [\v_j(2) + 1, \v_{j+1}(2)]\}$. 
We show that for $ \w \in \Z^2$ with $\w(2) \in [\v_j(2) + 1, \v_{j+1}(2)]$,
 we have 
 $$
 \w + (X_\w, Y_\w) \in \nabla(\v_{j+1}) \text{
if and only if }\overline{\w} + (-X_{\overline{\w}}, Y_{\overline{\w}}) 
\in \nabla(\overline{\v}_{j+1}).
$$
This follows from the fact that
\begin{align*}
(\w - \v_{j+1})(1) + X_\w & = -(\overline{\w} - \overline{\v}_{j+1})(1) - X_{\overline{\w}} 
\text{ and }\\
(\w - \v_{j+1})(2) + Y_\w & = (\overline{\w} - \overline{\v}_{j+1})(2) + Y_{\overline{\w}}  .
\end{align*}
Hence, occurrence of the event  
$\text{Out}(\v_{j+1})$ w.r.t. the point process $V$ implies and implied by 
occurrence of the event $\text{Out}(\overline{\v}_{j+1})$ 
w.r.t. the point process $V^\prime$.

Essentially, we proved that for any $n \geq 1$, the $n$-th step $h^n(\v_j) = \x$
 is a renewal step (w.r.t. the point process $V$) if and only if
 the corresponding step w.r.t. $V^\prime$ is given by 
$h^n(\v_j, V^\prime) = \overline{\x}$ 
is a renewal step as well. This proves that starting from $\v_j$ the step  
$$ 
h^{t_0}(\v_j, V^\prime ) = \overline{\v}_{j+1}
$$ 
gives the next renewal step w.r.t. $V^\prime$.
 This completes the proof. 
\qed


\begin{thebibliography}{AAAAa}
 \bibitem[A79]{A79}
 R.~Arratia. (1979).
 Coalescing Brownian motions on the line.
 Ph.D. Thesis, University of Wisconsin, Madison.

\bibitem[A81]{A81}
 R.~Arratia.
 Coalescing Brownian motions and the voter model on $\Z$.
 Unpublished partial manuscript. Available from rarratia@math.usc.edu.

\bibitem[A03]{A03}
S.~Asmussen. 
Applied Probability and Queues. 
Springer, New York, 2003. 


\bibitem[CFD09]{CFD09}
 C.F.~Coletti, L.R.G.~Fontes, and E.S.~Dias.
 Scaling limit for a drainage network model.
 {\em J.\ Appl.\ Probab.}~46, 1184--1197, 2009.


\bibitem[CT13]{CT13}
 D.~Coupier, V.C.~Tran.
 The 2d-directed spanning forest is almost surely a tree.
 {\em Random Structures and Algorithms}~42(1):59-72, 2013. 
 

\bibitem[CSST20]{CSST20}
 D.~Coupier, K.~Saha, A.~Sarkar and V.C.~Tran.
 The 2-d directed spanning forest converges to the Brownian web.
 {\em Ann.  Probab.} 49(1), 435--484, 2020. 


 
\bibitem[FFW05]{FFW05}
 P.A.~Ferrari, L.R.G.~Fontes, and X.-Y.~Wu.
 Two-dimensional Poisson Trees converge to the Brownian web.
 {\em Ann.\ Inst.\ H.\ Poincar\'e Probab.\ Statist.} 41, 851--858, 2005.


\bibitem[FINR04]{FINR04}
 L.R.G.~Fontes, M.~Isopi, C.M.~Newman, and K.~Ravishankar.
 The Brownian web: characterization and convergence.
 {\em Ann.\ Probab.}~32(4), 2857--2883, 2004.

\bibitem[FLT04]{FLT04}
 P.A.~Ferrari, C.~Landim, and H.~Thorisson.
 Poisson trees, succession lines and coalescing random walks.
 {\em Ann.\ Inst.\ H.\ Poincar\'e Probab.\ Statist.} 40, 141--152, 2004.

\bibitem[H71]{H71}
 \textsc{Howard, A., D.} \newblock{(1971)}.
 \newblock{Simulation of stream networks by headward growth and branching}.
 \newblock \emph{ Geogr.\ Anal.} \textbf{3} 29--50.

\bibitem[LL62]{LL62}
L.B. Leopold and W. B. Langbein.
 The concept of entropy in landscape evolution. 
{\em Report, U.S. Geol. Surv. Prof. Pap.} 500-A.

\bibitem[LPW08]{LPW08}
D.~Levin, Y.~Peres, and E.~Wilmer.
\newblock {\em Markov chains and mixing times}.
\newblock A. M. S., Providence, Rhode Island, 2009.

\bibitem[NRS05]{NRS05}
 C.M.~Newman, K.~Ravishankar, and R.~Sun.
 Convergence of coalescing nonsimple random walks to the Brownian web,
 {\em Electron.\ J.\ Prob.} 10, 21--60, 2005.
 
 \bibitem[GL17]{GL17} S. Ghosh and J. L.  Lebowitz. 
Fluctuations, large deviations and rigidity in hyperuniform systems: a brief
survey. {\em Indian J. Pure Appl. Math.} 48(4): 609 -- 631, 2017.
 
  
\bibitem[GS20]{GS20} S. Ghosh and K. Saha. 
Transmission and Navigation on Disordered Lattice Networks, Directed Spanning Forests and Brownian Web.
{\em Journ. Statist. Phys.} 52(3): 1106 - 1143, 2020.


\bibitem[HS13]{HS13}  A. E. Holroyd and T. Soo. 
Insertion and Deletion Tolerance of Point Processes, 
Directed Spanning Forests and Brownian Web.
{\em Electron. Journ. Probab.} 18(74): 1 - 24, 2013.

\bibitem[PS14]{PS14} 
 Y. Peres and  A. Sly. Rigidity and tolerance for perturbed lattices, 
 arXiv 1409.4490, 2014.


 \bibitem[RR97]{RR97} 
I. Rodriguez-Iturbe and A. Rinaldo. 
 Fractal river basins: chance and self-organization, 
  Cambridge Univ. Press, New York, 1997.

\bibitem[RSS16A]{RSS16A} R. Roy, K. Saha and A. Sarkar. 
Random directed forest and the Brownian web.
{\em Ann. Inst. H. Poincare Probab. Statist.} 52(3): 1106 - 1143, 2016.

\bibitem[RSS16B]{RSS16B} R. Roy, K. Saha and A. Sarkar. 
Hack’s law in a drainage network model: a Brownian web approach.
{\em Ann. Appl. Probab.}  26(3): 1807 - 1836, 2016.


\bibitem[S67]{S67} A. E. Scheidegger. 
 A stochastic model for drainage pattern into an intramontane trench.
{\em Bull. Ass. Sci. Hydrol.}  12 : 15–20, 1967.
  

\bibitem[SSS19]{SSS19}
E.~Schertzer, R.~Sun, and J.~Swart.
The {B}rownian web, the {B}rownian net, and their universality,
{\em Advances in Disordered Systems, Random Processes and Some
  Applications},  Cambridge University Press, 270--368, 2019.

 \bibitem[STW00]{STW00}
F.~Soucaliuc, B.~T\'{o}th, and W.~Werner, 
Reflection and coalescence between one-dimensional {B}rownian paths,
{\em Ann. Inst. Henri Poincar\'e. Probab. Statist.}, 36: 509--536,
  2000.
  
  
 \bibitem[T02]{T02}
S.~Torquato, Random Heterogeneous Materials: Microstructure and Macroscopic Properties,
Springer, New York, 2002.  
  
 \bibitem[TS03]{TS03}
S.~Torquato and F. Stillinger, Local density fluctuations, hyperuniformity, and order metrics.
{\em Phys. Rev. E}, {\bf 68} 041113, 2003.    
  
  
\end{thebibliography}
\end{document}